\documentclass[12pt]{amsart}
\usepackage[left=3cm, right=3cm, top=3.8cm]{geometry}

\usepackage{texproject/format-palatino}

\usepackage{texproject/macro-arxiv-typesetting}
\usepackage{texproject/macro-general}
\usepackage{texproject/macro-refs_ams}
\usepackage{texproject/macro-theorem_ams}
\usepackage{texproject/macro-analysis}
\usepackage{texproject/macro-tikz}
\usepackage{project-macros}
\newcommand{\titlename}{A Multifractal Decomposition for Self-similar Measures with Exact Overlaps}
\newcommand{\shorttitlename}{Multifractal Decomposition}
\newcommand{\docclasses}{28A78, 28A80}
\newcommand{\dockeywords}{self-similar measures, multifractal formalism, weak separation condition}

\subjclass[2020]{\docclasses}

\begin{document}
\title[\shorttitlename]{\titlename}
\author{Alex Rutar}
\address{Mathematical Institute, North Haugh, St Andrews, Fife KY16 9SS, Scotland}
\email{ar339@st-andrews.ac.uk}
\thanks{The author was supported by EPSRC Grant EP/V520123/1.}
\thanks{This paper is in final form and no version of it will be submitted for publication elsewhere.}

\begin{abstract}
    We study self-similar measures in $\R$ satisfying the weak separation condition along with weak technical assumptions which are satisfied in all known examples.
    For such a measure $\mu$, we show that there is a finite set of concave functions $\{\tau_1,\ldots,\tau_m\}$ such that the $L^q$-spectrum of $\mu$ is given by $\min\{\tau_1,\ldots,\tau_m\}$ and the multifractal spectrum of $\mu$ is given by $\max\{\tau_1^*,\ldots,\tau_m^*\}$, where $\tau_i^*$ denotes the concave conjugate of $\tau_i$.
    In particular, the measure $\mu$ satisfies the multifractal formalism if and only if its multifractal spectrum is a concave function.
    This implies that $\mu$ satisfies the multifractal formalism at values corresponding to points of differentiability of the $L^q$-spectrum.
    We also verify existence of the limit for the $L^q$-spectra of such measures for every $q\in\R$.
    As a direct application, we obtain many new results and simple proofs of well-known results in the multifractal analysis of self-similar measures satisfying the weak separation condition.
\end{abstract}

\maketitle
\tableofcontents
\section{Introduction}
Given a finite Borel measure $\mu$, a standard way to quantify the density of $\mu$ at a given point $x$ in its support is through the \defn{local dimension}, which is the quantity
\begin{equation*}
    \dim_{\loc}(\mu,x) = \lim_{r\to 0}\frac{\log B(x,r)}{\log r}
\end{equation*}
when the limit exists.
A natural question to ask is the following: what is the structure of the set of points which have a prescribed local dimension $\alpha$?
In many interesting cases, these level sets of local dimensions are uncountable and dense in $\supp \mu$, but have $\mu$-measure zero for most values of $\alpha$.
We will focus on the Hausdorff dimensions of these level sets of local dimensions, which we denote by
\begin{equation*}
    f_\mu(\alpha)=\dim_H\{x\in\supp\mu:\dim_{\loc}(\mu,x)=\alpha\}.
\end{equation*}
The function $f_\mu$ is commonly known as the \defn{(fine Hausdorff) multifractal spectrum} of $\mu$.

Related to the multifractal spectrum is the \defn{$L^q$-spectrum} of the measure $\mu$, which is given by
\begin{equation*}
    \tau_\mu(q)=\liminf_{r\to 0}\frac{\log\sup\sum_i\mu(B(x_i,r))^q}{\log r}
\end{equation*}
where the supremum is taken over all disjoint families of balls $\{B(x_i,r)\}_i$ with $x_i\in\supp \mu$.
A standard application of Hölder's inequality shows that $\tau_\mu$ is a concave function of $q$.
The $L^q$-spectrum is related to the multifractal spectrum through a heuristic relationship known as the \defn{multifractal formalism}.
It states that if the measure $\mu$ is ``sufficiently nice'', then the multifractal spectrum is the concave function given by
\begin{equation*}
    f_\mu(\alpha)=\tau_\mu^*(\alpha)=\inf\{\alpha q-\tau_\mu(q):q\in\R\}.
\end{equation*}

One can think of the $L^q$-spectrum as a sort of box-counting dimension, whereas the multifractal spectrum is a generalization of the Hausdorff dimension.
Of course, the multifractal formalism does not hold in general: for example, in the presence of non-conformality, $f_\mu$ and $\tau_\mu$ can both be concave conjugate functions but $f_\mu(\alpha)<\tau_\mu^*(\alpha)$ for all $\alpha$ \cite{jr2011}.
In some sense, this is a consequence of the fact that the box and Hausdorff dimensions of non-conformal sets are, in general, not the same.
However, even when a measure is ``locally nice'', the multifractal formalism can fail: if $\mu_1$ and $\mu_2$ are probability measures with disjoint supports each satisfying the multifractal formalism and $\nu=(\mu_1+\mu_2)/2$, a straightforward exercise from the definitions shows that
\begin{equation}\label{e:min-formula}
    \begin{aligned}
        \tau_\nu(q)&=\min\{\tau_{\mu_1}(q),\tau_{\mu_2}(q)\}\\
        f_\nu(q) &= \max\{f_{\mu_1}(q),f_{\mu_2}(q)\}.
    \end{aligned}
\end{equation}
In particular, $\nu$ satisfies the multifractal formalism if and only if $\tau_{\mu_1}(q)\leq\tau_{\mu_2}(q)$ or $\tau_{\mu_2}(q)\leq\tau_{\mu_1}(q)$.
Our main result states, for a certain class of conformal measures, that this phenomenon is the only way in which the multifractal formalism can fail.

We will focus on the multifractal analysis of self-similar measures in $\R$, which are defined as follows.
Given a finite set of maps $(S_i)_{i\in\mathcal{I}}$ where each $S_i:\R\to\R$ is given by $S_i(x)=r_i x+d_i$ where $0<|r_i|<1$ and probabilities $(p_i)_{i\in\mathcal{I}}$ with $p_i>0$ and $\sum p_i=1$, the self-similar measure $\mu$ is uniquely defined by
\begin{equation*}
    \mu=\sum_{i\in\mathcal{I}}p_i\cdot S_i\mu
\end{equation*}
where $S_i\mu$ is the pushforward of $\mu$ by $S_i$.

Self-similar measures are ``locally nice'' by nature of their construction (indeed, they have equal box and Hausdorff dimensions \cite{fal1997}), so one might be more optimistic for nice multifractal properties.
For example, self-similar measures are exact-dimensional \cite{fh2009}, which means that there is precisely one value $\alpha$ for which the level set $\{x\in\supp\mu:\dim_{\loc}(\mu,x)=\alpha\}$ has full $\mu$-measure.
If there is an open set $U$ satisfying $\bigcup_{i\in\mathcal{I}}S_i(U)\subseteq U$ where the union is disjoint, we say that $\mu$ satisfies the \defn{open set condition} \cite{hut1981}.
For such measures, the $L^q$-spectrum is the unique smooth function satisfying $\sum_{i\in\mathcal{I}}p_i^qr_i^{-\tau_\mu(q)}=1$, and the multifractal formalism holds \cite{cm1992,pat1997}.

However, for self-similar measures with overlaps, the multifractal formalism can fail.
One of the earliest known examples of this fact is due to Hu and Lau \cite{hl2001}, where they show that the three-fold convolution of the Cantor measure has an isolated point in its set of local dimensions, and therefore fails the multifractal formalism.
This measure, and generalizations, have been studied in \cite{flw2005,hhn2018,lw2005,shm2005} among other papers.
Another class of well-studied measures are the Bernoulli convolutions, which is the law of the random variable $\sum_{n=0}^\infty\pm\lambda^n$ for $\lambda\in(0,1)$ where the $+$ and $-$ signs are chosen with equal probabilities.
In this case, for any parameter $\lambda\in(1/\phi,1)$ where $\phi$ is the Golden mean, the set of local dimensions has an isolated point \cite[Prop. 2.2]{hh2019}, and $\phi$ is maximal with this property.
Testud \cite{tes2006a} constructed self-similar measures associated with digit-like sets for which the multifractal spectrum is non-concave and the maximum of two non-trivial concave functions.
Thus behaviour similar to \cref{e:min-formula} can occur for self-similar measures with overlaps.

On the other hand, for Bernoulli convolutions with contraction ratio the reciprocal of a simple Pisot number (the unique positive root of a polynomial $x^k-x^{k-1}-\cdots-x-1$ for some $k\geq 2$), the multifractal formalism is known to hold \cite{fen2005}.
It is also shown in \cite{rut2021} that any self-similar measure associated with the IFS $\{\lambda_1 x,\lambda_2 x+\lambda_1(1-\lambda_2),\lambda_2 x+(1-\lambda_2)\}$ for $\lambda_1,\lambda_2>0$ and $\lambda_1+2\lambda_2-\lambda_1\lambda_2\leq 1$ satisfies the multifractal formalism.
The $L^q$-spectra of self-similar measures also have a certain amount of regularity: the limit defining $\tau_\mu(q)$ is known to exist for any $q\geq 0$ \cite{ps2000}.

We see that, even for self-similar measures, a wide variety of behaviour is possible.
Determining precisely when the multifractal formalism is satisfied, and more generally understanding properties of the $L^q$-spectrum and multifractal spectrum when it is not, is a very challenging question and little is known.

In this paper, we develop a general theory in an attempt to remedy this.
We will show for an important class of self-similar measures that the varied multifractal behaviour observed above follows from a decomposition similar in form to \cref{e:min-formula}.
More precisely, we show that the $L^q$-spectrum of $\mu$ is given by the minimum of a finite set of concave functions, and the multifractal spectrum of $\mu$ is given by the maximum of their concave conjugates.
These concave functions can be loosely interpreted as the $L^q$-spectra of a decomposition of $\mu$ as a sum of subadditive set functions, each satisfying a multifractal formalism.
By standard arguments involving concave functions, this shows that the multifractal formalism holds for $\mu$ in the following generic sense: $\mu$ satisfies the multifractal formalism if and only if $f_\mu$ is a concave function.
This is in stark contrast to measures associated with iterated function systems of non-conformal maps as discussed above.

\subsection{The weak separation condition and finite type conditions}
Many of the examples mentioned in the preceding section satisfy various closely-related finite type conditions \cite{fen2003,hhrtoappear,ln2007,nw2001}.
Heuristically, these finite type conditions require that there are only ``finitely many overlaps''
These separation conditions are all special cases of the \defn{weak separation condition} of Lau and Ngai \cite{ln1999}, which states that there is a uniform bound on the number of simultaneous ``distinct overlaps'' (see \cref{e:wsc} for a precise statement).
Note that the weak separation condition is strictly weaker than the open set condition.
When the invariant set $\supp\mu$ is a closed interval, the generalized finite type condition coincides with the weak separation condition \cite{fen2016,hhrtoappear}.
It is an open question to determine, outside certain degenerate situations, if these two separation conditions are equivalent in general.

The multifractal analysis of such measures have been extensively studied since.
Such measures have enough structure to allow strong results, yet the class contains many interesting examples and exceptional behaviour.
The most significant general result to date, due to Feng and Lau \cite{fl2009}, states that for self-similar measures satisfying the weak separation condition, the multifractal formalism holds for any $q\geq 0$, and for $q<0$ there is an open set $U_0$ on which $\mu$ is sufficiently regular so that the $L^q$- and multifractal spectra restricted to $U_0$ satisfy the multifractal formalism.
However, the relatively open set $U_0\cap K$ is almost always a proper subset of $K$, so this result only gives a (somewhat coarse) lower bound for $f_\mu$.
The case for $q<0$ is more challenging to establish in general: indeed, we already saw for such self-similar measures that the multifractal formalism need not hold.

For measures satisfying the weak separation condition in $\R$, the author recently established general conditions based on connectivity properties of an associated graph for which the regularity on the set $U_0$ can be extended to the entire set $K$ \cite{rut2021}.
This can be applied to verify the multifractal formalism for all $q\in\R$ for certain examples such as those discussed in \cite[Prop. 4.3]{lw2004} or \cite[Ex. 8.5]{dn2017}.

Our work here vastly extends these results under a slightly more specialized hypothesis (detailed in \cref{d:pfnc}).
We will discuss our technical conditions and results in detail in the following section.
We are not aware of any IFS satisfying the weak separation condition for which the technical conditions do not hold.

\subsection{Main results and outline of the paper}
\subsubsection{Symbolic encoding and the transition graph}
In \cref{s:gt-construct}, we define a generalized version of the constructions in \cite{fen2003,hhs2018,rut2021} which provides a more cohesive perspective on the ``net interval'' constructions defined therein and simplifies the study of certain examples.
The construction is based on the idea of an \defn{iteration rule} $\Phi$ (see \cref{d:iter}), which describes how to define inductively a nested heirarchy of partitions $\{\mathcal{P}_n\}_{n=0}^\infty$ in a way which depends only on the local geometry of $K$ (see \cref{p:ttype}).
The end result is to construct a rooted directed graph $\mathcal{G}$, which we call the \defn{transition graph}.
The edges of the graph $\mathcal{G}$ are equipped with matrices $T(e)$, such that norms of products of matrices corresponding to finite paths beginning at the root vertex encode the measure $\mu$ on a rich family of subsets (this result is given in \cref{p:mat-mu}).
When the transition graph $\mathcal{G}$ is finite, we say that the IFS satisfies the \defn{finite neighbour condition with respect to $\Phi$}, or the $\Phi$-FNC for short (see \cref{d:pfnc}).
For the remainder of this paper, we will assume that this condition is satisfied.

We denote by $\Omega^\infty$ the set of infinite paths in $\mathcal{G}$ originating at the root vertex, which is equipped with an ``almost injective'' Lipschitz projection $\pi:\Omega^\infty\to K$.
The set $\Omega^\infty$ can be thought of as ``symbolic'' analogue of $K$, where the weights $W(e)$ encode the metric structure of $K$ and the matrices $T(e)$ encode the self-similar measure $\mu$.
The overarching approach in this paper is to establish results in the space $\Omega^\infty$, and then using the projection $\pi$ obtain corresponding results about the multifractal analysis of the self-similar measure $\mu$.
The main technical challenge is that the map $\pi$ is not, in general, bi-Lipschitz.

The graph $\mathcal{G}$ need not be strongly connected.
We call the non-trivial connected components of $\mathcal{G}$ \defn{loop classes}, which we define fully in \cref{ss:irred}.
Since the tail of any infinite path is an infinite path in a loop class, we obtain a decomposition
\begin{equation*}
    \Omega^\infty=\bigcup_{i=1}^m\Omega^\infty_{\mathcal{L}_i}
\end{equation*}
for appropriate sets $\Omega^\infty_{\mathcal{L}_i}$, where the union is disjoint.
This decomposition of $\mathcal{G}$ will correspond directly (outside certain degenerate situations) with the decomposition given in \cref{t:meas-split}.
For example, a graphic of a (hypothetical) transition graph is given in \cref{f:gen-tr-graph} and one can observe that there are 4 non-trivial strongly connected components $\mathcal{L}_i$ for $i=1,\ldots,4$.
\begin{figure}[ht]
    \begin{tikzpicture}[
    baseline=(current bounding box.center),
    scale=1.3
    ]
    \node[vtx,label=above right:{$\vroot$}] (vroot) at (0,0) {};

    \coordinate (loop1) at (-2.7, -0.4);
    \coordinate (loop2) at (3, -1.6);
    \coordinate (loop3) at (-2.5, -2);
    \coordinate (ess) at (1,-3.5);

    \node[vtx] (v1) at (-0.5,-1) {};
    \node[vtx] (v2) at (1,-1.5) {};
    \node[vtx] (v3) at (0.5,-2.5) {};
    \node[vtx] (v4) at (2.9,-2.7) {};

    \draw[thick,dotted] (loop1) ellipse (0.8cm and 0.5cm);
    \node[fill=white] (l1Label) at ($(-2.7,-0.4) + (180:0.8cm and 0.5cm) $) {$\mathcal{L}_1$};

    \node[vtx] (l1a) at (-2.5,-0.5) {};

    \draw[edge] (l1a) .. controls +(145+45:1) and +(145-45:1) .. (l1a);

    \draw[thick,dotted] (loop2) ellipse (1.2cm and 0.6cm);
    \node[fill=white] (l2Label) at ($(3,-1.6) + (0:1.3cm and 0.7cm) $) {$\mathcal{L}_2$};

    \node[vtx] (l2a) at (3.5,-1.9) {};
    \node[vtx] (l2b) at (2.2,-1.5) {};
    \node[vtx] (l2c) at (3.1,-1.2) {};

    \draw[edge] (l2a) -- (l2b);
    \draw[edge] (l2b) -- (l2c);
    \draw[edge] (l2c) -- (l2a);

    \draw[thick,dotted] (-2.5,-2) ellipse (1.3cm and 0.7cm);
    \node[fill=white] (l3Label) at ($(-2.5,-2) + (180:1.3cm and 0.7cm) $) {$\mathcal{L}_3$};

    \node[vtx] (l3a) at (-1.9,-1.8) {};
    \node[vtx] (l3b) at (-2.5,-2.3) {};
    \node[vtx] (l3c) at (-3.0,-1.6) {};
    \node[vtx] (l3d) at (-3.0,-2.2) {};

    \draw[edge] (l3a) -- (l3b);
    \draw[edge] (l3b) -- (l3d);
    \draw[edge] (l3d) -- (l3c);
    \draw[edge] (l3c) -- (l3a);
    \draw[edge] (l3b) -- (l3c);

    \draw[thick,fill=gray!30] (ess) ellipse (1.5cm and 0.5cm);
    \node[fill=white] (essLabel) at (ess) {$\mathcal{L}_4$};

    \draw[edge] (vroot) -- (l1a);
    \draw[edge] (vroot) -- (v1);
    \draw[edge] (vroot) -- (v2);

    \draw[edge] (l3b) -- ($ (ess) + (140:1.5cm and 0.5cm) $);
    \draw[edge] (v3) -- ($ (ess) + (100:1.5cm and 0.5cm) $);
    \draw[edge] (l2b) -- ($ (ess) + (70:1.5cm and 0.5cm) $);
    \draw[edge] (v4) -- ($ (ess) + (30:1.5cm and 0.5cm) $);

    \draw[edge] (l1a) -- (l3c);
    \draw[edge] (v1) -- (v3);
    \draw[edge] (v2) -- (v3);
    \draw[edge] (v1) -- (l3a);
    \draw[edge] (vroot) -- (l2b);
    \draw[edge] (l2a) -- (v4);

\end{tikzpicture}
    \caption{A ``generic'' transition graph}
    \label{f:gen-tr-graph}
\end{figure}
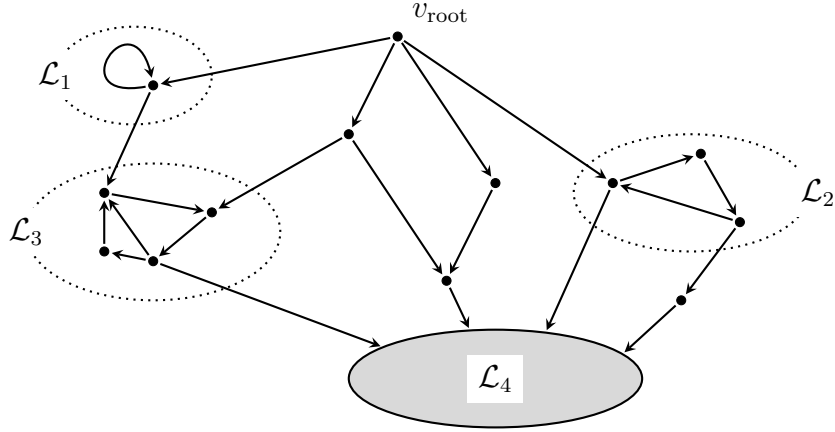

\subsubsection{Loop classes and the upper bounds}
There can be components $\mathcal{L}_i$ where the corresponding sets $\pi(\Omega_{\mathcal{L}_i})\subseteq K$ have measure 0 (in \cref{f:gen-tr-graph}, this is $\mathcal{L}_1$, $\mathcal{L}_2$, and $\mathcal{L}_3$).
However, even though the measure $\mu$ cannot be restricted to $\pi(\Omega_{\mathcal{L}_i})$ in a sensible way, the corresponding symbolic measure (which we denote by $\rho$) does restrict properly.
In \cref{ss:symb-defs}, we define symbolic analogues $\tau_{\mathcal{L}_i}$ of the $L^q$-spectrum and $f_{\mathcal{L}_i}$ of the multifractal spectrum for the loop classes $\mathcal{L}_i$.
These functions can be interpreted as $L^q$-spectra and multifractal spectra of some appropriate subadditive set functions defined on $\pi(\Omega_{\mathcal{L}_i})$.

In \cref{l:lq-upper-bound} and \cref{t:m-upper-bound}, we establish the following general upper bounds.
\begin{itheorem}\label{t:gen-upper}
    Suppose $\mu$ is a self-similar measure satisfying the $\Phi$-FNC with loop classes $\mathcal{L}_1,\ldots,\mathcal{L}_m$ and corresponding symbolic $L^q$-spectra $\tau_{\mathcal{L}_1},\ldots,\tau_{\mathcal{L}_m}$.
    Then
    \begin{align*}
        f_\mu(\alpha)&\leq\max\{\tau_1^*(\alpha),\ldots,\tau_m^*(\alpha)\} & \tau_\mu(q)\leq\min\{\tau_1(q),\ldots,\tau_m(q)\}.
    \end{align*}
\end{itheorem}
Unlike the general upper bound $f_\mu\leq\tau_\mu^*$ \cite[Thm. 4.1]{ln1999}, this upper bound for $f_\mu$ follows by an argument which depends sensitively on the existence of the local dimension in the definition of $f_\mu(\alpha)$.
The precise ideas here can be found in \cref{l:approx-reg} and the surrounding discussion.
Note that upper bound given in \cref{t:gen-upper} is already a non-trivial improvement on the general bound $\tau_\mu^*$ when $\max\{\tau_1^*,\ldots,\tau_m^*\}$ is not a concave function.
Indeed, since $\tau_\mu(q)\leq\tau_i(q)$, we have $\tau_\mu^*(\alpha)\geq\tau_i^*(\alpha)$ so that
\begin{equation*}
    \tau_\mu^*\geq\max\{\tau_1^*,\ldots,\tau_m^*\},
\end{equation*}
but $\tau_\mu^*$ is necessarily concave.

\subsubsection{Irreducibility, decomposability, and the lower bounds}
In order to establish the lower bounds, we require two main assumptions.
The first, which we call \defn{irreducibility}, can be interpreted as an internal connectivity property for the loop classes, and depends only on properties of the paths and transition matrices internal to some loop class $\mathcal{L}_i$ (see \cref{sss:irreducibility}).
This property was introduced and studied in \cite{fen2009}; as with that paper, this result is essential for establishing the symbolic multifractal formalism in \cref{t:multi-f}.
The irreducibility assumption is also important to resolve the fact that the projection $\pi$ is not, in general, bi-Lipschitz.
This technical result is given in \cref{t:reg-sub}.
While irreducibility formally depends on the choice of probabilities, in practice, every example of which the author is aware can be verified by the slightly stronger hypothesis of \cref{l:irred}, which does not depend on the choice of probabilities.

The second main assumption, which we call \defn{decomposability}, is a statement about the finite paths which do not have any edges in loop classes (see \cref{sss:decomposable}).
This property is closely related to the positive transition matrix assumption in \cite{hhstoappear}, and our proof of \cref{t:lq-lower-bound} largely follows the ideas in that document.
This assumption allows a product-like decomposition of $\Omega^\infty$ as $\Omega_{\mathcal{L}_1}^\infty\times\cdots\times\Omega_{\mathcal{L}_m}^\infty$ in a way which preserves the norms of matrices.
See \cref{e:Psi-def} for the precise statement and application of this idea.

We will also assume a simple non-degeneracy property (given in \cref{d:degen}).
Similar statements can be made assuming some loop classes are degenerate, but we omit this discussion for simplicity.
We then have the following result, proven in \cref{c:m-spectrum} and \cref{t:lq-lower-bound}.
\begin{itheorem}\label{t:meas-split}
    Suppose $\mu$ is a self-similar measure satisfying the $\Phi$-FNC with loop classes $\mathcal{L}_1,\ldots,\mathcal{L}_m$ and corresponding symbolic $L^q$-spectra $\tau_{\mathcal{L}_1},\ldots,\tau_{\mathcal{L}_m}$.
    Suppose each loop class is non-degenerate.
    Then:
    \begin{enumerate}[nl,r]
        \item If the irreducibility assumption is satisfied,
            \begin{equation*}
                f_\mu(\alpha)=\max\{\tau_{\mathcal{L}_1}^*(\alpha),\ldots,\tau_{\mathcal{L}_m}^*(\alpha)\}.
            \end{equation*}
        \item If the decomposability assumption is satisfied, the limit defining $\tau_\mu(q)$ exists for every $q\in\R$.
            Moreover,
            \begin{equation*}
                \tau_\mu(q)=\min\{\tau_{\mathcal{L}_1}(q),\ldots,\tau_{\mathcal{L}_m}(q)\}.
            \end{equation*}
    \end{enumerate}
\end{itheorem}
Outside the open set condition \cite{ap1996} and the case $q\geq 0$ \cite{ps2000}, there does not appear to be any general existence results for the limit $\tau_\mu(q)$ when $\mu$ is a self-similar measure.
Moreover, the author is not aware of any self-similar measure satisfying the weak separation condition which does not satisfy all the hypotheses in \cref{t:meas-split}.
To provide evidence for this claim, we observe that the hypotheses are satisfied for a number of examples (see \cref{s:multi-examples}).
However, verifying these conditions in general seems to be a very challenging question.

We can now use \cref{t:meas-split} to describe precisely when the multifractal formalism holds.
We say that $\mu$ satisfies the multifractal formalism at $\alpha$ if $f_\mu(\alpha)=\tau_\mu^*(\alpha)$.
Recall that the subdifferential $\partial\tau_{\mathcal{L}_i}(q)$ is the interval from the right derivative to the left derivative of $\tau_{\mathcal{L}_i}$ at $q$.
The following result is proven in \cref{c:multi-validity}.
\begin{icorollary}
    Let $\mu$ satisfy the same hypotheses as \cref{t:meas-split}, along with the irreducibility and decomposability assumptions.
    Then $\mu$ satisfies the multifractal formalism at $\alpha$ if and only if $\alpha\in\partial\tau_{\mathcal{L}_i}(q)$ for some $1\leq i\leq m$ and $q\in\R$ with $\min\{\tau_{\mathcal{L}_1}(q),\ldots,\tau_{\mathcal{L}_m}(q)\}=\tau_{\mathcal{L}_i}(q)$.
    In particular, if the derivative $\alpha=\tau_\mu'(q)$ exists, then $\mu$ satisfies the multifractal formalism at $\alpha$.
\end{icorollary}
In other words, the multifractal formalism fails precisely on phase transitions (values of $\alpha$ corresponding to points of non-differentiability of the $L^q$-spectrum) caused by transitions in $\min\{\tau_{\mathcal{L}_1}(q),\ldots,\tau_{\mathcal{L}_m}(q)\}$ from some $\tau_{\mathcal{L}_i}(q)$ to $\tau_{\mathcal{L}_j}(q)$ for $i\neq j$.
This corollary is highlighted in \cref{f:phase-t} with two loop class $\mathcal{L}_1$ and $\mathcal{L}_2$ such that $\tau_{\mathcal{L}_1}$ and $\tau_{\mathcal{L}_2}$ intersect.
For values of $\alpha$ corresponding to the phase transition of $\tau_\mu=\min\{\tau_{\mathcal{L}_1},\tau_{\mathcal{L}_2}\}$ at their intersection point $q_0$, we see that $\tau_\mu^*$ differs from $f_\mu=\max\{\tau_{\mathcal{L}_1}^*,\tau_{\mathcal{L}_2}^*\}$.
Here, the multifractal formalism is satisfied at $\alpha$ if and only if $\alpha\notin(\alpha_2,\alpha_1)$.
In fact, $\tau_\mu^*$ is the infimal concave function bounded below by $f_\mu$.
Thus the phase transitions which cause the multifractal formalism to fail are fundamentally linked to the connectivity properties of the transition graph.
For example, this provides a general explanation for the phenomenon observed by Testud \cite{tes2006a} for self-similar measures associated with digit-like sets (see \cref{ss:tes-ex}).
\begin{figure}[htp]
    \subfloat[$L^q$-spectra]{
        \def\intPointx{-0.330566}
\def\intPointy{-0.984925}
\def\fintPointx{1.15269}
\def\fintPointy{0.201244}
\def\phaseStart{0.863887}
\def\phaseEnd{1.83555}
\def\phaseStartY{0.699355}
\def\phaseEndY{0.378154}
\begin{tikzpicture}[>=stealth,xscale=2,font=\tiny]
    \draw[->,] (-1.6,0) -- (1.6,0);
    \draw[->] (0,-4) -- (0,2);
    \draw[thick,dotted] (\intPointx,\intPointy) -- (1.5, {\phaseStart * (1.5-\intPointx)+\intPointy)}) node[right]{$\begin{aligned}\text{\textit{slope }}&=\alpha_2\\&=\tau_{\mathcal{L}_2}'(q_0)\end{aligned}$};
    \draw[thick,dotted] (\intPointx,\intPointy) -- (1.5, {\phaseEnd * (1.5-\intPointx)+\intPointy)}) node[right]{$\begin{aligned}\text{\textit{slope }}&=\alpha_1\\&=\tau_{\mathcal{L}_1}'(q_0)\end{aligned}$};
    \node[label=left:{$\tau_{\mathcal{L}_1}$}] (t1Label) at (-1.5, -3.57841) {};
    \node[label=left:{$\tau_{\mathcal{L}_2}$}] (t2Label) at (-1.5, -2.07886) {};
    \begin{scope}
        \path[clip] (-1.5,\intPointy) rectangle (1.3,-4);
        \draw[thick] plot file {figures/tau1_points.txt};
        \draw[thick,gray] plot file {figures/tau2_points.txt};
    \end{scope}
    \begin{scope}
        \path[clip] (-1.5,\intPointy) rectangle (1.3,2);
        \draw[thick,gray] plot file {figures/tau1_points.txt};
        \draw[thick] plot file {figures/tau2_points.txt};
    \end{scope}
    \node[circle,inner sep=0.8pt,draw=black,fill=black] at (\intPointx,\intPointy) {};
    \draw[thick,dotted] (\intPointx,\intPointy) -- (\intPointx,0.2) node[above]{$q_0$};
    \matrix [draw,below right] at (current bounding box.north west) {
        \node [strike out,draw=black,thick,label=right:{$\min\{\tau_{\mathcal{L}_1},\tau_{\mathcal{L}_2}\}$}] {}; \\
        \node [strike out,draw=gray,thick,label=right:{$\tau_{\mathcal{L}_1}\text{ or }\tau_{\mathcal{L}_2}$}] {}; \\
    };
\end{tikzpicture}
    }

    \subfloat[Concave conjugates and multifractal spectra]{
        \def\intPointx{-0.330566}
\def\intPointy{-0.984925}
\def\fintPointx{1.15269}
\def\fintPointy{0.201244}
\def\phaseStart{0.863887}
\def\phaseEnd{1.83555}
\def\phaseStartY{0.699355}
\def\phaseEndY{0.378154}
\begin{tikzpicture}[>=stealth,scale=4,font=\tiny]
    \draw[fill=gray!5,draw=none] (\phaseStart,0) rectangle (\phaseEnd,1);
    \draw[->] (-0.1,0) -- (3,0);
    \draw[->] (0,-0.1) -- (0,1);
    \begin{scope}
        \path[clip] (0,0) rectangle (\phaseStart,1);
        \draw[thick] plot file {figures/f1_points.txt};
        \draw[thick] plot file {figures/f2_points.txt};
    \end{scope}
    \begin{scope}
        \path[clip] (\phaseStart,\fintPointy) rectangle (\phaseEnd,1);
        \draw[gray,thick] plot file {figures/f1_points.txt};
        \draw[gray,thick] plot file {figures/f2_points.txt};
    \end{scope}
    \begin{scope}
        \path[clip] (\phaseStart,0) rectangle (\phaseEnd,\fintPointy);
        \draw[dashed,thick] plot file {figures/f1_points.txt};
        \draw[dashed,thick] plot file {figures/f2_points.txt};
    \end{scope}
    \begin{scope}
        \path[clip] (\phaseEnd,0) rectangle (3,1);
        \draw[thick] plot file {figures/f1_points.txt};
        \draw[thick] plot file {figures/f2_points.txt};
    \end{scope}
    \draw[thick] (\phaseStart,\phaseStartY) -- (\phaseEnd,\phaseEndY);
    \draw[thick,dotted] (\phaseStart,-0.05) node[below]{$\alpha_2$} -- (\phaseStart,1);
    \draw[thick,dotted] (\phaseEnd,-0.05) node[below]{$\alpha_1$} -- (\phaseEnd,1);
    \matrix [draw,below left] at (current bounding box.north east) {
        \node [strike out,draw=black,thick,label=right:{$(\min\{\tau_{\mathcal{L}_1},\tau_{\mathcal{L}_2}\})^*$}] {}; \\
        \node [strike out,draw=gray,thick,label=right:{$\max\{\tau_{\mathcal{L}_1}^*,\tau_{\mathcal{L}_2}^*\}$}] {}; \\
    };
    \draw[<->,>=stealth,shorten <=1pt,shorten >=1pt] (\phaseStart,0.85) --node[fill=white]{\textit{phase transition}} (\phaseEnd,0.85);
\end{tikzpicture}
    }
    \caption{An example illustrating a non-trivial phase transition}
    \label{f:phase-t}
\end{figure}
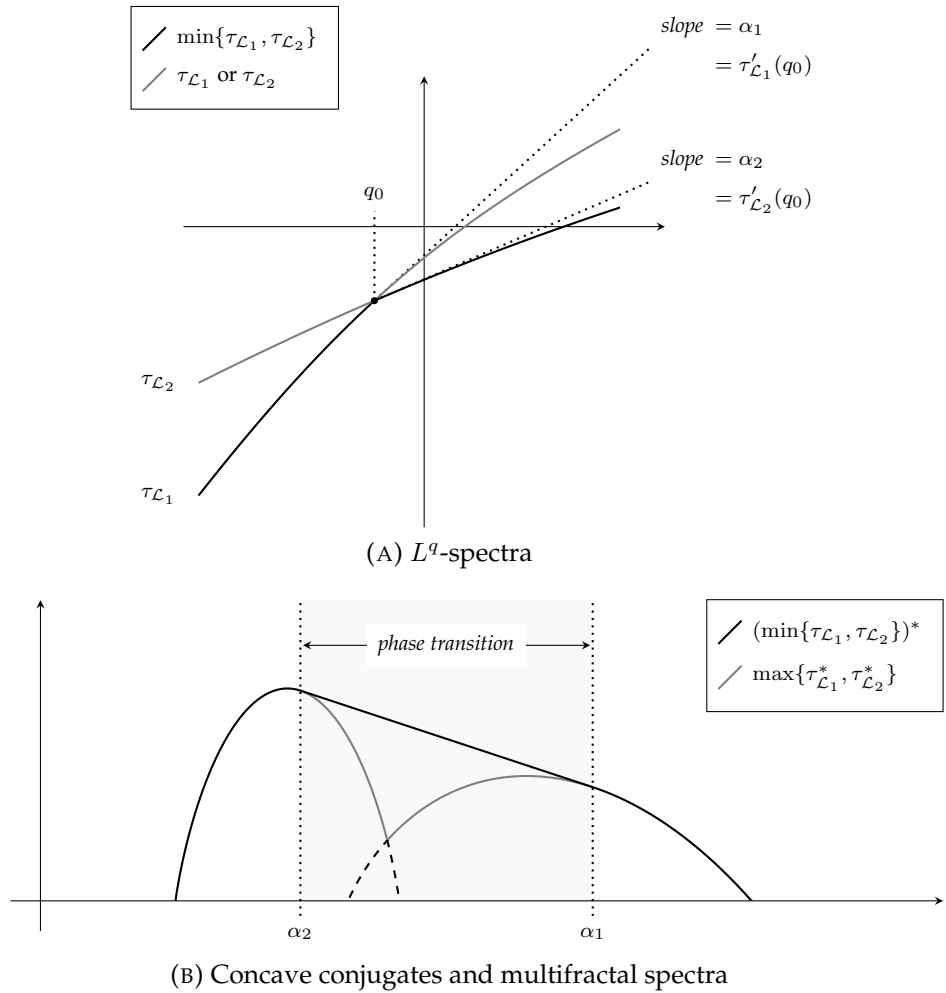

There can be phase transitions not of this form: for example, for the Bernoulli convolution associated with the Golden mean, $\tau_\mu=\tau_{\mathcal{L}}$ for a loop class $\mathcal{L}$ but $\tau_\mu(q)$ is not differentiable \cite{fen2005}.
Our results provide some explanation for the phenomenon of self-similar measures with non-differentiable $L^q$-spectra which still satisfy the multifractal formalism.

A \defn{simple} loop class is a loop class where the edges can be ordered to form a cycle which does not repeat vertices.
In \cref{f:gen-tr-graph}, the simple loop classes are given by $\mathcal{L}_1$ and $\mathcal{L}_2$.
As a straightforward application of \cref{t:meas-split} along with basic properties of concave functions, we obtain the following result.
The proof of this result can be found in \cref{ss:cons}.
\begin{icorollary}\label{c:all-simple}
    Let $\mu$ satisfy the same hypotheses as \cref{t:meas-split}, along with the irreducibility and decomposability assumptions.
    Then $\mu$ satisfies the multifractal formalism if and only if the multifractal spectrum is a concave function.
    In particular, if every non-essential loop class is simple, this happens if and only if the set of local dimensions is a closed interval.
\end{icorollary}

\subsubsection{Applications and analysis of examples}
The hypotheses in \cref{c:all-simple} are satisfied in many well-known examples.
Here, we list some IFSs for which \cref{c:all-simple} applies, so that any associated self-similar measure satisfies the multifractal formalism if and only if the set of local dimensions is a closed interval:
\begin{itemize}
    \item the family $\{\frac{x}{d}+\frac{j}{md}(d-1):j=0,1,\ldots,m\}$ with $m\geq d-1\geq 1$ integers, which includes the 3-fold convolution of the Cantor measure \cite{hl2001}, and is discussed in detail in \cite[Sec. 5]{hhn2018}.
    \item Bernoulli convolutions with parameters that are reciprocals of simple Pisot numbers \cite{fen2005}, or reciprocals of the Pisot roots of the polynomials $x^3 - 2x^2 + x - 1$, $x^4-x^3- 2x^2+ 1$, and $x^4- 2x^3+ x - 1$ (see \cref{ss:bconv-Pisot}).
    \item the IFS $\{\rho x,\rho^2 x+\rho-\rho^2,\rho^2x+1-\rho^2\}$ where $1/\rho$ is the Golden mean, considered in \cite[Sec. 5.3.3]{hr2021}.
\end{itemize}
By combining our results with the detailed study of sets of local dimensions contained in the references cited above, we obtain a number of new examples of measures satisfying the multifractal formalism which were not previously known in the literature.
Such results about the validity of the multifractal formalism were previously only known for Bernoulli convolutions associated with simple Pisot numbers \cite{fen2005}.
We refer the reader to \cite{hr2021} for details related to the computation of sets of local dimensions under similar assumptions to this paper.

To conclude this paper, we will provide a detailed study of some examples in \cref{s:multi-examples} to illustrate more concretely how our results may be applied in specific situations.
Our selection of examples does not attempt to be exhaustive, and the examples are primarily chosen to illustrate how our results explain the different multifractal phenomena exhibited by self-similar measures satisfying the weak separation condition.

In \cref{ss:tes-ex}, we study a family of self-similar measures associated with an IFS with maps of the form $x\mapsto x/\ell+i/\ell$ or $x\mapsto -x/\ell+(i+1)/\ell$ where $\ell\geq 2$ is an integer and $i\in\{0,1,\ldots,\ell-1\}$.
Such measures were first studied by Testud \cite{tes2006a}, where he provided some of the first known examples of self-similar measures which exhibit non-trivial non-concave spectra.
Our results extend and contextualize his results, since we do not require any assumptions on the digit sets.
This also extends results obtained by Olsen and Snigireva  \cite{os2008} for such measures.

In \cref{ss:bconv-Pisot}, we provide a simple (given our general results) verification of the multifractal formalism for Bernoulli convolutions with parameters that are reciprocals of simple Pisot numbers.
This fact was first observed by Feng \cite{fen2005}.
Our technique is more general and depends only on establishing certain structural properties of the transition graph.
Our results also apply, for example, to the polynomials $x^3 - 2x^2 + x - 1$, $x^4-x^3- 2x^2+ 1$, and $x^4- 2x^3+ x - 1$.

Finally, in \cref{ss:non-e}, we verify the multifractal formalism for any self-similar measure associated with a class of IFS generalizing an example of Lau and Wang \cite{lw2004}, which is the IFS $\{\lambda_1 x, \lambda_2 x+\lambda_1(1-\lambda_2),\lambda_2 x+(1-\lambda_2)\}$.
The multifractal formalism for the self-similar measure studied by Lau and Wang was first verified by the author in a recent paper \cite{rut2021}.
We provide a simplified proof of this fact, which generalizes naturally to a family of related examples (which also includes \cite[Ex. 8.5]{dn2017}).

\subsubsection{Questions}
We conclude this section with three natural questions.
\begin{enumerate}[nl]
    \item Are the hypotheses in \cref{t:meas-split} satisfied for every measure $\mu$ satisfying the weak separation condition?
        Both a counterexample or a proof of non-existence here would be very interesting.
    \item In what generality does a version of \cref{t:meas-split} hold?
        Is the multifractal spectrum of any self-similar measure always the maximum of a finite set of concave functions?
    \item If $\mu$ is a self-similar (or self-conformal) measure and $f_\mu$ is a concave function, is it necessarily true that $\mu$ satisfies the multifractal formalism?
\end{enumerate}
\subsection{Acknowledgements}
The author would like to thank Kathryn Hare for many extensive discussions concerning the topics in this paper.
The author also thanks Jonathan Fraser and Kenneth Falconer for detailed comments on a draft version of this paper, and more generally for helpful comments and suggestions.

\subsection{Notation}
Given a general set $X$, we denote by $\#X$ the cardinality of the set $X$.

We denote by $\R$ the set of reals equipped with the standard Euclidean metric.
All sets and functions considered in this document are Borel unless otherwise noted.
If $\mu$ is a Borel measure and $f$ a measurable function, we denote by $f\mu$ the push-forward of $\mu$ by $f$, which is given by the rule
\begin{equation*}
    f\mu(E)=\mu(f^{-1}(E)).
\end{equation*}
Given a Borel set $E$, we write $E^\circ$ to denote the topological interior and $\diam(E)$ the diameter of $E$.
A map $f:\R\to\R$ is Lipschitz if $|f(x)-f(y)|\leq C|x-y|$ for some $C>0$.
Then $f$ bi-Lipschitz if $f$ is invertible with Lipschitz inverse, and a similarity of equality holds.

Given families $(a_i)_{i\in I}$ and $(b_i)_{i\in I}$ of non-negative real numbers, we write $a_i\preccurlyeq b_i$ if there exists some constant $C>0$ such that $a_i\leq Cb_i$ for each $i\in I$.
We say $a_i\asymp b_i$ if $a_i\preccurlyeq b_i$ and $b_i\preccurlyeq a_i$.
We will always allow such relationships to depend implicitly on the governing iterated function system (including the probabilities) and the transition rule $\Phi$.
Any other dependence, unless otherwise stated, will be indicated explicitly with a subscript.


\section{Some brief preliminaries}
\subsection{Weighted iterated function systems}
In our setting, a \defn{weighted iterated function system} (WIFS) is a tuple $(S_{i},p_i)_{i\in\mathcal{I}}$ where
\begin{equation}  \label{e:ifs}
    S_{i}(x)=r_{i}x+d_{i}:\mathbb{R}\rightarrow \mathbb{R}\text{ for each }i\in\mathcal{I}
\end{equation}
with $0<\left\vert r_{i}\right\vert <1$, so that each $S_i$ is a contracting similaritity in $\R$, and the $p_i$ satisfy $p_i>0$ and $\sum p_i=1$.
We refer to the tuple $(S_i)_{i\in\mathcal{I}}$ simply as an \defn{iterated function system} (IFS).

There are two important invariant objects associated with a WIFS, both of which can be realized as the unique fixed point of a contraction mapping on an appropriate metric space.
The first is a non-empty, compact set $K$ satisfying 
\begin{equation*}
    K=\bigcup_{i\in\mathcal{I}}S_{i}(K), 
\end{equation*}
known as the \defn{self-similar set} associated with the WIFS.
The second is a Borel probability measure $\mu$ satisfying
\begin{equation}\label{e:minv}
    \mu(E)=\sum_{i=1}^mp_{i}\cdot S_i\mu(E)
\end{equation}
for any Borel set $E\subseteq K$, where $S_i\mu(E)=\mu(S_i^{-1}(E))$ is the pushforward of $\mu$ by $S_i$.
We say that $\mu$ is the \defn{self-similar measure} associated with the WIFS.
We refer the reader to the book of Falconer \cite{fal1997} for details concerning the existence and uniqueness of these objects.

Note that $\supp\mu=K$.
Throughout this document, we will assume that $K$ is not a singleton, so that $\mu$ is a non-atomic measure.
By conjugating the maps as necessary (which amounts to an appropriate translation of the $d_i$), we may assume that the convex hull of $K$ is $[0,1]$.

Let $\mathcal{I}^*=\bigcup_{n=0}^\infty\mathcal{I}^n$ denote the set of all finite tuples on $\mathcal{I}$.
Given $\sigma = (\sigma_1,\ldots,\sigma_n)\in\mathcal{I}^n$, we write
\begin{equation*}
    S_\sigma = S_{\sigma_1}\circ\cdots\circ S_{\sigma_n},\qquad r_\sigma = r_{\sigma_1}\circ\cdots\circ r_{\sigma_n},
\end{equation*}
and
\begin{equation*}
    p_\sigma = p_{\sigma_1}\circ\cdots\circ p_{\sigma_n}.
\end{equation*}
Abusing notation slightly, we denote the empty word (the unique word of length zero) by $\emptyset$ and write $S_\emptyset=\id$, $p_\emptyset=1$, and $r_\emptyset=1$.
Given another word $\tau=(\tau_1,\ldots,\tau_m)$, the \defn{concatenation} $\sigma\tau$ is the word $(\sigma_1,\ldots,\sigma_n,\tau_1,\ldots,\tau_m)$.
We say that a word $\sigma$ is a \defn{prefix} of $\tau$ if there exists some $\omega$ such that $\tau=\sigma\omega$.

\subsection{Concave functions}
Let $f:\R\to\R\cup\{-\infty\}$ be a concave function.
The \defn{subdifferential} of $f$ at $x$ is given by
\begin{equation*}
    \partial f(x)=\{\alpha: \alpha(y-x)+f(x)\geq f(y)\text{ for any }y\in\R\}.
\end{equation*}
Of course, if $f$ is differentiable at $x$, then $\partial f(x)=\{f'(x)\}$.
The \defn{concave conjugate} of $f$ is the function
\begin{equation*}
    f^*(\alpha):=\inf\{\alpha x-f(x):x\in\R\}.
\end{equation*}
Naturally, the infimum may be attained at $-\infty$.
Note that $f^*$ is always concave, and concave convolution is involutive (i.e. $f^{**}=f$ when $f$ is a concave function).
We will use the fact that $f^*(\alpha)+f(x)=\alpha x$ whenever $\alpha\in\partial f(x)$.

We refer the reader to \cite{roc1970} for more detail and proofs of these facts.

\subsection{Local dimensions and multifractal analysis}\label{ss:mfa}
Let $\mu$ be a finite Borel measure in $\R$ with compact support.
\begin{definition}
    Let $x\in \supp\mu$ be arbitrary.
    Then the \defn{lower local dimension of $\mu$ at $x$} is given by
    \begin{equation*}
        \underline{\dim}_{\loc}(\mu,x)=\liminf_{t\to 0}\frac{\log \mu(B(x,t))}{\log t}
    \end{equation*}
    and the \defn{upper local dimension} $\overline{\dim}_{\loc}(\mu,x)$ is given similarly with the limit inferior replaced by the limit superior.
    When the values of the upper and lower local dimension agree, we call the shared value the \defn{local dimension of $\mu$ at $x$}, denoted $\dim_{\loc}(\mu,x)$.
\end{definition}
We are primarily interested in understanding geometric properties of the level sets of local dimensions.
Define
\begin{equation*}
    E_\mu(\alpha) =\bigl\{x\in \supp\mu:\underline{\dim}_{\loc}(\mu,x)=\overline{\dim}_{\loc}(\mu,x)=\alpha\bigr\}.
\end{equation*}
We will focus on the \defn{(fine Hausdorff) multifractal spectrum} of $\mu$, which is the function $f_\mu:\R\to\R\cup\{-\infty\}$ given by
\begin{equation*}
    f_\mu(\alpha):=\dim_H E_\mu(\alpha)
\end{equation*}
where, by convention, we write $\dim_H \emptyset = -\infty$.

A different (but related) way to quantify the density of $\mu$ is through the $L^q$-spectrum of the measure.
\begin{definition}
    The \defn{$L^q$-spectrum} of $\mu$ is given by
    \begin{equation*}
        \tau_\mu(q) = \liminf_{t\to 0}\frac{\log \sup \sum_i\mu(B(x_i,t))^q}{\log t}
    \end{equation*}
    where the supremum is over families of disjoint balls $\{B(x_i,t)\}_{i=1}^m$ centred at $x_i\in K$.
\end{definition}
Standard arguments show that the function $\tau_\mu$ is an increasing concave function of $q$.
Set
\begin{align*}
    \alpha_{\min}(\mu) &= \lim_{q\to\infty}\frac{\tau_\mu(q)}{q} & \alpha_{\max}(\mu) &= \lim_{q\to-\infty}\frac{\tau_\mu(q)}{q}.
\end{align*}
When $\mu$ is a self-similar measure, it is known that $\alpha_{\min}$ and $\alpha_{\max}$ are finite real numbers (see, for example, \cite[Cor. 3.2]{fl2009}).

The multifractal formalism is a heuristic relationship introduced in \cite{hjk+1986} which relates the $L^q$-spectrum and the multifractal spectrum of $\mu$ under certain conditions.
\begin{definition}
    Given $\alpha\in\R$, we say that the measure $\mu$ satisfies the \defn{multifractal formalism at $\alpha$} if
    \begin{equation*}
        f_\mu(\alpha)=\tau_\mu^*(\alpha)
    \end{equation*}
    where $\tau_\mu^*$ is the concave conjugate of $\tau_\mu$.
    We say that $\mu$ satisfies the \defn{(complete) multifractal formalism} if $\mu$ satisfies the multifractal formalism at every $\alpha\in\R$.
\end{definition}
In particular, $f_\mu(\alpha)$ is a concave function which takes finite values precisely on the interval $[\alpha_{\min}(\mu),\alpha_{\max}(\mu)]$.

It always holds that $f_\mu(\alpha)\leq\tau_\mu^*(\alpha)$ (see, for example, \cite[Thm. 4.1]{ln1999}) and if $E_\mu(\alpha)$ is non-empty, then $\alpha_{\min}\leq\alpha\leq\alpha_{\max}$ \cite[Cor. 3.2]{fl2009}.
However, as discussed in the introduction, the set of $\alpha$ where $E_\mu(\alpha)\neq\emptyset$ need not be a closed interval and, even if it is, the multifractal formalism need not hold \cite{tes2006a}.

\section{A generalized transition graph construction}\label{s:gt-construct}
Self-similar measures have a natural encoding as a projection of self-similar measures in sequence space.
Let $\mathcal{I}^\infty$ denote the set of all infinite sequences on the alphabet $\mathcal{I}$ equipped with the natural product metric.
Given a sequence $(i_n)_{n=1}^\infty\in\mathcal{I}^\infty$, define the projection $\pi_0:\mathcal{I}^\infty\to K$ by the rule
\begin{equation*}
    \pi_0((i_n)_{n=1}^\infty)=\lim_{n\to\infty}S_{i_1}\circ\cdots\circ S_{i_n}(0).
\end{equation*}
When the compact sets $S_i(K)$ are disjoint for distinct $i\in\mathcal{I}$, the map $\pi_0$ is bi-Lipschitz.
In this case, the value of the measure $\mu$ has a simple formula for a rich family of subsets of $K$, namely
\begin{equation}\label{e:meas-p}
    \mu\bigl(S_{i_1}\circ\cdots\circ S_{i_n}(K)\bigr)=p_{i_1}\cdots p_{i_n}.
\end{equation}
However, when the measure $\mu$ has overlaps, such a simple formula no longer holds since the projection $\pi_0$ fails (in some situations quite badly) to be bi-Lipschitz.

A technique to overcome this limitation was first introduced by Feng \cite{fen2003} and extended in \cite{hhs2018,rut2021}.
In the subsequent sections, we will introduce a convenient framework which generalizes the prior net interval constructions; this will allow the simplification of analysis of examples in \cref{s:multi-examples}.
As this construction underlies all the results in this paper, we informally summarize the main ideas here.

Recall that the convex hull of $K$ is $[0,1]$.
In \cref{ss:gen-net-iv}, we inductively construct a nested sequence of partitions of $K$ with mesh size tending to 0, which we will denote by $(\mathcal{P}_n)_{n=0}^\infty$.
Here, a partition $\mathcal{P}_n$ is a finite collection of closed intervals $\{\Delta_1,\ldots,\Delta_\ell\}$ where $\Delta_i^\circ\cap\Delta_j^\circ=\emptyset$ for $i\neq j$, $\Delta_i^\circ\cap K\neq\emptyset$, and $K\subset\bigcup_{i=1}^k\Delta_i$.
We set $\mathcal{P}_0=\{[0,1]\}$.
We will associate to each $\Delta\in\mathcal{P}_n$ a \defn{neighbour set} $\vs(\Delta)$ (an ordered tuple of similarity maps from $\R$ to $\R$) such that each similarity map is a normalized version of some word $S_\sigma$ with $S_\sigma(K)\cap\Delta^\circ$.
In the sense of \cref{l:dk-max}, we also require that $\vs(\Delta)$ does not contain repetitions and satisfies a sort of maximality.
For a given $\Delta\in\mathcal{P}_n$, we want that the \defn{children} $\{\Delta':\Delta'\in\mathcal{P}_{n+1},\Delta'\subset\Delta\}$ depend uniquely on $\vs(\Delta)$, as made precise in \cref{p:ttype}.
The \defn{(basic) iteration rule} given in \cref{d:biter} and \cref{d:iter} underpins this inductive construction (we think of the domain of an iteration rule as the set of all possible neighbour sets of net intervals), and the technical hypotheses ensure that the various properties listed above are satisfied.
Now, in \cref{ss:tg-sr}, we construct a directed \defn{transition graph} $\mathcal{G}$ with root vertex $\vroot$ such that the finite paths in $\mathcal{G}$ of length $n$ are in bijection with the partitions $\mathcal{P}_n$ (see \cref{l:netiv-sr}).
We associate to the edges in $\mathcal{G}$ \defn{transition matrices} such that the $\mu$-measure of a net interval is the norm of the corresponding products of matrices (see \cref{p:mat-mu}).

Our main assumption from this point on will be that the graph $\mathcal{G}$ is finite; more details on this assumption are given in \cref{ss:fnc}.
While $\mathcal{G}$ is not, in general, strongly connected, we can enumerate the non-trivial maximal strongly connected components as $\{\mathcal{L}_1,\ldots,\mathcal{L}_m\}$ (we refer to these as \defn{loop classes}, as defined in \cref{d:loop-class}).
Denote the set of infinite paths in $\mathcal{G}$ beginning at $\vroot$ by $\Omega^\infty$.
Given an infinite path $(e_n)_{n=1}^\infty$ in $\Omega^\infty$, there is a unique loop class $\mathcal{L}_i$ such that for all $n$ sufficiently large, $e_n$ is an edge in some loop class $\mathcal{L}_i$.
The bijections from finite paths of length $n$ to $\mathcal{P}_n$ induce a Lipschitz surjection $\pi:\Omega^\infty\to K$ (note that the metric structure on $\Omega^\infty$ is defined in \cref{ss:symb-defs}, and depends on the edge weights).
The space $\Omega^\infty$ equipped with the projection $\pi$ is analgous to the space $\mathcal{I}^\infty$ along with the projection $\pi_0$.
Moreover, since the $\mathcal{P}_n$ are partitions of $K$, $\pi$ is nearly a bijection (it is injective on all but countably many points), and while $\pi$ need not be bi-Lipschitz, it is close to being so in a heuristic sense.
The main cost is that we must replace the products of scalars in \cref{e:meas-p} with norms of products of matrices.
This introduces additional technical challenges, which necessitate the assumptions of irreducibility and decomposability, which are discussed in more detail in \cref{ss:irred}.

\subsection{Partitions and net intervals}\label{ss:gen-net-iv}
We write by $\Sim(\R)=\{f(x)=ax+b:a\in\R\setminus\{0\},b\in\R\}$ the set of similarity maps from $\R$ to $\R$, and equip $\Sim(\R)$ with the total order induced by the lexicographic order on the pairs $(a,b)$ (or any other fixed total order).
We then denote by $\Sim^*(\R)$ the set of finite tuples $(f_1,\ldots,f_m)$ where $f_1<\cdots<f_m$ and $m\in\N$ is arbitrary.
\begin{definition}\label{d:biter}
    A \defn{basic iteration rule} is a map $\Phi$ which associates to each tuple $(f_1,\ldots,f_m)$ in $\Sim^*(\R)$ a tuple $(\mathcal{C}_1,\ldots,\mathcal{C}_m)$ where each $\mathcal{C}_i$ is a finite subset of $\mathcal{I}^*$ satisfying the following condition: for all $n\in\N$ sufficiently large, every $\sigma\in\mathcal{I}^n$ has a unique prefix in $\mathcal{C}_i$.
\end{definition}
A good example to keep in mind is the rule basic iteration rule $\Phi(f_1,\ldots,f_m)=(\mathcal{I},\ldots,\mathcal{I})$.
This example is discussed in more detail in \cref{ex:uniform-transition}.

Given a closed interval $J\subseteq\R$, we denote by $T_J$ the unique similarity $T_J(x)=rx+a$ with $r>0$ such that
\begin{equation*}
    T_J([0,1])=J:
\end{equation*}
Of course, $r=\diam(J)$ and $a$ is the left endpoint of $J$.

Using the notion of a basic iteration rule, we can inductively construct a heirarchy of partitions of $K$ as follows.
First, suppose we are given a pair $(\Delta,v)$ where $\Delta=[a,b]$ is a closed interval and $v=(f_1,\ldots,f_m)\in\Sim^*(\R)$.
Let
\begin{align}
    \begin{split}\label{e:H-def}
        \mathcal{Y} = \mathcal{Y}(\Delta,v) &= \bigcup_{f\in v}\{T_\Delta\circ f\circ S_\tau:\tau\in\mathcal{C}_i,1\leq i\leq m\}\\
        Y = Y(\Delta,v) &= \{a,b\}\cup\{g(z):g\in\mathcal{Y},z\in\{0,1\},g(z)\in\Delta\}
    \end{split}
\end{align}
and write the elements of $Y$ as $a=y_1<\cdots<y_{k+1}=b$.
Order the intervals $\{[y_i,y_{i+1}]:(y_i,y_{i+1})\cap K\neq\emptyset\}$ from left to right as $(\Delta_1,\ldots,\Delta_n)$.
We then define the \defn{children} of $\Delta$ (with respect to $\Phi$) as the set of pairs $(\Delta_i,v_i)$ where $v_i$ is given by ordering the distinct elements of the set
\begin{equation}\label{e:ch-nb-def}
    \{T_{\Delta_i}^{-1}\circ g: g\in\mathcal{Y}, g(K)\cap\Delta_i^\circ\neq\emptyset\}.
\end{equation}
If $\Delta_i=[a_i,b_i]$, the \defn{position index} is given by $q(\Delta_i,\Delta)=(a_i-a)/\diam(\Delta)$.
The position index is used to distinguish distinct children of $\Delta$ with the same neighbour set.

Now, using the above procedure, we can inductively construct our net intervals and neighbour sets.
Begin with the pair $\{([0,1],(\id))\}=\mathcal{N}_0$.
Having constructed $\mathcal{N}_n$ for some $n\in\N\cup\{0\}$, we denote by $\mathcal{N}_{n+1}$ the set of all children of pairs $(\Delta,v)\in\mathcal{N}_n$, and let $\mathcal{N}=\bigcup_{n=0}^\infty\mathcal{N}_n$.
Set
\begin{equation*}
    \mathcal{P}_n=\{\Delta:(\Delta,v)\in\mathcal{N}_n\},
\end{equation*}
which is the set of all net intervals at level $n$.
Since the net intervals are disjoint except on endpoints, one may think of $\mathcal{P}_n$ as a partition of $K$.

Note that distinct intervals $\Delta$ overlap at most on endpoints, and for each $x\in K$ and $n\in\N$ there is some $\Delta\in\mathcal{P}_n$ with $x\in\Delta$.
Given some $(\Delta,v)\in\mathcal{N}_n$, we say that $\Delta$ is a \defn{net interval} of level $n$, and that $v$ is the \defn{neighbour set} of $\Delta$.
We refer to a similarity $f\in v$ as a \defn{neighbour} of $\Delta$.
When the level $n$ is implicit, we write $\vs(\Delta)$ to denote the neighbour set $v$.
For an example computing the net intervals and neighbour sets, see \cref{ex:uniform-transition}.

We make two basic observations which follow immediately from the construction by an induction argument.
\begin{itemize}
    \item Let $(\Delta,v)\in\mathcal{V}$ with $f\in v$.
        Then $T_\Delta\circ f=S_\sigma$ for some $\sigma\in\mathcal{I}^*$.
    \item If $[a,b]=\Delta\in\mathcal{P}_m$, there exists some $(\Delta_0,v)\in\mathcal{V}_k$ with $k\leq m$ and $f\in v$ such that $a=T_{\Delta_0}\circ f(z)$ for some $z\in\{0,1\}$.
        The same statement also holds for $b$.
\end{itemize}
Here is a short example illustrating the net interval construction, along with these two observations.
\begin{example}\label{e:gen-ifs}
    Consider the IFS given by the maps
    \begin{align*}
        S_1(x) &= \frac{x}{3} & S_2(x) &= \frac{x}{3}+\frac{2}{9} & S_3(x) &= \frac{x}{3}+\frac{2}{3}
    \end{align*}
    along with the basic iteration rule given by $\Phi(f_1,\ldots,f_n)=(\mathcal{I},\ldots,\mathcal{I})$.
    By definition of $\Phi$, we have
    \begin{equation*}
        \mathcal{Y}([0,1],\{\id\})=\{T_{[0,1]}\circ \id\circ S_i:i\in\mathcal{I}\}=\{S_1,S_2,S_3\}
    \end{equation*}
    and, expanding the definition,
    \begin{equation*}
        Y([0,1],\{\id\})=\{S_i(z):i\in\mathcal{I},z\in\{0,1\}\}= \bigl\{0,\frac{2}{9},\frac{1}{3},\frac{5}{9},\frac{2}{3},1\bigr\}.
    \end{equation*}
    Note that $(2/3,1)\cap K=\emptyset$ (since $(2/3,1)\cap S_i([0,1])=\emptyset$ for each $i\in\mathcal{I}$), so $\Delta$ has children $\Delta_1=[0,2/9]$, $\Delta_2=[2/9,1/3]$, $\Delta_3=[1/3,5/9]$, and $\Delta_4=[2/3,1]$.
    These net intervals are depicted in \cref{f:ex-net-intervals} along with their positions relative to the intervals $S_i([0,1])$ for $i=1,2,3$.
    For illustrative purposes, we also compute $\vs(\Delta_2)$.
    Note that $T_{\Delta_2}(x)=x/3+2/9$ so that $T(x):=T_{\Delta_2}^{-1}(x)=9x-2$, so we have
    \begin{align*}
        \vs(\Delta_2)&=\{T_{\Delta_2}^{-1}\circ g:g\in\mathcal{Y},g(K)\cap\Delta_2^\circ\neq\emptyset\}=\{T\circ S_i:i=1,2\}\\
                     &= \{x\mapsto T(x/3), x\mapsto T(x/3+2/9)\}=\{x\mapsto 3x-2,x\mapsto 3x\}.
    \end{align*}
    If, furthermore, we wanted to compute the children of $\Delta_2$ in $\mathcal{P}_2$, we would begin by computing
    \begin{align*}
        \mathcal{Y}(\Delta_2,\vs(\Delta_2)) &= \{T_{\Delta_2}\circ f\circ S_i:i\in\mathcal{I},f\in\vs(\Delta_2)\}\\
                                            &= \{S_i\circ S_j:i\in\{1,2\},j\in\mathcal{I}\}
    \end{align*}
    and then continuing as above.
\end{example}
\begin{figure}[ht]
    \def\lb{0.544}
\begin{tikzpicture}[
    xscale=14,
    netiv/.style={thick,shorten <= 1pt,shorten >= 1pt,<->}]
    \drawiv{0}{1/3}{0}
    \drawiv{2/9}{5/9}{0.5}
    \drawiv{2/3}{1}{0}
    \foreach \x/\lbl in {0/$0$,{2/9}/$\frac{2}{9}$,{1/3}/$\frac{1}{3}$,{5/9}/$\frac{5}{9}$,{2/3}/$\frac{2}{3}$,1/$1$} {
        \draw[thick,dotted] (\x,0.8) -- (\x,-1.1) node[below]{\lbl};
    }

    \draw[netiv] (0,-0.7) -- node[fill=white]{$\Delta_1$} (2/9,-0.7);
    \draw[netiv] (2/9,-0.7) -- node[fill=white]{$\Delta_2$} (1/3,-0.7);
    \draw[netiv] (1/3,-0.7) -- node[fill=white]{$\Delta_3$} (5/9,-0.7);
    \draw[netiv] (2/3,-0.7) -- node[fill=white]{$\Delta_4$} (1,-0.7);
\end{tikzpicture}
    \caption{Net intervals in $\mathcal{P}_1$ as described in \cref{e:gen-ifs}.}
    \label{f:ex-net-intervals}
\end{figure}
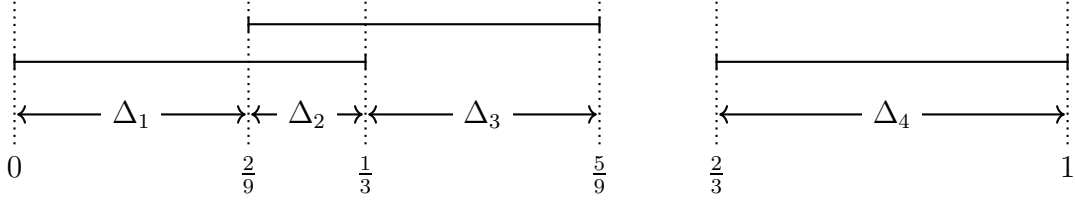

In order to avoid certain degenerate situations, we require two additional assumptions on the basic iteration rule $\Phi$.
\begin{definition}\label{d:iter}
    Let $\Phi$ be a basic iteration rule.
    We say that $\Phi$ is an \defn{iteration rule} if
    \begin{enumerate}[nl,r]
        \item $\displaystyle\lim_{n\to\infty}\max_{(\Delta,v)\in\mathcal{V}_n}\max_{\{x\mapsto rx+a\}\in v} r \diam(\Delta)=0$, and
        \item if $(\Delta,v)\in\mathcal{V}$ and $f_1\neq f_2\in v$, then for any $\sigma\in\mathcal{I}^*$, we have $f_1\circ S_\sigma\neq f_2$.
    \end{enumerate}
\end{definition}
Note that if $f\in\vs(\Delta)$ is any neighbour, then $f(x)=ax+b$ for some $a\geq 1$.
Thus, (i) implies that the diameters of net intervals also tend uniformly to zero.
In fact, since $K\subseteq\bigcup_{\Delta\in\mathcal{P}_n}\Delta$ and the endpoints of each $\Delta$ are elements of $K$,
\begin{equation*}
    \bigcap_{n=0}^\infty\bigcup_{\Delta\in\mathcal{P}_n}\Delta=K.
\end{equation*}

We now have the following basic lemma.
\begin{lemma}\label{l:dk-max}
    Fix some pair $(\Delta,v)\in\mathcal{V}_n$.
    Then for each $f\in v$, $f(K)\cap(0,1)\neq\emptyset$, and
    \begin{equation}\label{e:dk}
        \Delta^\circ\cap K=\Delta^\circ\cap \bigcup_{f\in v}T_\Delta\circ f(K).
    \end{equation}
    Moreover, if $\sigma\in\mathcal{I}^*$ is any word satisfying $S_\sigma(K)\cap\Delta^\circ\neq\emptyset$, there is a unique word $\tau$ such that $T_{\Delta}^{-1}\circ S_\tau\in v$ and either $\tau$ is a prefix of $\sigma$ or $\sigma$ is a prefix of $\tau$.
\end{lemma}
\begin{proof}
    We prove \cref{e:dk} by induction on $n$.
    The case $n=0$ is immediate, so now let $(\Delta,v)\in\mathcal{V}_n$ have parent $(\Delta',v')\in\mathcal{V}_{n-1}$.
    Write $v'=(f_1,\ldots,f_m)$ and $\Phi(v')=(\mathcal{C}_1,\ldots,\mathcal{C}_m)$.
    Note that by definition of $\Phi$, for each $i$,
    \begin{equation*}
        K=\bigcup_{\sigma\in\mathcal{C}_i}S_\sigma(K).
    \end{equation*}
    Thus by the inductive hypothesis,
    \begin{equation*}
        (\Delta')^\circ\cap K = (\Delta')^\circ\cap\bigcup_{i=1}^m\bigcup_{\sigma\in\mathcal{C}_i}T_{\Delta'}\circ f_i\circ S_\sigma(K).
    \end{equation*}
    But by construction, if $T_{\Delta'}\circ f_i\circ S_\sigma(K)\cap\Delta^\circ\neq\emptyset$, then $T_\Delta^{-1}\circ T_{\Delta'}\circ f_i\circ S_\sigma\in v$, so the result follows.

    For the second part, the existence of the word $\tau$ follows by construction, and uniqueness follows from (ii) in \cref{d:iter}.
\end{proof}
We now have the following fundamental result, the proof of which is similar to \cite[Thm. 2.8]{rut2021}.
We include the main details and leave additional verification to the reader.
\begin{proposition}\label{p:ttype}
    Let $(f_i)_{i\in\mathcal{I}}$ be an IFS with basic iteration rule $\Phi$.
    Then for any $\Delta^{(1)}\in\mathcal{P}_{n_1}$ with children $(\Delta_1^{(1)},\ldots,\Delta_{m_1}^{(1)})$, if $\Delta^{(2)}\in\mathcal{P}_{n_2}$ where $\vs(\Delta^{(1)})=\vs(\Delta^{(2)})$ has children $(\Delta_1^{(2)},\ldots,\Delta_{m_2}^{(2)})$, then $m_1=m_2$ and for each $1\leq i\leq m$,
    \begin{enumerate}[nl,r]
        \item $\vs(\Delta_i^{(1)})=\vs(\Delta_i^{(2)})$,
        \item $q(\Delta,\Delta_i)=q(\Delta^{(1)},\Delta_i^{(2)})$, and
        \item $\diam(\Delta)/\diam(\Delta_i)=\diam(\Delta')/\diam(\Delta_i')$.
    \end{enumerate}
\end{proposition}
\begin{proof}
    For each $j=1,2$, let $\mathcal{Y}_j,Y_j$ correspond to the set $\Delta^{(j)}$ as in \cref{e:H-def} in the definition of children.
    Then with $\psi=T_{\Delta^{(1)}}\circ T_{\Delta^{(2)}}^{-1}$, we have $\mathcal{Y}_1=\{\psi\circ g:g\in\mathcal{Y}_2\}$ and $\psi(Y_2)=Y_1$.
    Thus with the elements of $Y_1$ in order as $y_1<\cdots<y_{k+1}$, the elements of $Y_2$ are given in order as $\psi(y_1)<\cdots<\psi(y_{k+1})$.

    Now since
    \begin{equation*}
        (\Delta^{(j)})^\circ\cap K = (\Delta^{(j)})^\circ\cap\bigcup_{f\in v} T_{\Delta^{(j)}}\circ f(K)
    \end{equation*}
    from \cref{l:dk-max}, it follows that $\psi:\Delta^{(2)}\cap K\to\Delta^{(1)}\cap K$ is a surjection so that $(y_i,y_{i+1})\cap K\neq\emptyset$ if and only if $(\psi(y_i),\psi(y_{i+1}))\cap K\neq\emptyset$.
    Thus $m_1=m_2$.

    From here, (i), (ii), and (iii) follow by direct computation.
\end{proof}

\begin{remark}\label{r:pr-ch}
    Sometimes, it can hold that $(\Delta,v)$ has a unique child $(\Delta',v')$ where $\Delta=\Delta'$.
    For technical purposes, in order to avoid this degenerate situation, it is convenient to redefine the iteration rule $\Phi$ as follows.
    Write $v=(f_1,\ldots,f_m)$, $g=(g_1,\ldots,g_k)$, and suppose $\Phi(v)=(\mathcal{C}_1,\ldots,\mathcal{C}_m)$ and $\Phi(v')=(\mathcal{C}_1',\dots,\mathcal{C}_k')$.
    Since $\Delta=\Delta'$, for each $f_i$ and $\sigma\in\mathcal{C}_i$, either $f_i\circ S_\sigma(K)\cap\Delta^\circ=\emptyset$ or $f_i\circ S_\sigma = g_j$ for some $j$.
    Now for each $1\leq i\leq m$ and $\sigma\in\mathcal{C}_i$, set
    \begin{equation*}
        \mathcal{U}_{i,\sigma}=\begin{cases}
            \{\emptyset\} &: f_i\circ S_\sigma(K)\cap\Delta^\circ=\emptyset\\
            \mathcal{C}_j &: f_i\circ S_\sigma=g_j
        \end{cases}
    \end{equation*}
    and define
    \begin{equation*}
        \widetilde{\mathcal{C}}_i = \bigcup_{\sigma\in\mathcal{C}_i}\{\sigma\tau:\tau\in\mathcal{U}_{i,\sigma}\}.
    \end{equation*}
    Then define $\widetilde{\Phi}$ by $\widetilde{\Phi}(v)=(\widetilde{\mathcal{C}}_1,\ldots,\widetilde{\mathcal{C}}_m)$, and $\widetilde{\Phi}=\Phi$ otherwise.
    It is straightforward to verify that $\widetilde{\Phi}$ is an iteration rule, and with this definition, the children of $(\Delta,v)$ with respect to $\widetilde{\Phi}$ are precisely the children of $(\Delta',v')$ with respect to $\Phi$.

    Note that an infinite sequence of children where all the net intervals are identical is disallowed by (i) in \cref{d:iter}.
    Repeating this construction, we may thus assume that each net interval $\Delta$ has at least two distinct children, and for any $\Delta\in\mathcal{P}$, there is a unique $n$ such that $\Delta\in\mathcal{P}_n$.
\end{remark}

We conclude with two examples explaining the relationship with our general net interval construction and earlier net interval constructions.
In practice, all iteration rules the author has used fall into these two classes.
\begin{example}\label{ex:uniform-transition}
    As discussed, the rule
    \begin{equation*}
        \Phi(f_1,\ldots,f_m)=(\mathcal{I},\ldots,\mathcal{I})
    \end{equation*}
    always defines an iteration rule.
    Here, the neighbour sets and net intervals can be described in a slightly different way.
    Enumerate the points $\{S_\sigma(0),S_\sigma(1):\sigma\in\mathcal{I}^n\}$ in increasing order as $0=y_0<y_1<\cdots<y_{s(n)}= 1$.
    We claim that
    \begin{equation}\label{e:pn-formula}
        \mathcal{P}_n=\{[y_i,y_{i+1}]:(y_i,y_{i+1})\cap K\neq\emptyset,0\leq i < s(n)\}
    \end{equation}
    and for a net interval $\Delta\in\mathcal{P}_n$,
    \begin{equation}\label{e:vs-formula}
        \vs(\Delta)=\{T_\Delta^{-1}\circ S_\sigma:\sigma\in\mathcal{I}^n,S_\sigma(K)\cap\Delta^\circ\neq\emptyset\}.
    \end{equation}
    Let us prove that this holds by induction.
    When $n=0$, \cref{e:pn-formula} and \cref{e:vs-formula} both hold trivially.
    Thus suppose $n\in\N$ is arbitrary and $\Delta=[a,b]\in\mathcal{P}_n$.
    From the definition in \cref{e:H-def} along with \cref{e:pn-formula} and \cref{e:vs-formula}, we observe that
    \begin{align*}
        \mathcal{Y}&=\bigcup_{f\in\vs(\Delta)}\{T_\Delta\circ f\circ S_i:i\in\mathcal{I}\}\\
                   &= \bigcup_{\{\sigma\in\mathcal{I}^n:S_\sigma(K)\cap\Delta^\circ\neq\emptyset\}}\{S_{\sigma i}:i\in\mathcal{I}\}\\
                   &= \{S_{\sigma i}:\sigma\in\mathcal{I}^n,S_\sigma(K)\cap\Delta^\circ\neq\emptyset,i\in\mathcal{I}\}\\
                   &=\{S_\tau:\tau\in\mathcal{I}^{n+1},S_\tau(K)\cap\Delta^\circ\neq\emptyset\}.
    \end{align*}
    and therefore
    \begin{equation*}
        Y = \{a,b\}\cup\{S_\tau(z):z\in\{0,1\},\tau\in\mathcal{I}^{n+1},S_\tau(z)\in\Delta\}.
    \end{equation*}
    Thus the children of $\Delta$ in $\mathcal{P}_{n+1}$ are precisely of the form given in \cref{e:pn-formula}, and if $\Delta_i$ is any child of $\Delta$, from the definition \cref{e:ch-nb-def} it has neighbour set 
    \begin{equation*}
        \vs(\Delta_i)=\{T_{\Delta_i}^{-1}\circ g: g\in\mathcal{Y}, g(K)\cap\Delta_i^\circ\neq\emptyset\}=\{T_{\Delta_i}^{-1}\circ S_\tau:\tau\in\mathcal{I}^{n+1},S_\tau(K)\cap\Delta_i^\circ\neq\emptyset\}
    \end{equation*}
    and \cref{e:vs-formula} holds for $\Delta_i$.
    Since any net interval $\Delta$ satisfies $\Delta^\circ\cap K\neq\emptyset$, every net interval in $\mathcal{P}_{n+1}$ must be given in this way.
    Thus \cref{e:pn-formula} and \cref{e:vs-formula} hold for $\mathcal{P}_{n+1}$.

    If each $S_i(x)=\lambda x+d_i$ for some fixed $0<\lambda<1$, the net intervals are the same as those considered by Feng \cite{fen2003}, and our definition of a neighbour set is closely related to the characteristic vector defined in that paper.
    See \cite[Rem. 2.2]{rut2021} for more details on this relationship.
\end{example}
\begin{example}\label{ex:weighted-transition}
    Given a tuple of similarities $(f_1,\ldots,f_m)$ with each $f_i(x)=a_ix+b_i$, let $a=\max\{|a_i|:1\leq i\leq m\}$ and define
    \begin{equation*}
        \mathcal{C}_i=
        \begin{cases}
            \mathcal{I} &: |a_i|=a\\
            \{\emptyset\} &: |a_i|<a
        \end{cases}
    \end{equation*}
    where we recall that $\emptyset$ denotes the empty word.
    Then the map $\Phi(f_1,\ldots,f_m)=(\mathcal{C}_1,\ldots,\mathcal{C}_m)$ defines an iteration rule which gives the net intervals and neighbour sets as defined in \cite[Sec. 2.2]{rut2021}.
    Indeed, with this construction, the rule defining children of $\Delta$ described above coincides exactly with the notion of the child of a net interval from \cite[Sec. 2.3]{rut2021}.
\end{example}

\subsection{The transition graph and symbolic representations}\label{ss:tg-sr}
We begin by introducing some useful terminology from graph theory.
By a rooted graph $\mathcal{G}$, we mean a directed graph (possibly with loops and multiple edges) consisting of a set $V(\mathcal{G})$ of vertices with a distinguished vertex $\vroot\in V(\mathcal{G})$, and a set $E(\mathcal{G})$ of edges.
By an \defn{edge} $e$, we mean a triple $e=(v_1,v_2,q)$ where $v_1\in V(\mathcal{G})$ is the \defn{source}, $v_2\in V(\mathcal{G})$ is the \defn{target}, and $q$ is the \defn{label} of the edge $e$.
The point of the label is to distinguish multiple edges, but it is safe to imagine that the graph does not have multiple edges.

A \defn{finite path} in $\mathcal{G}$ is a sequence $\eta=(e_1,\ldots,e_n)$ of edges in $\mathcal{G}$ such that the target of each $e_i$ is the source of $e_{i+1}$.
We say that the \defn{length} of $\eta$ is $n$, and denote this by $|\eta|$.
A finite path is a \defn{cycle} if, in addition, the source of $e_1$ is the target of $e_n$.
A (one-way) infinite path is a sequence $(e_i)_{i=1}^\infty$ where the target of each $e_i$ is the source of $e_{i+1}$ for $i\in\N$.
Given paths $\eta_1=(e_1,\ldots,e_n)$ and $\eta_2=(e_{n+1},\ldots,e_{n+m})$, if the target of $e_n$ is the source of $e_{n+1}$, the \defn{concatenation} $\eta_1\eta_2$ is the path $(e_1,\ldots,e_{n+m})$.
When it is convenient, we will abuse notation and treat edges as paths of length $1$.

We say that a (finite or infinite) path is \defn{rooted} if it begins at the root vertex $\vroot$, and we denote by $\Omega^\infty$ (resp. $\Omega^*$) the set of all infinite (resp. finite) rooted paths.
For any $n\in\N\cup\{0\}$, $\Omega^n$ denotes the set of all rooted paths of length $n$.
We say that $\eta_1$ is a \defn{prefix} of $\eta$ in $\Omega^*$ (resp. $\Omega^\infty$) if $\eta=\eta_1\eta'$ for some finite (resp. infinite) path $\eta'$.
Given a path $\gamma=(e_i)_{i=1}^\infty\in\Omega^\infty$, we denote the unique prefix of $\gamma$ in $\Omega^n$ by $\gamma|n=(e_1,\ldots,e_n)$.

We now define the main object in consideration in this document.
Fix a WIFS $(S_i,p_i)_{i\in\mathcal{I}}$ along with an iteration rule $\Phi$.
Then the \defn{transition graph} $\mathcal{G}=\mathcal{G}\bigl((S_i,p_i)_{i\in\mathcal{I}},\Phi\bigr)$, is a rooted graph defined as follows.
The vertex set of $\mathcal{G}$ is the set of all neighbour sets $\{v:{(\Delta,v)\in\mathcal{V}}\}$ with root vertex $\vroot=\{\id\}$ corresponding to the net interval $[0,1]$.
Now whenever $(\Delta,v)$ has child $(\Delta',v')$, we introduce an edge $(v,v',q(\Delta,\Delta'))$, where the position index $q(\Delta,\Delta')$ is the label distinguishing multiple edges between the vertices $v$ and $v'$.
This construction is well-defined by \cref{p:ttype}.
Given a vertex $v\in V(\mathcal{G})$, which is a neighbour set $v=(f_1,\ldots,f_m)$, we write $d(v)=m$.
For the remainder of this document, the set $\Omega^\infty$ (and $\Omega^*$, $\Omega^n$) will always be associated with the transition graph $\mathcal{G}$

Now given a path $\eta=(e_1,\ldots,e_n)\in\Omega^n$, there is a unique sequence of net intervals $(\Delta_i,v_i)_{i=0}^n$ with $\Delta_0=[0,1]$ where each $(\Delta_i,v_i)\in\mathcal{V}_i$, $\Delta_{i+1}$ is the child of $\Delta_i$, and
\begin{equation*}
    e_i=\bigl(v_i,v_{i+1},q(\Delta_i,\Delta_{i+1})\bigr).
\end{equation*}
This follows directly by construction of the edge set of the transition graph $\mathcal{G}$.
Since net intervals either coincide or overlap only on endpoints, this path is uniquely determined by the last net interval in the sequence.
Thus we may define a map $\pi:\Omega^n\to\mathcal{P}_n$ by $\pi(\eta)=\Delta_n$.
\begin{lemma}\label{l:netiv-sr}
    The map $\pi:\Omega^n\to\mathcal{P}_n$ is a well-defined bijection for each $n\in\N\cup\{0\}$.
\end{lemma}
Given a net interval $\Delta\in\mathcal{P}_n$, the \defn{symbolic representation} of $\Delta$ is the path $\pi^{-1}(\Delta)\in\Omega^n$.

Now given an infinite path $\gamma=(e_i)_{i=1}^\infty\in\Omega^\infty$, there corresponds a sequence of net intervals $(\Delta_i)_{i=0}^\infty$ with $\Delta_i\in\mathcal{P}_i$ where $\Delta_0=[0,1]$ and $\Delta_{i+1}$ is a the child of $\Delta_i$ corresponding to the edge $e_{i+1}$.
Of course, $\Delta_n=\pi(\gamma|n)$.
Since $\lim_{i\to\infty}\diam(\Delta_i)=0$, there exists a unique point in $K$, which we call $\pi(\gamma)$, satisfying
\begin{equation*}
    \{\pi(\gamma)\}=\bigcap_{i=1}^\infty\Delta_i.
\end{equation*}
In analogy to the net interval case, we refer to a path $\gamma\in\pi^{-1}(x)$ as a \defn{symbolic representation} of $x$.
It is clear by construction of net intervals that the map $\pi:\Omega^\infty\to K$ is surjective.
Note that $\pi$ need not be injective, but if $x\in K$ has fibre $\pi^{-1}(x)$ with cardinality greater than 1, then $x$ must be an endpoint of some net interval $\Delta$.
In this situation, $\pi^{-1}(x)$ contains two paths.
Since there are only countably many net intervals, $\pi$ is injective on all but at most countably many paths.
We say that $x$ is an \defn{interior point} of $K$ if $\pi^{-1}(x)$ has cardinality 1.

\subsection{Edge weights and transition matrices}
For our purposes, perhaps the two most important attributes of a net interval $\Delta$ are its diameter $\diam(\Delta)$ and measure $\mu(\Delta)$.
Moreover, recall that we have a correspondence $\pi:\Omega^n\to\mathcal{P}_n$ taking rooted paths in the transition graph to net intervals in $\R$.
Through this correspondence, we get the corresponding ``symbolic diameter'' $\diam\circ\pi$ and ``symbolic measure'' $\mu\circ\pi$ defined on the set of rooted finite paths $\Omega^*$.
In this section, we will define two natural objects which takes values on $E(\mathcal{G})$ which will allow us to encode the functions $\diam\circ\pi$ and $\mu\circ\pi$ respectively in a way intrinsic to the transition graph.

We first describe $\diam\circ\pi$ as a product of weights on edges.
\begin{definition}\label{d:edge-weight}
    The \defn{edge weight function} for $\mathcal{G}$ is the map $W:E(\mathcal{G})\to(0,1)$ such that if the edge $e$ corresponds to the child $\Delta'\subseteq\Delta$, then $W(e)=\diam(\Delta')/\diam(\Delta)$.
    Given a path $\eta=(e_1,\ldots,e_n)$, we write $W(\eta)=W(e_1)\cdots W(e_n)$.
\end{definition}
Note that the edge weight is well-defined by \cref{p:ttype} (see also \cref{r:pr-ch}).
Of course, when $\Delta\in\mathcal{P}_n$ has symbolic representation $\eta=(e_i)_{i=1}^n$,
\begin{equation*}
    \diam(\Delta)=\diam(\pi(\eta))=W(e_1)\cdots W(e_n)=W(\eta),
\end{equation*}
so that $\diam\circ\pi=W$.

We now describe $\mu\circ\pi$ as the norm of products of matrices associated with edges.
Let $e\in E(\mathcal{G})$ be an edge corresponding to $(\Delta_1,(f_1,\ldots,f_m))$ the parent of $(\Delta_2,(g_1,\ldots,g_n))$, and let $\Phi(f_1,\ldots,f_m)=(\mathcal{C}_1,\ldots,\mathcal{C}_m)$ where $\Phi$ is the iteration rule.
For each $(i,j)\in\{1,\ldots,m\}\times\{1,\ldots,n\}$, set
\begin{equation*}
    \mathcal{E}_{i,j}=\{\omega\in\mathcal{C}_i:T_{\Delta_1}\circ f_i\circ S_\omega=T_{\Delta_2}\circ g_j\}.
\end{equation*}
Then the \defn{transition matrix} is the $n\times m$ matrix $T(e)$ given by
\begin{equation}\label{e:tr-mat}
    T(e)_{i,j}=\frac{f_i\mu((0,1))}{g_j\mu((0,1))}\cdot\sum_{\omega\in\mathcal{E}_{i,j}}p_\omega
\end{equation}
where we recall that $f\mu$ is the pushforward of $\mu$ by the function $f$.
It is clear that the transition matrix depends only on the edge $e$.
We note the following important observations.
\begin{lemma}\label{l:pos-col-row}
    If $e\in E(\mathcal{G})$ is any edge, then $T(e)$ has a positive entry in each column.
    Moreover, for any $v\in V(\mathcal{G})$ and $1\leq i\leq d(v)$, there is an edge $e$ with source $v$ such that $T(e)$ has a positive entry in row $i$.
\end{lemma}
\begin{proof}
    Since each neighbour $g_j$ is of the form $T_{\Delta_2}\circ T_{\Delta_1}^{-1}\circ f_i\circ S_\sigma$ for some $i$ and (possibly empty) word $\sigma$, each column of $T(e)$ has a non-negative entry.

    To see the second part, let $\Delta$ be a net interval with $\vs(\Delta)=v$.
    Since $T_\Delta((0,1))\cap f_i(K)\neq\emptyset$ for each $1\leq i\leq m$, there is some child $\Delta'\subseteq\Delta$ such that $T_{\Delta'}((0,1))\cap f_i(K)\neq\emptyset$.
    Then if $e$ is the edge corresponding to $\Delta'\subseteq\Delta$, $T(e)$ has a positive entry in row $i$ by \cref{l:dk-max} and the definition of the transition matrix.
\end{proof}
Given a path $\eta=(e_1,\ldots,e_n)$, we write $T(\eta)=T(e_1)\cdots T(e_n)$.
We write $\norm{T(\eta)}=\sum_{i,j}T(\eta)_{i,j}$ to denote the matrix $1$-norm.

Now fix a pair $(\Delta,(f_1,\ldots,f_m))\in\mathcal{N}_n$, let $\muv(\Delta)=(q_1,\ldots,q_m)$ where
\begin{equation}\label{e:qi-formula}
    q_i=f_i\mu((0,1))\sum_{\substack{\sigma\in\mathcal{I}^*\\S_\sigma=T_\Delta\circ f_i}}p_\sigma.
\end{equation}
Using the self-similarity relation of $\mu$, the definition of the iteration rule $\Phi$, and condition (ii) in \cref{d:iter}, one can verify that
\begin{equation*}
    \mu(\Delta)=\norm{\muv(\Delta)}.
\end{equation*}
Now a similar argument as the proof of \cite[Thm. 2.12]{rut2021} gives the following result:
\begin{proposition}\label{p:mat-mu}
    Let $(S_i)_{i\in\mathcal{I}}$ have associated self-similar measure $\mu$ and fix an iteration rule $\Phi$.
    Then $\muv\circ\pi=T$ for every $n\in\N$, so if $\eta\in\Omega^n$,
    \begin{equation*}
        \mu(\pi(\eta))=\norm{T(\eta)}.
    \end{equation*}
\end{proposition}
\begin{proof}
    Let $\Delta=\pi(\eta)$ and let $\eta$ end at the vertex $v$.
    Given $(\Delta,v)\in\mathcal{N}_n$, there exists a unique sequence $(\Delta_i,v_i)_{i=0}^n$ where $\Delta_i\in\mathcal{N}_i$, $\Delta_{i+1}$ is a child of $\Delta_i$, and $\Delta_n=\Delta$.
    Now for $f\in v$, let $T_\Delta\circ f=S_\sigma$ for some $\sigma\in\mathcal{I}^*$.
    Then one can write $\sigma=\sigma_1\ldots\sigma_n$ if and only if $\sigma_i\in\mathcal{C}_{j(i)}$ where $j(i)$ satisfies $T_{\Delta_i}^{-1}\circ f^{(i)}_{j(i)}=S_{\sigma_i}$ with $v_i=(f^{(i)}_1,\ldots,f^{(i)}_{m_i})$.
    Thus the entry of $T(\eta)$ corresponding to the index $f$ is the sum of $p_\sigma$ over all $\sigma$ satisfying $T_\Delta\circ f=S_\sigma$.
\end{proof}
We observe that the transition matrices are analgous to the role of the probabilities $(p_i)_{i\in\mathcal{I}}$ described in \cref{e:meas-p}.

We conclude by mentioning the following straightforward but important property of transition matrices.
\begin{lemma}\label{l:left-prod}
    If $\eta=\eta_1\eta_2\in\Omega^*$ with $\eta_1\in\Omega^n$, then $\norm{T(\eta)}\asymp_n\norm{T(\eta_2)}$.
\end{lemma}
\begin{proof}
    By \cref{l:pos-col-row}, every transition matrix has a non-zero entry in each column, so a straightforward calculation shows that there exists some constant $a=a(\eta_1)$ such that $\norm{T(\eta_1\eta_2)}\geq a(\eta_1)\norm{T(\eta_2)}$.
    On the other hand, $\norm{T(\eta_1\eta_2)}\leq\norm{T(\eta_1)}\norm{T(\eta_2)}$ by submultiplicativity of the matrix norm.
    But there are only finitely many paths in $\Omega^n$, giving the result.
\end{proof}

\subsection{The finite neighbour condition}\label{ss:fnc}
Throughout this section, we have made no assumptions about the IFS $(S_i)_{i\in\mathcal{I}}$ or the transition graph $\mathcal{G}$.
We now introduce the main restriction of this paper.

The finite neighbour condition was introduced in \cite{hhrtoappear} as a variation of the generalized finite type condition introduced by Lau and Ngai \cite{ln2007}.
In general, such ``finite type'' conditions attempt to capture the idea that an IFS only has finitely many possible overlaps.
It is known that the finite neighbour condition is equivalent to the generalized finite type condition holding with respect to the interval $(0,1)$ \cite{hhrtoappear}.
We introduce the following definition, which is a natural generalization of the usual finite neighbour condition with respect to our more general transition graph construction.
\begin{definition}\label{d:pfnc}
    We say that the IFS $(S_i)_{i\in\mathcal{I}}$ satisfies the \defn{finite neighbour condition with respect to the iteration rule $\Phi$}, or the $\Phi$-FNC for short, if the corresponding transition graph is a finite graph.
\end{definition}
Closely related to this finite neighbour condition is the \defn{weak separation condition}.
This separation condition is satisfied if
\begin{equation}\label{e:wsc}
    \sup_{x\in K,r>0}\#\{S_\sigma:r\cdot r_{\min}<|r_\sigma|\leq r,S_\sigma(K)\cap(x-r,x+r)\neq\emptyset\}<\infty.
\end{equation}
The weak separation condition was introduced by Lau and Ngai in \cite{ln1999}; this definition is not the original but equivalent by \cite[Thm. 1]{zer1996}.
Standard arguments show that any IFS satisfying the $\Phi$-FNC necessarily satisfies the weak separation condition (see, for example, \cite{hhrtoappear,ln2007}).
Moreover, when $K$ is a convex set, the weak separation condition implies that the $\Phi$-FNC holds with respect to the iteration rule $\Phi$ from \cref{ex:weighted-transition} \cite{hhrtoappear}.

\section{Loop classes, irreducibility, and decomposability}
In this section, we introduce the notion of a loop class of the transition graph $\mathcal{G}$, and other related definitions.
These definitions are required to state the main technical assumptions (irreducibility and decomposability) which underpin the main results presented later in this paper.
We also discuss certain general situations in which the technical assumptions are satisfied.

For the remainder of the paper (including this section), we will assume that $(S_i,p_i)_{i\in\mathcal{I}}$ satisfies the $\Phi$-FNC with finite transition graph $\mathcal{G}$.
Note that many concepts in this section hold more generally for an arbitrary transition graph $\mathcal{G}$, but we do not distinguish this during the subsequent discussions for sake of simplicity.

\subsection{Loop classes}\label{ss:irred}
Let $G$ be a directed multigraph.
Recall that a graph $H$ is an \defn{induced subgraph} of $G$ if $H$ is the graph consisting of the vertices $V(H)$ and any edge $e\in E(G)$ such that $e$ connects two vertices in $H$.
\begin{definition}\label{d:loop-class}
    Let $H$ be an induced subgraph of $G$.
    We say that $H$ is \defn{strongly connected} if for any vertices $v,w\in V(H)$, there is a directed path from $v$ to $w$.
    Then $H$ is a \defn{loop class} (in $G$) if it is strongly connected, contains at least one edge, and is maximal with these properties.

    Now if $H$ is a loop class, we say that $H$ is \defn{simple} if each vertex in $H$ has exactly one outgoing edge (in $H$).
    We say that $H$ is \defn{essential} if for any $v\in V(H)$ and $w\in V(G)$, if there is a directed path from $v$ to $w$, then $w\in V(H)$ as well.
\end{definition}
Of course, any essential loop class is necessarily not simple.
Note that distinct loop classes have disjoint vertex and edge sets, but there may be vertices which do not belong to any loop class.
\begin{remark}
    Previous authors (e.g. \cite{hhn2018}) distinguished between loop classes and maximal loop classes.
    In this document, our loop classes are always maximal.
\end{remark}
\begin{example}
    In \cref{f:gen-tr-graph}, the loop classes are given by $\{\mathcal{L}_1,\mathcal{L}_2,\mathcal{L}_3,\mathcal{L}_4\}$.
    The loop classes $\mathcal{L}_1$ and $\mathcal{L}_2$ are simple, while $\mathcal{L}_3$ is not; and $\mathcal{L}_4$, being an essential loop class, is not simple.
\end{example}

Since the transition graph $\mathcal{G}$ is a finite graph, there are only finitely many loop classes.
Given any path $\gamma=(e_i)_{i=1}^\infty\in\Omega^\infty$, there is a unique loop class $\mathcal{L}$ such that there is some $N$ such that for all $k\geq N$, $e_k$ is an edge in $\mathcal{L}$.
We say that $\gamma$ is \defn{eventually in $\mathcal{L}$} and denote the set of all such $\gamma$ by $\Omega^\infty_{\mathcal{L}}$.

We may now set
\begin{align*}
    K_{\mathcal{L}} &= \{x\in K:\pi^{-1}(x)\cap\Omega^\infty_{\mathcal{L}}\neq\emptyset\} & \kint_{\mathcal{L}} &= \{x\in K:\pi^{-1}(x)\subseteq\Omega^\infty_{\mathcal{L}}\}.
\end{align*}
Of course, for each $x\in K$ there is at least one loop class $\mathcal{L}$ such that $x\in K_{\mathcal{L}}$, and at most two such sets.
Note that $\kint_{\mathcal{L}}$ is the topological interior of $K_{\mathcal{L}}$ (relative to $K$) if and only if $\mathcal{L}$ is an essential loop class.
If $x\in K$ is an interior point, then $x\in\kint_{\mathcal{L}}$ for a unique loop class $\mathcal{L}$.

Our analysis is focused on two technical assumptions, which we call irreducibility and decomposability.
We discuss these assumptions in the following two sections.
\subsection{Irreducibility}\label{sss:irreducibility}
Irreducibility can be loosely interpreted as a type of ``measure connectivity'' within the loop class.
\begin{definition}\label{d:irred}
    Let $\mathcal{L}$ be a loop class.
    We say that $\mathcal{L}$ is \defn{irreducible} if there exists a finite set of paths $\mathcal{H}$ such that for any paths $\eta_1,\eta_2$ in $\mathcal{L}$, there is some $\gamma\in\mathcal{H}$ such that $\eta_1\gamma\eta_2$ is a path and
    \begin{equation*}
        \norm{T(\eta_1\gamma\eta_2)}\asymp\norm{T(\eta_1)}\norm{T(\eta_2)}.
    \end{equation*}
    We say that the transition graph $\mathcal{G}$ is \defn{irreducible} if every loop class is irreducible.
\end{definition}
Since $\mathcal{L}$ is a finite graph, by submultiplicativity of the matrix norm, one can always guarantee that
\begin{equation*}
    \norm{T(\eta_1\gamma\eta_2)}\preccurlyeq\norm{T(\eta_1)}\norm{T(\eta_2)}.
\end{equation*}
On the other hand, establishing the lower inequality is more challenging.
This notion of irreducibility is motivated by various hypotheses studied by past authors \cite{fen2009,hr2021}.
We are not aware of any loop class of any IFS satisfying the finite neighbour condition that does not satisfy this irreducibility hypothesis.

In the following lemmas, we observe that this technical hypothesis is satisfied in a number of general cases.

Enumerate $V(\mathcal{L})=\{v_1,\ldots,v_k\}$.
For each $1\leq i,j\leq k$, let
\begin{align*}
    \mathcal{A}_{i,j}&=\{e\in E(G):e\text{ is an edge from }v_i\text{ to }v_j\} & M_{i,j}=\sum_{e\in\mathcal{A}_{i,j}}T(e)
\end{align*}
and define the block matrix
\begin{equation*}
    M=M(\mathcal{L}):=\begin{pmatrix}
        M_{1,1}&\cdots&M_{1,k}\\
        \vdots&\ddots&\vdots\\
        M_{k,1}&\cdots&M_{k,k}
    \end{pmatrix}
\end{equation*}
The following proof is straightforward and included in, say, \cite{hr2021}.
Recall the matrix $M$ is irreducible if for each $i,j$, there exists some $n=n(i,j)$ such that $(M^n)_{i,j}>0$.
\begin{lemma}\label{l:irred}
    Suppose the matrix $M$ is irreducible.
    Then $\mathcal{L}$ is an irreducible loop class.
\end{lemma}
\begin{proof}
    It suffices to show for any vertices $v,w\in V(\mathcal{L})$, $1\leq i\leq d(v)$, and $1\leq j\leq d(w)$, there exists some path $\gamma$ from $v$ to $w$ such that $T(\gamma)_{i,j}>0$.

    Let $\mathcal{H}$ be the finite set of paths in \cref{d:irred} and let $A$ be the smallest strictly positive entry for any $\eta\in\mathcal{H}$ and let $d_{\max}=\max\{d(v):v\in V(\mathcal{L})\}$ (which is also the maximum number of rows or columns of any transition matrix $T(e)$ where $e\in E(\mathcal{L})$).
    Let $\eta_1,\eta_2$ be any finite paths in $\mathcal{L}$.
    By the pidgeonhole principle, there exists some $k,i,j,\ell$ such that $T(\eta_1)_{k,i}\geq d_{\max}^{-2}\norm{T(\eta_1)}$ and $T(\eta_2)_{j,\ell}\geq d_{\max}^{-2}\norm{T(\eta_2)}$.
    Let $\phi\in\mathcal{H}$ be a path from the target vertex of $\eta_1$ to the source vertex of $\eta_2$ such that $T(\phi)_{i,j}\geq A>0$.
    Then
    \begin{equation*}
        \norm{T(\eta_1\phi\eta_2)}\geq T(\eta_1)_{k,i}T(\phi)_{i,j}T(\eta_2)_{j,\ell}\geq Ad_{\max}^{-4}\norm{T(\eta_1)}\norm{T(\eta_2)}.
    \end{equation*}
    The upper bound follows since
    \begin{equation*}
        \norm{T(\eta_1\phi\eta_2)}\leq\norm{T(\eta_1)}\norm{T(\eta_2)}\max\{\norm{T(\phi)}:\phi\in\mathcal{H}\}
    \end{equation*}
    by submultiplicativity of the norm.
\end{proof}

We next observe that an essential loop class is always irreducible.
\begin{lemma}\label{l:ess-irred}
    Let $\mathcal{L}$ be an essential loop class of $\mathcal{G}$.
    Then $\mathcal{L}$ is irreducible.
\end{lemma}
\begin{proof}
    In fact, we will show for any $v,w\in V(\mathcal{L})$ and $1\leq i\leq d(v)$, there exists some path $\gamma$ from $v$ to $w$ such that row $i$ of $T(\gamma)$ is strictly positive.
    The required result will then follow by \cref{l:irred}.

    Let $\Delta$ be any net interval with $\vs(\Delta)=v$ and neighbour set $v=(f_1,\ldots,f_{d(v)})$.
    Since $T_\Delta\circ f_i=S_{\sigma_0}$ for some $\sigma_0\in\mathcal{I}^*$ with $S_{\sigma_0}(K)\cap\Delta^\circ\neq\emptyset$, there exists some word $\sigma$ with prefix $\sigma_0$ such that $S_{\sigma}\subseteq\Delta$.
    Let $U=(x-r,x+r)$ attain the supremum in \cref{e:wsc} with words $\tau_1,\ldots,\tau_\ell$ satisfying $r\cdot r_{\min}<|r_{\tau_k}|\leq r$ and $S_{\tau_k}(K)\cap U\neq\emptyset$.
    Observe that $S_\sigma(U)$ also attains the supremum in \cref{e:wsc} with words $\sigma\tau_1,\ldots,\sigma\tau_\ell$.
    By condition (i) in \cref{d:iter} and since $\mathcal{L}$ is an essential loop class, there exists some net interval $\Delta_1\subseteq S_\sigma(U)$ with $\vs(\Delta_1)=w$ such that if $g$ is any neighbour of $\Delta_1$, then the contraction ratio $T_{\Delta_1}\circ g$ is less than $|r_\sigma|r$.
    Let $\gamma$ be the path corresponding to $\Delta_1\subseteq \Delta$, which is necessarily a path from $v$ to $w$ in $\mathcal{L}$.

    It remains to show that row $i$ of $T(\gamma)$ is strictly positive.
    Let $g\in\vs(\Delta_1)$ be arbitrary and let $S_\omega=T_{\Delta_1}\circ g$; by choice of $\Delta_1$, we have $|r_\omega|\leq |r_\sigma|r$.
    Since $S_\omega(K)\cap\Delta^\circ\neq\emptyset$, we have $S_\omega(K)\cap S_\sigma(U)\neq\emptyset$.
    Let $\xi$ be the unique prefix of $\omega$ with minimal length satisfying $|r_\xi|\leq|r_\sigma|r$.
    In particular, $|r_\xi|>|r_\sigma|r\cdot r_{\min}$ and $S_\xi(K)\cap S_\sigma(U)\neq\emptyset$, forcing $S_\xi=S_{\sigma\tau_j}$ for some $1\leq j\leq\ell$ by maximality of $\ell$.
    Unpacking definitions, this means that there is some word $\phi$ such that
    \begin{equation*}
        T_{\Delta_1}\circ g = S_{\sigma_0}\circ S_\phi=T_\Delta\circ f_i\circ S_\phi.
    \end{equation*}
    In other words, the entry in $T(\gamma)$ corresponding to the neighbours $f_i$ of $v$ and $g$ of $w$ is strictly positive.
    Since $g$ was an arbitrary neighbour of $\Delta_1$, the result follows.
\end{proof}
\subsection{Decomposability}\label{sss:decomposable}
Unlike irreducibility which, up to a fixed constant multiple, states that one can join paths within a loop class without changing the norm of the corresponding transition matrix, decomposability states that for a given path passing through multiple loop classes, the norm is comparable to the norms of the components of the path within each loop class it passes through.

We begin by defining the notion of an \defn{initial path} and a \defn{transition path} in the transition graph $\mathcal{G}$ as follows.
Let $\mathcal{G}$ have loop classes $\mathcal{L}_1,\ldots,\mathcal{L}_m$ and root vertex $\vroot$.
Let $\psi=(e_1,\ldots,e_n)$ be a path in $\mathcal{G}$ connecting vertices $(v_0,v_1,\ldots,v_n)$.
We say a path $\psi$ is a \defn{transition path} if
\begin{enumerate}[nl]
    \item $v_0$ is a vertex in $V(\mathcal{L}_j)$ for some $j$,
    \item $v_n$ is a vertex in $V(\mathcal{L}_k)$ for some $k\neq j$, and
    \item each $v_i$ with $0<i<n$ is not a vertex in any loop class.
\end{enumerate}
Similarly, we say that $\psi$ is an \defn{initial path} if we replace condition (1) with
\begin{enumerate}[nl,r]
    \item[(1')] $v_0=\vroot$
\end{enumerate}
There are only finitely many initial paths and transition paths since they cannot repeat vertices.

By definition of the loop class, we can sort the loop classes $\mathcal{L}_1,\ldots,\mathcal{L}_m$ in a (not necessarily unique) order such that if $\psi$ is any transition path joining loop classes $\mathcal{L}_i$ and $\mathcal{L}_j$, then $i<j$.
Now suppose $\eta=(e_1,\ldots,e_n)\in\Omega^*$ is any finite rooted path.
Then we can uniquely write
\begin{equation*}
    \eta = \phi\lambda_1\psi_1\ldots\psi_{m-1}\lambda_m
\end{equation*}
for possibly empty paths $\phi,\psi_i,\lambda_i$, where $\phi$ is an initial path, each $\lambda_i$ is a path in $\mathcal{L}_i$, and each $\psi_i$ is a transition path.
We call the tuple $(\lambda_1,\ldots,\lambda_m)$ the \defn{decomposition} of the path $\eta$.
\begin{example}
    In \cref{f:gen-tr-graph}, an example of a valid order is $\mathcal{L}_1,\mathcal{L}_2,\mathcal{L}_3,\mathcal{L}_4$.
    Note that any decomposition can contain a maximum of $3$ non-empty paths $\lambda_i$ (corresponding to the loop classes $\mathcal{L}_1,\mathcal{L}_3,\mathcal{L}_4$).
\end{example}
By convention, if $\lambda_i$ is an empty path, we write $\norm{T(\lambda_i)}=1$.
\begin{definition}
    We say that the transition graph $\mathcal{G}$ is \defn{decomposable} if for any path $\eta\in\Omega^*$ with decomposition $(\lambda_1,\ldots,\lambda_m)$, we have
    \begin{equation*}
        \norm{T(\eta)}\asymp \norm{T(\lambda_1)}\cdots \norm{T(\lambda_m)}
    \end{equation*}
    with constants depending only on the transition graph $\mathcal{G}$.
\end{definition}
We now discuss a few examples in which the transition graph $\mathcal{G}$ is decomposable.
\begin{lemma}\label{l:pos-tr}
    Suppose every transition path $\eta$ has that $T(\eta)$ is a strictly positive matrix.
    Then $\mathcal{G}$ is decomposable.
\end{lemma}
\begin{proof}
    Since there are only finitely many transition paths, there exists a constant $C>0$ such that for any transition path $\psi$ and valid indices $i,j$, $T(\eta)_{i,j}\geq C$.
    But now if $\eta\in\Omega^*$ has decomposition $(\lambda_1,\ldots,\lambda_m)$, we can uniquely write $\eta=\phi\lambda_1\psi_1\ldots\psi_{m-1}\lambda_m$ where each $\psi_i$ is a transition path.
    Then by \cref{l:left-prod},
    \begin{align*}
        \norm{T(\eta)}&=\norm{T(\phi\lambda_1\psi_1\ldots\psi_{m-1}\lambda_m)} \asymp \norm{T(\lambda_1\psi_1\ldots\psi_{m-1}\lambda_m)}\\
                      &\geq C^{m-1}\norm{T(\lambda_1)}\cdots\norm{T(\lambda_m)}.
    \end{align*}
    Of course, $C$ and $m$ depend only on $\mathcal{G}$, so the lower bound holds.
    The upper bound always follows by submultiplicativity of the matrix norm, since there are only finitely many choices for the paths $\phi$, $\psi_i$.
\end{proof}
\begin{lemma}\label{l:size-one-loops}
    Suppose that every vertex in a non-essential loop class is a neighbour set of size one.
    Then $\mathcal{G}$ is decomposable.
\end{lemma}
\begin{proof}
    By \cref{l:pos-tr}, it suffices to show for any transition path $\eta$ that $T(\eta)$ is a strictly positive matrix.
    By definition of an essential loop class, if $\eta$ is a transition path from loop classes $\mathcal{L}_i$ to $\mathcal{L}_j$, then $\mathcal{L}_i$ is non-essential.
    Thus by assumption, $\eta$ is a path beginning at a vertex with neighbour set consisting of a single neighbour, so that $T(\eta)$ is a matrix with 1 row.
    Since every transition matrix has a positive entry in every column by \cref{l:pos-col-row}, $T(\eta)$ is a strictly positive matrix, so the result follows by \cref{l:pos-tr}
\end{proof}
We now establish an irreducibility type condition to guarantee that the transition graph is decomposable when all non-essential loop classes are simple.

We begin with some general observations about non-negative irreducible matrices.
Let $M$ be an irreducible matrix with spectral radius $r$.
It is known that if $M$ is a strictly positive matrix, then the limit $\lim_{k\to\infty} M^k/r^k$ exists \cite{sen1981}.
While this limit need not exist in general if $M$ is irreducible, using similar arguments, one can show that there are constants $c_1,c_2>0$ such that for all $n$ sufficiently large, either $M^n_{i,j}=0$ or
\begin{equation*}
    c_1 r^n\leq M^n_{i,j}\leq c_2 r^n.
\end{equation*}
In particular, suppose $M_1,\ldots,M_k$ are irreducible matrices and $A_1,\ldots,A_{k+1}$ are such that $A=A_1 M_1^{n_1}\cdots A_kM_k^{n_k}A_{k+1}\neq 0$.
Then for all $n_i$ sufficiently large, if $M_i$ has spectral radius $r_i$, either $A_{k,\ell}=0$ or
\begin{equation}\label{e:irred-power}
    A_{j,\ell}\asymp_{A_1,\ldots,A_{k+1}} r_1^{n_1}\cdots r_k^{n_k}.
\end{equation}
This observation is the main idea in the following result.
\begin{lemma}
    Suppose every non-essential loop class is simple.
    For each simple loop class $\mathcal{L}$, suppose there is a path $\theta$ in $\mathcal{L}$ beginning and ending at the same vertex such that $T(\theta)$ is an irreducible matrix.
    Then $\mathcal{G}$ is decomposable.
\end{lemma}
\begin{proof}
    For simplicity, we assume there is a unique essental class; the proof in the general case follows similarly.
    Denote the simple loop classes by $\mathcal{L}_1,\ldots,\mathcal{L}_k$, and for each $1\leq i\leq k$, let $\theta_i$ be a cycle in $V(\mathcal{L}_i)$ such that $T(\theta_i)$ is an irreducible matrix.
    Let $T(\theta_i)$ have spectral radius $r_i$.

    If $\eta\in\Omega^*$ is an arbitrary path, it has decomposition of the form $(\lambda_1,\ldots,\lambda_k,\xi)$ where $\lambda_i=\gamma_i^{(1)}\theta_i^{n_i}\gamma_i^{(2)}$ with $n_i$ maximal and $\xi$ is a path in the essential loop class.
    Since $n_i$ is maximal and $\mathcal{L}_i$ is simple, the paths $\gamma_i^{(1)}$ and $\gamma_i^{(2)}$ have length at most the length of $\theta$, so there are only finitely many possible paths $\gamma_i^{(j)}$.
    Thus by \cref{e:irred-power},
    \begin{equation*}
        \norm{T(\lambda_i)}\asymp r_i^{n_i}.
    \end{equation*}
    Now, we may write
    \begin{equation*}
        \eta = \phi\theta_1^{n_1}\psi_1\theta_2^{n_2}\ldots\theta_k^{n_k}\psi_k\xi
    \end{equation*}
    where $\phi=\phi'\gamma_1^{(1)}$ with $\phi'$ an initial path, and each $\psi_i$ is of the form $\gamma_i^{(2)}\psi_i'\gamma_{i+1}^{(1)}$ for $i<k$ or $\gamma_k^{(2)}\psi_k'$, and the paths $\psi_i'$ are transition paths.
    Of course, some of the paths $\gamma_i^{(j)}$ or $\psi_i'$ may be the empty path.
    The point here is that there are only finitely many possible choices for the paths $\phi,\psi_1,\ldots,\psi_k$, independent of the choice of $\eta$.

    It always holds that $\norm{T(\eta)}\preccurlyeq\norm{T(\lambda_1)}\cdots\norm{T(\lambda_k)}\norm{T(\xi)}$.
    Thus it suffices to show that
    \begin{equation*}
        \norm{T(\eta)}\succcurlyeq r_1^{n_1}\cdots r_k^{n_k}\norm{T(\xi)}.
    \end{equation*}
    Let $M=T(\phi\theta_1^{n_1}\psi_1\theta_2^{n_2}\ldots\theta_k^{n_k}\psi_k)$.
    By \cref{e:irred-power}, for each index $j,\ell$, either $M_{j,\ell}=0$ or
    \begin{equation*}
        M_{j,\ell}\asymp r_1^{n_1}\cdots r_k^{n_k}.
    \end{equation*}
    But now if $T(\xi)$ has maximal entry $T(\xi)_{p,q}$, we have $\norm{T(\xi)}\succcurlyeq T(\xi)_{p,q}$.
    Since column $p$ of the matrix $M$ has a non-negative entry by \cref{l:pos-col-row}, get $p'$ such that $M_{p',p}\succcurlyeq r_1^{n_1}\cdots r_k^{n_k}$ and
    \begin{equation*}
        \norm{T(\eta)}=\norm{M\cdot T(\xi)}\geq M_{p',p}T(\xi)_{p,q}\succcurlyeq r_1^{n_1}\cdots r_k^{n_k}\norm{T(\xi)}
    \end{equation*}
    as required.
\end{proof}

\section{Loop class spectra and a multifractal formalism}
\subsection{Measures and metric structure on paths in the transition graph}\label{ss:symb-defs}
The set $\Omega^\infty$ of infinite rooted paths has a natural metric space structure given by the weights.
Given paths $\gamma_1,\gamma_2$, define
\begin{equation*}
    d(\gamma_1,\gamma_2)=\inf\{W(\eta):\eta\text{ a prefix of }\gamma_1\text{ and }\gamma_2\}.
\end{equation*}
The topology is generated by the closed and open cylinders
\begin{equation*}
    [\eta]:=\{\gamma\in\Omega^*:\eta\text{ a prefix of }\gamma\}.
\end{equation*}
It is easy to see that this space is compact and totally disconnected.
Of course, $\pi([\eta])=\pi(\eta)\cap K$ where we recall that $\pi(\eta)$ is the net interval with symbolic representation $\eta$.
It is productive to interpret the space $\Omega^\infty$ with the above metric as a ``separated'' version of the set $K$.

We have the following straightforward result:
\begin{lemma}\label{l:pi-Lip}
    The map $\pi:\Omega^\infty\to K$ is Lipschitz with constant $1$.
\end{lemma}
\begin{proof}
    Let $\gamma_1$ and $\gamma_2$ be two distinct paths in $\Omega^\infty$ with maximal common prefix $\eta\in\Omega^n$, so that $d(\gamma_1,\gamma_2)=W(\eta)$.
    Let $\Delta\in\mathcal{P}_n$ have symbolic representation $\eta$.
    By definition, $W(\eta)=\diam(\Delta)$.
    But then $\pi(\gamma_1),\pi(\gamma_2)\in\Delta$ so
    \begin{equation*}
        |\pi(\gamma_1)-\pi(\gamma_2)|\leq\diam(\Delta) = W(\eta)=d(\gamma_1,\gamma_2)
    \end{equation*}
    as required.
\end{proof}
Our general philosophy is to establish multifractal properties of the space $\Omega^\infty$ (in terms of the corresponding subspaces $\Omega^\infty_{\mathcal{L},\zeta}$ defined below), and then translate these results to the self-similar measure $\mu$.
However, the main difficulty in establishing corresponding multifractal results is that, in general, the map $\pi$ is not in general bi-Lipschitz (even when restricted to the maximal domain on which it is injective).
Many of the technical results in the following sections are established to overcome this.

Since the graph $\mathcal{G}$ is not in general strongly connected, we will study subspaces of $\Omega^\infty$ corresponding to the loop classes.
Fix a loop class $\mathcal{L}$, and let $\zeta\in\Omega^*$ be a fixed path which ends at a vertex $v$ in $\mathcal{L}$.
Recall that $\Omega^\infty_{\mathcal{L}}$ is the set of infinite paths eventually in the loop class $\mathcal{L}$.
Now, define the set
\begin{equation*}
    \Omega^\infty_{\mathcal{L},\zeta}:=\{\gamma\in\Omega^\infty_{\mathcal{L}}:\zeta\text{ is a prefix of }\gamma\}
\end{equation*}
This is a compact subspace of $\Omega^\infty$ (note that the sets $\Omega^\infty_{\mathcal{L}}$ need not be compact).
We also define the analgous sets
\begin{equation*}
    \Omega^*_{\mathcal{L},\zeta}:=\{\eta\in\Omega^*_{\mathcal{L}}:\zeta\text{ is a prefix of }\eta\}
\end{equation*}
consisting of finite, rather than infinite, paths.
Often, given $\eta\in\Omega^*_{\mathcal{L},\zeta}$, we will abuse notation and write $[\eta]$ to denote the cylinder $[\eta]\cap\Omega^\infty_{\mathcal{L},\zeta}\subseteq\Omega^\infty$.

Suppose $\Delta=\pi(\zeta)$ is the net interval with symblic representation $\zeta$.
Then one can verify that
\begin{equation*}
    \pi\bigl(\Omega^\infty_{\mathcal{L},\zeta}\bigr)=K_{\mathcal{L}}\cap\Delta.
\end{equation*}

We now turn our attention to the measure $\mu$.
Since distinct net intervals in the same level overlap only on endpoints and the self-similar measure $\mu$ is non-atomic, one can verify that the rule
\begin{equation*}
    \mu(\pi(\eta))=\norm{T(\eta)}
\end{equation*}
for paths $\eta\in\Omega^*$ extends to a unique Borel measure on $\Omega^\infty$.
We would like to restrict this measure $\mu\circ \pi$ to the subsets $\Omega^\infty_{\mathcal{L},\zeta}$ in a meaningful way.
However, these sets can have measure $0$ in $\Omega^\infty$ (in fact, they have non-zero measure if and only if $\mathcal{L}$ is an essential loop class).

Regardless, it is convenient to simply consider the measure $\mu\circ\pi$ as being defined on finite rooted paths (or the corresponding cylinders).
With this in mind, we define a function $\rho:\Omega^*\to [0,1]$ by the rule
\begin{equation*}
    \rho(\eta)=\mu\circ\pi(\eta).
\end{equation*}
Now $\rho$ restricts naturally to a function $\rho:\Omega^*_{\mathcal{L},\zeta}\to[0,1]$, though this restriction is not in general additive.

\subsubsection{Loop class $L^q$-spectra}
We now use the function $\rho$ to define an analogue of the $L^q$-spectrum of measures for loop classes.

To motivate this, we first state an equivalent formulation of the $L^q$-spectrum of $\mu$ using the function $\rho$.
Set
\begin{equation*}
    \mathcal{F}(t) = \bigl\{\eta=(e_1,\ldots,e_n)\in\Omega^*:W(e_1\ldots e_n)\leq t<W(e_1\ldots e_{n-1})\bigr\}.
\end{equation*}
One can think of the sets $\mathcal{F}(t)$ as a ``scale-uniform'' analogue in $\Omega^\infty$ of the partitions $\mathcal{P}_n$ (which may contaimay contain intervals with vastly different diameters).
We then have the following standard result, which is a weighted version of \cite[Prop. 4.3]{hhstoappear} or \cite[Prop. 5.6]{fen2003}.
We include the main details for the convenience of the reader.
\begin{proposition}\label{p:lq-lim}
    Let $\mu$ be a self-similar measure satisfying the finite neighbour condition.
    Then
    \begin{equation*}
        \tau_\mu(q)= \liminf_{t\to 0}\frac{\log\sum_{\eta\in\mathcal{F}(t)}\rho(\eta)^q}{\log t}
    \end{equation*}
\end{proposition}
\begin{proof}
    First suppose $x\in K$ and $t>0$ is arbitrary, and let $\{B(x_i,t)\}_i$ be any centred packing of $K$.
    If $\eta\in\mathcal{F}(t)$ has $x\in\pi(\eta)$, we always have $\pi(\eta)\subseteq B(x,t)$, so for $q<0$
    \begin{equation*}
        \sum_{\eta\in\mathcal{F}(t)}\rho(\eta)^q\geq\sum_i\mu(B(x_i,t))^q.
    \end{equation*}
    On the other hand, for $q\geq 0$, since there are only finitely many edge weights, there is some $N$ such that $\#\{\eta\in\mathcal{F}(t):\pi(\eta)\cap B(x,t)\neq\emptyset\}\leq N$.
    Since a given net interval $\pi(\eta)$ overlaps with at most 2 distinct balls in $\{B(x_i,t)\}_i$, for $q\geq 0$ by Jensen's inequality
    \begin{equation*}
        \sum_i\mu(B(x_i,t))^q\leq\sum_i\Bigl(\sum_{\substack{\eta\in\mathcal{F}(t)\\\pi(\eta)\cap B(x_i,t)\neq\emptyset}}\rho(\eta)\Bigr)^q\preccurlyeq_q\sum_{\eta\in\mathcal{F}(t)}\rho(\eta)^q.
    \end{equation*}
    Thus
    \begin{equation*}
        \tau_\mu(q) \geq \liminf_{t\to 0}\frac{\log\sum_{\eta\in\mathcal{F}(t)}\rho(\eta)^q}{\log t}.
    \end{equation*}

    Conversely, suppose $\Delta=\pi(\eta)$ is some net interval, where $\eta\in\mathcal{F}(t)$.
    If $f$ is any neighbour of $\Delta$, since $f(K)\cap(0,1)\neq\emptyset$, there exists some word some $\epsilon>0$ and $\tau\in\mathcal{I}^*$ depending only on $f$ such that $f\circ S_\tau(K)\subseteq(\epsilon,1-\epsilon)$ and $0<r_\tau<\epsilon$.
    Since there are only finitely many neighbour sets, and hence only finitely many neighbours, we may assume $\epsilon\geq \epsilon_0>0$ and $p_\tau\geq p_0>0$ for some fixed $\epsilon_0,p_0$.

    Now by \cref{p:mat-mu} along with \cref{e:qi-formula}, there is some $M>0$ fixed such that there is $f\in\vs(\Delta)$ satisfying
    \begin{equation*}
        \rho(\eta)=\mu(\Delta)\leq M\sum_{\substack{\sigma\in\mathcal{I}^*\\S_\sigma=T_\Delta\circ f}}p_\sigma.
    \end{equation*}
    But then by choice of $\tau$, with $x_\eta=T_\Delta\circ f\circ S_\tau(0)\in K$ and $r_\eta=\epsilon|r_\tau|\diam(\Delta)$, we have $B(x_\eta,r_\eta)\subseteq\Delta$ and
    \begin{equation*}
        \rho(\eta)\geq \mu(B(x_\eta,r_\eta))\geq \mu(T_\Delta\circ f\circ S_\tau(K))\succcurlyeq\sum_{\substack{\sigma\in\mathcal{I}^*\\S_\sigma=T_\Delta\circ f}}p_\sigma p_\tau\succcurlyeq\rho(\eta).
    \end{equation*}
    Thus the centred packing $\{B(x_\eta,r_\eta)\}_\eta$ satisfies
    \begin{equation*}
        \sum_{\eta\in\mathcal{F}(t)}\mu(B(x_\eta,r_\eta))^q\asymp_q \sum_{\eta\in\mathcal{F}(t)}\rho(\eta)^q.
    \end{equation*}
    But $r_\eta\asymp t$, so for any $q\in\R$,
    \begin{equation*}
        \tau_\mu(q) \leq \liminf_{t\to 0}\frac{\log\sum_{\eta\in\mathcal{F}(t)}\rho(\eta)^q}{\log t}.
    \end{equation*}
    This gives the desired result.
\end{proof}
When $q\geq 0$, for any ``sufficiently uniform'' (e.g. dyadic) partition of $K$, one may always define the $L^q$-spectrum of any finite measure with respect to such a partition (see, for example, \cite[Prop. 3.1]{ln1999}).
When $q<0$, such an operation is more delicate since the intervals in a partition can intersect $K$ on sets of disproportionately small $\mu$-measure.
\cref{p:lq-lim} essentially states that the partitions $\{\pi(\eta):\eta\in\mathcal{F}(t)\}$ of $K$ for $t>0$ avoid this issue.

Now for $t>0$, and $\zeta$ and $\mathcal{L}$ defined as above, set
\begin{align*}
    \mathcal{F}_{\mathcal{L},\zeta}(t)&=\mathcal{F}(t)\cap\Omega^*_{\mathcal{L},\zeta}.
\end{align*}
We then define
\begin{equation*}
    \tau_{\mathcal{L},\zeta}(q) = \lim_{t\to 0}\frac{\log\sum_{\eta\in\mathcal{F}_{\mathcal{L},\zeta}(t)}\rho(\eta)^q}{\log t}.
\end{equation*}
We have the following basic lemma.
The argument giving existence of the limit is similar to \cite[Lem. 2.2]{fen2009}.
\begin{lemma}\label{l:lq-limit}
    The function $\tau_{\mathcal{L},\zeta}(q)$ is a concave function of $q$, and the limit exists for any $q\in\R$.
    Moreover, if $\zeta'$ is any other path ending in $\mathcal{L}$, then $\tau_{\mathcal{L},\zeta'}=\tau_{\mathcal{L},\zeta}$.
\end{lemma}
\begin{proof}
    Concavity is a standard application of Hölder's inequality.
    We now see existence of the limit.
    Write $A_q(t)=\sum_{\eta\in\mathcal{F}_{\mathcal{L},\zeta}(t)}\rho(\eta)^q$.
    All implicit constants below may depend on the choice of $\zeta$.

    First suppose $q\geq 0$ and set $r_0=W_{\min}\cdot\min\{W(\zeta\gamma_w):w\in V(\mathcal{L})\}$ where $W_{\min}=\min\{W(e):e\in E(\mathcal{L})\}$.
    Suppose $\eta\in\mathcal{F}_{\mathcal{L},\zeta}(r_0t_1t_2)$, so we may write $\eta=\eta_1\phi$ where $\eta_1\in\mathcal{F}_{\mathcal{L},\zeta}(t_2)$.
    If $\phi$ begins at the vertex $w$, by choice of $r_0$ write $\phi=\psi\phi_0$ such that with $\eta_2=\zeta\gamma_w\psi$, $\eta_2\in\mathcal{F}_{\mathcal{L},\zeta}(t_2)$.
    Observe that $W(\phi_0)\asymp r_0$ so there are only finitely possible values of $\norm{T(\phi_0)}$.
    Thus by \cref{l:left-prod} we have $\norm{T(\psi)}\asymp \rho(\eta_2)$ so that
    \begin{equation*}
        \rho(\eta)\leq \rho(\eta_1)\norm{T(\psi)}\norm{T(\phi_0)}\preccurlyeq\rho(\eta_1)\rho(\eta_2).
    \end{equation*}
    Thus $A_q(r_0t_1t_2)\preccurlyeq_q A_q(t_1)A_q(t_2)$ so the limit exists for $q\geq 0$ by submultiplicativity.

    Now suppose $q<0$.
    Let $\zeta$ end at the vertex $v\in V(\mathcal{L})$, and for each $w\in V(\mathcal{L})$, let $\gamma_w$ be a path in $\mathcal{L}$ from $v$ to $w$.
    Given $\eta_i\in\mathcal{F}_{\mathcal{L},\zeta}(t_i)$ for $i=1,2$ and $t_i$ sufficiently small, write $\eta_2=\zeta\phi_2$ so that $W(\phi_2)\asymp t_2$ and $\norm{T(\phi_2)}\asymp\norm{T(\eta_2)}$ by \cref{l:left-prod}.
    Then if the path $\eta_1$ ends at the vertex $w$, $\eta_1\gamma_w\phi_2$ is an admissible path with $W(\eta_1\gamma_w\phi_2)\asymp t_1t_2$, so there exists some fixed $r_0>0$ such that $W(\eta_1\gamma_w\phi_2)\geq r_0 t_1t_2$.
    Thus get a path $\psi$ with $W(\psi)\asymp r_0$ such that $\eta_1\gamma_w\phi_2\psi\in\mathcal{F}_{\mathcal{L},\zeta}(r_0t_1t_2)$, and by \cref{l:left-prod}
    \begin{equation*}
        \rho(\eta_1\gamma_w\phi_2\psi)^q\succcurlyeq_q\bigl(\norm{T(\gamma_w)}\norm{T(\psi)}\bigr)^q\cdot\rho(\eta_1)^q\rho(\eta_2)^q.
    \end{equation*}
    But $\mathcal{L}$ is a finite graph (so there are only finitely many paths $\gamma_w$) and $W(\psi)\asymp r_0$ (so there are only finitely many paths $\psi$).
    Thus $A_q(t_1)A_q(t_2)\preccurlyeq_q A_q(r_0 t_1t_2)$ and the limit exists for $q<0$ by supermultiplicativity.

    To see the final claim, suppose $\zeta$ ends at the vertex $v$ and $\zeta'$ ends at the vertex $v'$ where $v,v'$ are both vertices in $\mathcal{L}$.
    Let $\phi$ be any path in $\mathcal{L}$ from $v$ to $v'$.
    Let $\Psi:\Omega^*_{\mathcal{L}}(\zeta')\to\Omega^*_{\mathcal{L},\zeta}$ be given by $\Psi(\zeta'\eta)=\zeta\phi\eta$, and note that
    \begin{equation}\label{e:incl}
        \Psi(\mathcal{F}_{\mathcal{L},\zeta'}(W(\zeta')t))\subseteq\mathcal{F}_{\mathcal{L},\zeta}(W(\zeta\phi)t).
    \end{equation}
    Now if $\zeta'\eta\in\Omega^*_{\mathcal{L},\zeta}$, by \cref{l:left-prod},
    \begin{equation*}
        \rho(\Psi(\zeta'\eta))=\norm{T(\zeta\phi\eta)}\asymp\norm{T(\eta)}\asymp \norm{T(\zeta'\eta)}=\rho(\zeta'\eta)
    \end{equation*}
    and combining this with \cref{e:incl} yields
    \begin{equation*}
        \sum_{\eta\in\mathcal{F}_{\mathcal{L},\zeta'}(W(\zeta')t)}\rho(\eta)^q\succcurlyeq_q \sum_{\eta\in\mathcal{F}_{\mathcal{L},\zeta}(W(\zeta\phi)t)}\rho(\eta)^q.
    \end{equation*}
    Since $\zeta$, $\zeta'$, and $\phi$ are fixed , it follows that $\tau_{\mathcal{L},\zeta}(q)\geq\tau_{\mathcal{L},\zeta'}(q)$.
    The reverse inequality follow by the same argument with the roles of $\zeta$ and $\zeta'$ swapped.
\end{proof}
\begin{proposition}\label{p:ess-formula}
    Suppose $\mathcal{L}$ is an essential loop class of $\mathcal{G}$.
    Then if $\Delta$ is any net interval with neighbour set $\vs(\Delta)\in V(\mathcal{L})$, with $\nu=\mu|_{\Delta}$, we have
    \begin{equation*}
        \tau_{\mathcal{L}}(q)=\tau_\nu(q).
    \end{equation*}
    In particular, $\tau_{\mathcal{L}}(q)=\tau_\mu(q)$ for any $q\geq 0$.
\end{proposition}
\begin{proof}
    This follows by the same argument as \cref{p:lq-lim}, observing that if $\zeta\in\Omega^*$ is a path ending in an essential loop class $\mathcal{L}$, then $\eta\in\Omega^*_{\mathcal{L},\zeta}$ if and only if $\eta\in\Omega^*$ and $\zeta$ is a prefix of $\eta$.

    That $\tau_{\mathcal{L}}(q)=\tau_\mu(q)$ for $q\geq 0$ follows by standard arguments (see, for example, \cite[Prop. 3.1]{fl2009}).
\end{proof}
\begin{remark}\label{r:ess-unique}
    In fact, using arguments similar to the proof of \cite[Thm. 4.5]{rut2021}, one can show that if $\mathcal{L}$ and $\mathcal{L}'$ are essential classes, then $\tau_{\mathcal{L}}=\tau_{\mathcal{L}'}$.
    In practice, with the standard choices of iteration rules given in \cref{ex:uniform-transition} and \cref{ex:weighted-transition}, there will always be a unique essential class.
\end{remark}

\subsubsection{Loop class local dimensions}
Given an infinite path $\gamma=(e_n)_{n=1}^\infty\in\Omega^\infty$, recall that $\gamma|n=(e_1,\ldots,e_n)\in\Omega^n$.
We then define
\begin{equation*}
    \underline{\dim}_{\loc}(\rho,\gamma)=\liminf_{n\to\infty}\frac{\log\rho(\gamma|n)}{\log W(\gamma|n)}
\end{equation*}
with similar definitions for the (upper) local dimension.
With this, we define
\begin{equation*}
    E_{\mathcal{L},\zeta}(\alpha) = \{\gamma\in\Omega^\infty_{\mathcal{L},\zeta}:\underline{\dim}_{\loc}(\rho,\gamma)=\overline{\dim}_{\loc}(\rho,\gamma)=\alpha\}.
\end{equation*}
Now let $f_{\mathcal{L},\zeta}:\R\to\R\cup\{-\infty\}$ be given by
\begin{equation*}
    f_{\mathcal{L},\zeta}(\alpha):=\dim_H E_{\mathcal{L},\zeta}(\alpha).
\end{equation*}
Note that $E_{\mathcal{L},\zeta}(\alpha)$ may be the empty set; by convention, we write $\dim_H\emptyset=-\infty$.
Following the theme for $L^q$-spectra, we have the following easy result.
\begin{lemma}
    If $\mathcal{L}$ is a loop class and $\zeta,\zeta'\in\Omega^*$ both end in $\mathcal{L}$, then
    \begin{equation*}
        f_{\mathcal{L},\zeta}(\alpha)=f_{\mathcal{L},\zeta'}(\alpha).
    \end{equation*}
\end{lemma}
\begin{proof}
    Let $\zeta$ end at the vertex $v$ and $\zeta'$ end at the vertex $v'$, and let $\phi$ be a path in $\mathcal{L}$ from $v$ to $v'$.
    Now if $\zeta'\gamma\in\Omega^\infty_{\mathcal{L}}(\zeta')$, then $\zeta\phi\gamma\in\Omega^\infty_{\mathcal{L},\zeta}$ and by \cref{l:left-prod},
    \begin{equation*}
        \underline{\dim}_{\loc}(\rho,\zeta'\gamma)=\underline{\dim}_{\loc}(\rho,\zeta\phi\gamma).
    \end{equation*}
    Moreover, it is straightforward to verify that the map $\zeta'\gamma\mapsto\zeta\phi\gamma$ is bi-Lipschitz on its image, so that $f_{\mathcal{L},\zeta}(\alpha)\geq f_{\mathcal{L},\zeta'}(\alpha)$ for any $\alpha$.
    The same argument yields the converse inequality, as required.
\end{proof}

\subsection{Multifractal formalism for irreducible loop classes}\label{ss:mf-Moran}
We maintain notation from the previous section, recalling that $\zeta\in\Omega^*$ is a path ending at a vertex in the loop class $\mathcal{L}$.

In light of the results in the previous section, the following notions are well-defined.
\begin{definition}
    We define the \defn{loop class $L^q$-spectrum}, denoted by $\tau_{\mathcal{L}}(q)$, as $\tau_{\mathcal{L}}(q)=\tau_{\mathcal{L},\zeta}(q)$.
    Similarly, we define the \defn{loop class multifractal spectrum} by $f_{\mathcal{L}}(\alpha)=f_{\mathcal{L},\zeta}(\alpha)$.
\end{definition}
For convenience, we write
\begin{align}\label{e:a-min-max}
    \alpha_{\min}(\mathcal{L})&=\lim_{q\to\infty}\frac{\tau_{\mathcal{L}}(q)}{q}& \alpha_{\max}(\mathcal{L})&=\lim_{q\to-\infty}\frac{\tau_{\mathcal{L}}(q)}{q}.
\end{align}
The limits necessarily exist by concavity of $\tau_{\mathcal{L}}(q)$, and a straightforward argument shows that they are finite.

Our main result in this section is the following multifractal formalism, which relates the multifractal spectrum with the $L^q$-spectrum on the loop class.
\begin{theorem}\label{t:multi-f}
    Let $\mathcal{L}$ be an irreducible loop class in $\mathcal{G}$.
    Then $f_{\mathcal{L}}=\tau_{\mathcal{L}}^*$.
\end{theorem}
In particular, $f_{\mathcal{L}}$ is a concave function taking finite values precisely on the interval $[\alpha_{\min}(\mathcal{L}),\alpha_{\max}(\mathcal{L})]$.

By definition, it suffices to show $f_{\mathcal{L},\zeta}=\tau_{\mathcal{L},\zeta}^*$ for a path $\zeta\in\Omega^*$ ending at a vertex in $\mathcal{L}$.
There are many ways to prove this result.
One could use a weighted version of the arguments in \cite{fen2009}, which are proven in a similar irreducible matrix-product setting.
Another option is to follow the arguments in \cite{fl2009}.

We find it most efficient to use the following result, which is a simplified ``symbolic'' version of \cite[Thm. 2.2]{fen2012}.
\begin{proposition}[\cite{fen2012}]\label{p:asymp-good}
    Suppose for any $q\in\R$ such that the derivative $\tau_{\mathcal{L},\zeta}'(q)=\alpha$ exists, there exist numbers $b(q,k)$ and $c(q,k)$ such that the following properties hold:
    \begin{enumerate}[nl,r]
        \item We have $\lim_{k\to\infty}b(q,k)=0$.
        \item Suppose $n\in\N$ and $\eta\in\mathcal{F}_{\mathcal{L},\zeta}(2^{-n})$.
            Then for any $m\geq c(q,k)$, there are distinct paths $\eta_1,\ldots,\eta_N\in\mathcal{F}_{\mathcal{L},\zeta}(2^{-n-m})$ such that $\eta$ is a prefix of each $\eta_i$,
            \begin{equation*}
                N \geq 2^{m(\tau^*_{\mathcal{L},\zeta}(\alpha)-b(q,k))},
            \end{equation*}
            and
            \begin{equation*}
                2^{-m(\tau_{\mathcal{L},\zeta}'(q)+1/k)}\leq \frac{\rho(\eta_i)}{\rho(\eta)}\leq 2^{-m(\tau_{\mathcal{L},\zeta}'(q)-1/k)}.
            \end{equation*}
    \end{enumerate}
    Then $f_{\mathcal{L},\zeta}(\alpha)=\tau_{\mathcal{L},\zeta}^*(\alpha)$ for each $\alpha\in\R$.
\end{proposition}
We first observe the following standard counting lemma, which is similar to \cite[Prop. 3.3]{fl2009}, but the proof is easier.
\begin{lemma}\label{l:counting}
    Suppose $\mathcal{L}$ is any loop class (not necessarily irreducible) and the derivative $\tau_{\mathcal{L},\zeta}'(q)=\alpha$ exists.
    Then for any $\delta>0$, there is $t_0=t_0(\delta,q)$ such that for all $0<t<t_0$, there is $F^*(t)\subset \mathcal{F}_{\mathcal{L},\zeta}(t)$ such that
    \begin{enumerate}[nl,r]
        \item $\# F^*(t)\geq t^{-\tau_{\mathcal{L},\zeta}^*(\alpha)+\delta(|q|+1)}$ and
        \item $t^{\alpha+\delta}\leq \rho(\eta)\leq t^{\alpha-\delta}$ for each $\eta\in F^*(t)$.
    \end{enumerate}
\end{lemma}
\begin{proof}
    Write $A_q(t)=\sum_{\eta\in\mathcal{F}_{\mathcal{L},\zeta}(t)}\rho(\eta)^q$.
    Since $\tau_{\mathcal{L},\zeta}'(q)$ exists, get $\epsilon>0$ such that
    \begin{align*}
        (\alpha-\delta/2)\epsilon&\leq|\tau_{\mathcal{L},\zeta}(q\pm \epsilon)-\tau_{\mathcal{L},\zeta}(q)|\leq(\alpha+\delta/2)\epsilon.
    \end{align*}
    Let $0<\gamma<\min\{\epsilon\delta/6,\delta/2,1\}$ and observe that $\gamma$ depends only on $\delta$ and $q$.
    Then since the limit defining $\tau_{\mathcal{L},\zeta}$ exists by \cref{l:lq-limit}, get $t_0$ depending on $\gamma$ and $q$ such that for all $0<t<t_0$,
    \begin{equation*}
        t^{\tau_{\mathcal{L},\zeta}(q+z)+\gamma}\leq A_{q+z}(t)\leq t^{\tau_{\mathcal{L},\zeta}(q+z)-\gamma}
    \end{equation*}
    for each $z\in\{0,-\epsilon,\epsilon\}$.
    Next, write $\mathcal{F}_{\mathcal{L},\zeta}(t)=F^-(t)\cup F^*(t)\cup F^+(t)$ where
    \begin{align*}
        F^-(t) &= \{\eta\in\mathcal{F}_{\mathcal{L},\zeta}(t):\rho(\eta)\leq t^{\alpha+\delta}\} &F^+(t) &= \{\eta\in\mathcal{F}_{\mathcal{L},\zeta}(t):\rho(\eta)\geq t^{\alpha-\delta}\}
    \end{align*}
    and $F^*(t) = \mathcal{F}_{\mathcal{L},\zeta}(t)\setminus(F^-(t)\cup F^+(t))$.
    By definition, (ii) holds for $\eta\in F^*(t)$.

    Combining the above inequalities gives that
    \begin{align*}
        \sum_{\eta\in F^-(t)}\rho(\eta)^q&\leq A_{q+\epsilon}(t)t^{-\epsilon(\alpha-\delta)}\leq t^{\tau_{\mathcal{L},\zeta}(q)+\epsilon\delta/2-\gamma}
    \end{align*}
    with the analgous inequality for $F^+(t)$.
    Then since $\gamma,t\in(0,1)$ and $\gamma<\epsilon\delta/6$,
    \begin{align*}
        \sum_{\eta\in F^*(t)}\rho(\eta)^q &\geq t^{\tau_{\mathcal{L},\zeta}(q)+\gamma}-2t^{\tau_{\mathcal{L},\zeta}(q)+\epsilon\delta/2-\gamma} \geq t^{\tau_{\mathcal{L},\zeta}(q)+2\gamma}(t^{-\gamma}-2)\geq t^{\tau_{\mathcal{L},\zeta}(q)+2\gamma}.
    \end{align*}
    But now for each $\eta\in F^*(t)$, we have $\rho(\eta)^q\leq\max\{t^{(\alpha+\delta)q},t^{(\alpha-\delta)q}\}=t^{\alpha q-\delta|q|}$ so that
    \begin{equation*}
        \# F^*(t)\geq t^{-\alpha q+\delta|q|}\sum_{\eta\in F^*(t)}\rho(\eta)^q\geq t^{-\tau_{\mathcal{L},\zeta}^*(\alpha)+\delta|q|+2\gamma}\geq t^{-\tau_{\mathcal{L}}^*(\alpha)+\delta(|q|+1)}
    \end{equation*}
    since $\gamma<\delta/2$, giving (i).
\end{proof}
\begin{proofref}{t:multi-f}
    Let $q\in\R$ with $\tau_{\mathcal{L},\zeta}'(q)=\alpha$, $n\in\N$ and $\eta\in\mathcal{F}_{\mathcal{L},\zeta}(2^{-n})$.
    First suppose we are given some large $m\in\N$ and a path $\phi\in\mathcal{F}_{\mathcal{L},\zeta}(2^{-m})$.
    We construct a path $\Psi(\phi)$ as follows.
    Write $\phi=\zeta\psi$.
    By the irreducibility assumption, there is a path $\gamma\in\mathcal{H}$ such that $\eta_0=\eta\gamma\psi$ is an admissible path and by \cref{l:left-prod}
    \begin{equation}\label{e:rp-constants}
        \rho(\eta_0)\asymp\rho(\eta)\norm{T(\gamma)}\norm{T(\psi)}\asymp \rho(\eta)\rho(\phi).
    \end{equation}
    Since $W(\eta\gamma)\asymp 2^{-n}$ and $W(\phi)\asymp 2^{-m}$, we have $W(\eta_0)\geq 2^{-n-m-m'}$ for some $m'$ depending only on the (fixed) choice of $\zeta$.
    Again by the irreducibility assumption, we can thus obtain $\Psi(\phi)\in\mathcal{F}_{\mathcal{L},\zeta}(2^{-n-m-m'})$ such that $\eta_0$ is a prefix of $\Psi(\phi)$ and $\rho(\Psi(\phi))\asymp\rho(\eta_0)$.
    Let $C$ be a fixed constant such that $C^{-1}\rho(\phi)\leq \rho(\Psi(\phi))/\rho(\eta)\leq C\rho(\phi)$.

    Now by \cref{l:counting} with constant $\delta=1/2k$ such that for all $m_0\geq c_0(q,k)$ there are paths $\phi_1,\ldots,\phi_N\in \mathcal{F}_{\mathcal{L},\zeta}(2^{-m_0})$ such that $N\geq 2^{m_0(\tau_{\mathcal{L},\phi}^*(\alpha)-(|q|+1)/2k)}$ and $2^{-m_0(\alpha+1/2k)}\leq\rho(\phi_i)\leq 2^{-m_0(\alpha-1/2k)}$.
    Now with $m=m_0+m'$ and $\eta_i=\Psi(\phi_i)$, we observe that $\eta_1,\ldots,\eta_N\in\mathcal{F}_{\mathcal{L},\zeta}(2^{-n-m})$ and
    \begin{equation*}
        N\geq 2^{m_0(\tau_{\mathcal{L},\phi}^*(\alpha)-(|q|+1)/2k}\geq 2^{m(\tau_{\mathcal{L},\phi}^*(\alpha)-(|q|+1)/k)}
    \end{equation*}
    and
    \begin{equation*}
        \frac{\rho(\eta_i)}{\rho(\eta)}\leq C\rho(\phi_i)\leq C 2^{-m_0(\alpha-1/2k)}\leq 2^{-m(\alpha-1/k)}
    \end{equation*}
    with a similar lower bound, for all $m\geq c_0(q,k)+m'$ sufficiently large depending only on fixed quantities.
    Thus the conditions for \cref{p:asymp-good} are satisfied, giving the desired result.
\end{proofref}

\subsection{Regular points in level sets of local dimensions}
As before, we fix a path $\zeta\in\Omega^*$ ending at a vertex in the loop class $\mathcal{L}$.

Recall that
\begin{equation*}
    E_{\mathcal{L},\zeta}(\alpha) = \{\gamma\in\Omega^\infty_{\mathcal{L},\zeta}:\underline{\dim}_{\loc}(\rho,\gamma)=\overline{\dim}_{\loc}(\rho,\gamma)=\alpha\}.
\end{equation*}
We wish to show that the set $E_{\mathcal{L},\zeta}(\alpha)$ can be approximated (in the sense of dimensions) by sets of points which have particularly nice properties.
\begin{definition}\label{d:xi-reg}
    Let $\xi$ be a finite path (not necessarily rooted) in $\mathcal{G}$.
    We say that a path $\gamma=(e_n)_{n=1}^\infty\in\Omega^\infty$ is \defn{$\xi$-regular} if there exists a monotonically increasing sequence $(n_j)_{j=1}^\infty\subset\N$ such that $\xi$ is a prefix of $(e_{n_j},e_{n_j+1},\ldots)$ for each $j$ and
    \begin{equation*}
        \lim_{j\to\infty}\frac{n_{j+1}}{n_j}=1.
    \end{equation*}
\end{definition}

This will be of key importance in \cref{s:mf-properties}.
The proof of the following result is very similar to \cite[Prop. 3.2]{fen2009}, so we are somewhat terse with details.
The irreducibility hypothesis is critical in order to obtain this result.
\begin{theorem}\label{t:reg-sub}
    Suppose $\mathcal{L}$ is an irreducible loop class and $\zeta\in\Omega^*$ a path ending at a vertex $v$ in $\mathcal{L}$.
    Then for any $\alpha\in[\alpha_{\min}(\mathcal{L}),\alpha_{\max}(\mathcal{L})]$ and finite path $\xi$ contained in $\mathcal{L}$ beginning at the vertex $v$, there exists $\emptyset\neq\Gamma=\Gamma(\xi)\subset E_{\mathcal{L},\zeta}(\alpha)$ such that
    \begin{equation*}
        \dim_H\Gamma = \dim_H E_{\mathcal{L},\zeta}(\alpha) = f_{\mathcal{L}}(\alpha)
    \end{equation*}
    and $\Gamma$ is composed only of $\xi$-regular points.
\end{theorem}
\begin{proof}
    If $\mathcal{L}$ is simple, this result is immediate.

    Otherwise we assume $\mathcal{L}$ is not simple.
    All cylinders in the proof are taken relative to $\Omega^\infty_{\mathcal{L},\zeta}$.
    Set
    \begin{align*}
        F(\alpha; t,\epsilon) &:= \bigl\{\eta\in\mathcal{F}_{\mathcal{L},\zeta}(t):t^{\alpha+\epsilon}\leq \rho(\eta)\leq t^{\alpha-\epsilon}\bigr\}\\
        G(\alpha; s,\epsilon) &:= \bigcap_{0<t\leq s}\bigcup_{\eta\in F(\alpha;t,\epsilon)}[\eta].
    \end{align*}
    Of course, if $\gamma\in E_{\mathcal{L},\zeta}(\alpha)$, for any $\epsilon>0$, $\gamma\in G(\alpha; t,\epsilon)$ for all $t$ sufficiently small (depending on $\epsilon$) and thus
    \begin{equation*}
        E_{\mathcal{L},\zeta}(\alpha)\subseteq\bigcup_{s>0}G(\alpha; s,\epsilon).
    \end{equation*}
    Since each cylinder $[\eta]$ where $\eta\in F(\alpha; t,\epsilon)$ has diameter $W(\eta)\asymp t$, for any $s>0$ and $\epsilon>0$, 
    \begin{equation*}
        \dim_H G(\alpha; s,\epsilon)\leq \underline{\dim}_B G(\alpha; s,\epsilon)\leq\liminf_{t\to 0}\frac{\log \# F(\alpha; t,\epsilon)}{-\log t}.
    \end{equation*}
    This holds for any $\epsilon>0$.
    Thus by countable stability of the Hausdorff dimension,
    \begin{equation}\label{e:lim-H}
        \dim_H E_{\mathcal{L},\zeta}(\alpha)\leq\liminf_{\epsilon\to 0}\liminf_{t\to 0}\frac{\log \# F(\alpha; t,\epsilon)}{-\log t}.
    \end{equation}

    We now turn to the construction of the set $\Gamma$.
    For the remainder of this proof, unless otherwise stated, all implicit constants may depend on the (fixed) paths $\zeta$ and $\xi$.
    Let $\delta>0$ be arbitrary.
    By \cref{e:lim-H}, there are strictly monotonic sequences $(t_j)_{j=1}^\infty\subset (0,1)$ and $(\epsilon_j)_{j=1}^\infty$ both tending to 0 such that
    \begin{equation}\label{e:tj-def}
        \frac{\log \# F(\alpha; t_j,\epsilon_j)}{-\log t_j}>\dim_H E_{\mathcal{L},\zeta}(\alpha)-\delta
    \end{equation}
    for each $j\in\N$.
    Define a sequence $\{t_j^*\}_{j=1}^\infty$ by
    \begin{equation*}
        \underbrace{t_1,\ldots,t_1}_{N_1},\underbrace{t_2,\ldots,t_2}_{N_2},\ldots,\underbrace{t_i,\ldots,t_i}_{N_i},\ldots
    \end{equation*}
    where $N_j$ is defined recursively by $N_1=1$ and, for $j\geq 2$,
    \begin{equation*}
        N_j=2^{-\log t_{j+1}+N_{j-1}}.
    \end{equation*}
    For each $i$, let $A_i$ denote the set of indices $j\in\N$ where $t_j^*=t_i$.
    Set $\epsilon_j^*=\epsilon_i$ when $t_j^*=t_i$.

    Since there are only finitely many vertices in $\mathcal{L}$ and finitely many possible dimensions of transition matrices, by the pidgeonhole principle, for each $j$, there exists an index $(m_j,n_j)$ and a subset $G_j^*\subset F(\alpha;t_j^*,\epsilon_j^*)$ such that each path in $G_j$ begins and ends at the same vertex and $\# G_j^*\succcurlyeq\# F(\alpha;t_j^*,\epsilon_j^*)$.
    Let $\eta_j^*=\eta_i$ when $t_j^*=t_i$.

    Recall that the path $\xi$ begins at vertex $v$.
    There exist constants $C,D>0$ such that by repeatedly applying irreducibility of $\mathcal{L}$, for each path $\eta_j^*\in G_j^*$, there exist paths $\phi(\eta_j^*),\psi(\eta_j^*)\in\mathcal{H}$ such that the following two conditions hold:
    \begin{enumerate}[nl,r]
        \item the path $\theta(\eta_j^*):=\xi\phi(\eta_j^*)\eta_j^*\psi(\eta_j^*)$ is a cycle beginning and ending at vertex $v$, and
        \item for any $\gamma=\theta(\eta_1^*)\ldots\theta(\eta_k^*)\eta'$ where $\eta'$ is a prefix of $\theta(\eta_{k+1}^*)$,
            \begin{equation*}
                D^k\prod_{i=1}^{k}\norm{T(\eta_i^*)}\geq\norm{T(\gamma)}\geq C^k\prod_{i=1}^{k+1}\norm{T(\eta_i^*)}.
            \end{equation*}
    \end{enumerate}
    Then let for $n\in\N$
    \begin{equation*}
        \mathcal{G}_n=\bigl\{[\zeta\theta(\eta_1^*)\ldots \theta(\eta_n^*)]:(\eta^*_1,\ldots,\eta^*_n)\in\prod_{i=1}^n G_i^*\bigr\}
    \end{equation*}
    which is a nested sequence of families of cylinders, and set
    \begin{equation*}
        \Gamma_\delta= \bigcap_{n=1}^\infty\bigcup_{I\in\mathcal{G}_n}I.
    \end{equation*}
    A direct computation shows that $\Gamma_\delta\subset\Sigma(\alpha)$.

    We now show that $\dim_H\Gamma_\delta\geq\Lambda(\alpha)-\delta$.
    By \cite[Prop. 3.1]{flw2002} (the technical assumptions are immediate to verify), $\dim_H\Gamma_\delta=\liminf_{k\to\infty}a_k$ where $a_k$ satisfies
    \begin{equation*}
        \sum_{(\eta_1^*,\ldots,\eta_k^*)\in G_1\times\cdots\times G_k}W(\eta_1^*\ldots \eta_k^*)^{a_k}=1.
    \end{equation*}
    Let $1\leq j\leq k$ and choose $i$ such that $j\in A_i$.
    As $\xi$ is fixed, $W(\theta(\eta_j^*))\asymp W(\eta_j^*)\asymp t_i$ and $\eta_j^*\in F(\alpha;t_j^*,\epsilon_j^*)$ so $\eta_j^*\in\mathcal{F}_{t_j^*}(\Delta)=\mathcal{F}_{t_i}(\Delta)$.
    Let $r>0$ be such that $W(\eta_j^*)\geq rt_j^*$.
    Thus since $(t_j^*)\to 0$ and $\# F(\alpha; t_j^*,\epsilon_j^*)\to\infty$,
    \begin{align*}
        \dim_H\Gamma_\delta&\geq\liminf_{k\to\infty}\frac{\log\prod_{j=1}^k \# G_j^*}{-\log\prod_{j=1}^k(rt_j^*)}\\
                     &\geq \liminf_{k\to\infty}\frac{-k+\log\prod_{j=1}^k\# F(\alpha;t_j^*,\epsilon_j^*)}{k-\log\prod_{j=1}^k t_j^*}\\
                     &= \liminf_{k\to\infty}\frac{\log\prod_{j=1}^k\# F(\alpha;t_j^*,\epsilon_j^*)}{-\log\prod_{j=1}^k t_j^*}.
    \end{align*}
    Now by definition of the $N_j$ and \cref{e:tj-def}, it follows that
    \begin{align*}
        \dim_H\Gamma_\delta&\geq \dim_H E_{\mathcal{L},\zeta}(\alpha)-\delta
    \end{align*}
    as claimed.
    Take $\Gamma=\bigcup_{n=1}^\infty\Gamma_{2^{-n}}$, and the result follows.
\end{proof}

\section{Multifractal analysis of self-similar measures}\label{s:mf-properties}
We continue to use the notation of the previous section.
We fix a WIFS $(S_i,p_i)_{i\in\mathcal{I}}$ with self-similar measure $\mu$.
In particular, we assume that $\Phi$ is an iteration rule with corresponding finite transition graph $\mathcal{G}$, as described in \cref{ss:fnc}.

\subsection{Local dimensions and regular points}
Intuitively, the multifractal analysis of self-similar sets satisfying the finite neighbour condition is related to the multifractal analysis results for loop classes from the preceding section.
However, the exact relationship is somewhat more complicated to establish: while the local dimension of $\rho$ at a path $\gamma$ depends only on the single sequence of edges determining $\gamma$, the local dimension of $\mu$ at a point $x\in K$ can also depend on net intervals which are adjacent to net intervals containing $x$.
This happens when $x$ is the shared boundary point of two distinct net intervals, but it can also happen when $x$ is an interior point approximated very well by boundary points (so that balls $B(x,r)$ overlap significantly with neighbouring net intervals, for infinitely many values of $r$).

In order to better understand this adjacency structure, we introduce the notion of the approximation sequence of an interior point, as well as the set of regular points $K_R\subseteq K$.
Let $x\in K$ be an interior point, which we recall means that $\pi^{-1}(x)=\{\gamma\}$ is a single (infinite) path.
Let $(\Delta_i)_{i=0}^\infty$ with $\Delta_0=[0,1]$ and each $\Delta_{i+1}$ a child of $\Delta_i$ denote the sequence of net intervals corresponding to $\gamma$.
Of course, $\Delta_n=\pi(\gamma|n)$.
Given some $i$ and $[a,b]=\Delta_{i+1}\subseteq\Delta_i=[c,d]$, by the reductions described in \cref{r:pr-ch}, exactly one of $c=a<b<d$, $c<a<b=d$, or $c<a<b<d$ must hold.
Moreover, since $x$ is an interior point, it cannot hold that all $\Delta_k$ share a common left (resp. right) endpoint for all sufficiently large $k$ where the left (resp. right) endpoint is also the right (resp. left) endpoint of some adjacent net interval.
In particular, there exists a monotonically increasing infinite sequence $(n_j)_{j=1}^\infty$ such that there exists a neighbourhood of $\Delta_{n_j+2}$ in $K$ which is contained entirely in $\Delta_{n_j}$.

We now make the following definition:
\begin{definition}\label{d:reg-point}
    Given an interior point $x\in K$, we call the sequence $(n_j)_{j=1}^\infty$ described above the \defn{approximation sequence} of $x$.
    We then say that $x$ is \defn{regular} if its approximation sequence satisfies
    \begin{equation*}
        \lim_{j\to\infty}\frac{n_{j+1}}{n_j}=1.
    \end{equation*}
    We denote the set of regular points in $K$ by $K_R$.
\end{definition}
The intuition is that interior points in $K$ which are approximated very well by boundary points are contained in long sequences of net intervals which share left endpoints or right endpoints, so that regular points are those which are poorly approximated by boundary points.

The main point of the approximation sequence is that $x$ is bounded uniformly away from the neighbouring net intervals of $\Delta_{n_j}=\pi(\gamma|n_j)$ for each $j\in\N$.
To be precise, we have the following lemma.
\begin{lemma}\label{l:ap-sub}
    Let $x$ be an interior point with approximation sequence $(n_j)_{j=1}^\infty$.
    There exists some $R>0$ depending only on the IFS and $m=m(R)\in\N$ such that, for any $j\in\N$,
    \begin{equation*}
        \Delta_{n_j+m}\subseteq B(x, R\cdot\diam(\Delta_{n_j}))\cap K\subseteq\Delta_{n_j}.
    \end{equation*}
\end{lemma}
\begin{proof}
    Since there is a neighbourhood of $\Delta_{n_j+2}$ in $K$ contained entirely in $\Delta_{n_j}$, either $\Delta_{n_j+2}\subseteq\Delta_{n_j}^\circ$ or $\Delta_{n_j+2}$ shares an endpoint with $\Delta_{n_j}$ but there is no other adjacent net interval in $\mathcal{P}_{n_j}$.
    We only treat the first case; the second follows by similar arguments.
    We recall by \cref{p:ttype} that the position index $q(\Delta_{i+1},\Delta_i)$ depends only on the neighbour set of $\Delta_i$.
    Thus if we write $\Delta_{n_j+2}=[a,b]$ and $\Delta_{n_j}=[c,d]$ where $a<c<d<b$, there are only finitely many positive values for $(c-a)/(d-c)$ and $(b-d)/(d-c)$.
    The existence of $R$ follows.

    Moreover, recall that $W(e)=\diam(\Delta_i)/\diam(\Delta_{i+1})$ when $\Delta_{i+1}$ is the child of $\Delta_i$ corresponding to the edge $e$.
    Therefore, with $W_{\min}=\min\{W(e):e\in E(\mathcal{G})\}$, it suffices to take $m$ such that $W_{\min}^{m-2}\leq R$.
\end{proof}

Using the approximation sequence, we can establish some basic relationships between local dimensions and their loop class analogues.
A similar version of the following result was first proven in \cite{hr2021}.
\begin{proposition}\label{p:nper-dim}
    Suppose $x$ is an interior point with unique symbolic representation $\gamma$.
    \begin{enumerate}[nl,r]
        \item We always have
            \begin{equation*}
                \underline{\dim}_{\loc}(\mu,x)\leq\underline{\dim}_{\loc}(\rho,\gamma)\leq \overline{\dim}_{\loc}(\mu,x)\leq\overline{\dim}_{\loc}(\rho,\gamma).
            \end{equation*}
        \item If $\dim_{\loc}(\rho,\gamma)$ exists, then $\overline{\dim}_{\loc}(\mu,x)=\dim_{\loc}(\rho,\gamma)$.
        \item If $\dim_{\loc}(\mu,x)$ exists, then $\underline{\dim}_{\loc}(\rho,\gamma)=\dim_{\loc}(\mu,x)$.
    \end{enumerate}
\end{proposition}
\begin{proof}
    It suffices to show (i), since it is clear that (ii) and (iii) follow directly.

    For each $t>0$ let $n(t)$ be maximal such that $W(\gamma|n(t))\leq t$.
    Then if $\Delta_t=\pi(\gamma|n(t))$, we have $\Delta_t\subseteq B(x,t)$ so that
    \begin{align*}
    \overline{\dim}_{\loc}(\mu,x)= \limsup_{t\to 0}\frac{\log\mu(B(x,t))}{\log t}&\leq\limsup_{t\to 0}\frac{\log\rho(\gamma|n(t))}{\log t}\\
                                                                                     &=\overline{\dim}_{\loc}(\rho,\gamma).
    \end{align*}
    Replacing the limit superior with the limit inferior, we also have that $\underline{\dim}_{\loc}(\mu,x)\leq\underline{\dim}_{\loc}(\rho,\gamma)$.

    To get the remaining bound, let $x$ have approximation sequence $(n_k)_{k=1}^\infty$ and let $R,m$ be as in \cref{l:ap-sub}.
    We then have, since $W(\gamma|(n_k+m))\asymp \rho W(\gamma|n_k)$,
    \begin{align*}
        \overline{\dim}_{\loc}(\mu,x)&=\limsup_{t\to 0}\frac{\log\mu(B(x,t))}{\log t}\\
                                     &\geq\limsup_{k\to\infty}\frac{\log\mu(B(x,R\cdot W(\gamma|n_k)))}{\log R\cdot W(\gamma|n_k)}\\
                                     &\geq\limsup_{k\to\infty}\frac{\log\rho(\gamma|(n_k+m))}{\log W(\gamma|(n_k+m))}\\
                                     &\geq\underline{\dim}_{\loc}(\rho,\gamma).
    \end{align*}
    as required.
\end{proof}

When the local dimension exists, the content of the following lemma states that we can extend the nice properties along the approximation sequence to net intervals in similar levels.
A similar statement holds when the loop class local dimension exists.
\begin{lemma}\label{l:approx-reg}
    Suppose $x$ is an interior point with symbolic representation $\gamma$.
    Let $x$ have approximation sequence $(n_j)_{j=1}^\infty$ and let $(k_j)_{j=1}^\infty\subset\N$ satisfy $\lim_{j\to\infty}\frac{k_j}{n_j}=0$.
    \begin{enumerate}[nl,r]
        \item Suppose $\dim_{\loc}(\mu,x)$ exists.
            Then
            \begin{equation*}
                \lim_{j\to\infty}\frac{\log\rho(\gamma|n_j-k_j)}{\log\rho(\gamma|n_j)}=1.
            \end{equation*}
        \item Suppose $\dim_{\loc}(\rho,\gamma)$ exists.
            Then with $m$ from \cref{l:ap-sub},
            \begin{equation*}
                \lim_{j\to\infty}\frac{\log \mu\bigl(B(x, W(\gamma|n_j+m))\bigr)}{\log \mu\bigl(B(x,W(\gamma|n_j+m+k_j))\bigr)}=1
            \end{equation*}
    \end{enumerate}
\end{lemma}
\begin{proof}
    We first see (i).
    For each $i$ let $\Delta_i$ have symbolic representation $\gamma|i$, set $t_i=\diam(\Delta_{n_i})$, and let $\alpha=\dim_{\loc}(\mu,x)$.
    By \cref{l:ap-sub}, there exists some $R>0$ and $m\in\N$ such that
    \begin{equation}\label{e:rho-bd}
        B(x,R t_j)\subseteq\Delta_{n_j}\subseteq\Delta_{n_j-k_j}\subseteq B(x,c_j t_j)
    \end{equation}
    where $c_j=W_{\min}^{-k_j-1}$.
    But then since $\log c_jt_j\asymp n_j-k_j$,
    \begin{align*}
        \lim_{j\to\infty}\frac{(k_j+1)\log W_{\min}}{\log c_jt_j} =\lim_{j\to\infty}\frac{k_j}{n_j-k_j}=0
    \end{align*}
    so that
    \begin{align*}
        \lim_{j\to\infty}\frac{\log\mu(B(x,c_j t_j))}{\log t_j} &= \lim_{j\to\infty}\frac{\log\mu(B(x,c_j t_j))}{(k_j+1)\log W_{\min}+\log c_jt_j}\\
                                                                      &= \lim_{j\to\infty}\frac{\log\mu(B(x,c_j t_j))}{\log c_j t_j}=\alpha
    \end{align*}
    Arguing similarly, we also have $\alpha=\lim_{j\to\infty}\frac{\log\mu(B(x,R t_j))}{\log t_j}$.
    Thus
    \begin{equation*}
        \lim_{j\to\infty}\frac{\log\mu(B(x,c_j t_j))}{\log\mu(B(x,R t_j))}=1
    \end{equation*}
    and the result follows from \cref{e:rho-bd}.

    The proof of (ii) follows similarly after observing that
    \begin{align*}
        \Delta_{n_j}&\supseteq B(x,R\cdot W(\gamma|n_j))\supseteq B(x,W(\gamma|n_j+m))\\
                    &\supseteq B(x,W(\gamma|n_j+m+k_j))\supseteq \Delta_{n_j+m+k_j}.
    \end{align*}
\end{proof}

Finally, the situation is nicest when $x\in K_R$ is a regular point.
Note that this strengthens the usual observations in \cref{p:nper-dim}.
\begin{corollary}\label{c:reg-loc-dim}
    Suppose $x\in K_R$ is a regular point with unique symbolic representation $\gamma$ eventually in the loop class $\mathcal{L}$.
    If either $\dim_{\loc}(\mu,x)$ exists or $\dim_{\loc}(\rho,\gamma)$ exists, then they both exist and are equal.
\end{corollary}
\begin{proof}
    We see this when $\dim_{\loc}(\rho,\gamma)=\alpha$ exists; the proof when $\dim_{\loc}(\mu,x)$ exists is analgous.
    Set $k_j=n_{j+1}-n_j$.
    But then for all $i$ sufficiently large, $n_j + m\leq i\leq n_j+m+k_j$ for some $j$.
    Then since $x$ is a regular point, \cref{l:approx-reg} applies with $(k_j)_{j=1}^\infty$ and
    \begin{align*}
        \limsup_{i\to\infty}\frac{\log\mu\bigl(B(x,W(\gamma|i))\bigr)}{\log W(\gamma|i)}&\leq \limsup_{j\to\infty}\frac{\log \mu\Bigl(B\bigl(x,W(\gamma|(n_j+m))\bigr)\Bigr)}{\log W(\gamma|(n_j+m))}\\
                                                                                        &\leq \limsup_{j\to\infty}\frac{\log\rho(\gamma|n_j)}{\log W(\gamma|n_j)}=\alpha.
    \end{align*}
    The lower bound follows similarly, so that $\dim_{\loc}(\mu,x)=\alpha$.
\end{proof}

\subsection{The upper bound for the multifractal spectrum}
Set
\begin{equation*}
    E_\mu(\alpha;\mathcal{L})=\{x\in\kint_{\mathcal{L}}:\dim_{\loc}(\mu,x)=\alpha\}=\kint_{\mathcal{L}}\cap E_\mu(\alpha).
\end{equation*}
Given a path $\zeta\in\Omega^*$ ending at a vertex in $\mathcal{L}$, one can think of $E_\mu(\alpha;\mathcal{L})\cap\pi(\zeta)$ as an analogue of the set $E_{\mathcal{L},\zeta}(\alpha)$ from \cref{ss:symb-defs}.
In \cref{t:multi-f} the upper bound $f_{\mathcal{L}}(\alpha)\leq\tau_{\mathcal{L}}^*(\alpha)$ always holds, with no assumptions on $\mathcal{L}$.
Here, we show that $\tau_{\mathcal{L}}^*(\alpha)$ is also an upper bound for the Hausdorff dimension of the level sets $E_\mu(\alpha;\mathcal{L})$.

Compare part (i) in \cref{t:m-upper-bound} with \cite[Prop. 4.4]{hr2021}.
Note that our definition of $\alpha_{\min}(\mathcal{L})$ and $\alpha_{\max}(\mathcal{L})$ (as defined in \cref{e:a-min-max}) is formally different from that paper.
Regardless, one can shown that they coincide when $\mathcal{L}$ is an irreducible loop class.
\begin{theorem}\label{t:m-upper-bound}
    Let $(S_i,p_i)_{i\in\mathcal{I}}$ be a WIFS satisfying the finite neighbour condition with associated self-similar measure $\mu$.
    Let $\mathcal{L}$ be a loop class.
    Then
    \begin{enumerate}[nl,r]
        \item $\dim_{\loc}(\mu,x)\in[\alpha_{\min}(\mathcal{L}),\alpha_{\max}(\mathcal{L})]$ for any $x\in \kint_{\mathcal{L}}$ for which the local dimension exists, and
        \item $\dim_H E_\mu(\alpha;\mathcal{L})\leq\tau_{\mathcal{L}}^*(\alpha)$ for any $\alpha\in\R$.
    \end{enumerate}
\end{theorem}
We first recall some notation from \cref{ss:mf-Moran}.
Fix some $\Delta_0\in\mathcal{F}$ such that $\vs(\Delta_0)\in V(\mathcal{L})$.
Let $\Delta_0$ have symbolic representation $\zeta_0$ and set
\begin{equation*}
    A_q(t) = \sum_{\eta\in\mathcal{F}_{\mathcal{L},\zeta_0}(t)}\rho(\eta)^q.
\end{equation*}
so that
\begin{align*}
    \tau_{\mathcal{L}}(q)=\tau_{\mathcal{L},\zeta_0}(q)&=\lim_{t\to 0}\frac{\log A_q(t)}{\log t}.
\end{align*}
The projection $\pi$ taking paths in $\Omega^*$ to net intervals in $\mathcal{P}$ restricts to the map
\begin{equation*}
    \pi:\Omega^*_{\mathcal{L},\zeta_0}\to\{\Delta\in\mathcal{P}:\Delta\subseteq\Delta_0,\vs(\Delta)\in V(\mathcal{L})\}.
\end{equation*}
We recall that $\rho=\mu\circ\pi$.
As defined in \cref{t:reg-sub}, we also set
\begin{equation*}
    F(\alpha; t,\epsilon) := \bigl\{\eta\in\mathcal{F}_{\mathcal{L},\zeta_0}(t):t^{\alpha+\epsilon}\leq \rho(\eta)\leq t^{\alpha-\epsilon}\bigr\}.
\end{equation*}

We first prove the following standard counting result on the size of the sets $F(\alpha; t,\epsilon)$.
This is essentially the same as, for example, \cite[Lem. 4.1]{ln1999}.
\begin{lemma}\label{l:ft-count}
    Let $\alpha\geq 0$ be arbitrary and $q\in\partial \tau_{\mathcal{L}}^*(\alpha)$.
    Then there exists some $r>0$ such that for all $0<t<r$,
    \begin{equation*}
        \#F(\alpha;t,\epsilon)\leq t^{-\tau_{\mathcal{L}}^*(\alpha)-(1+|q|)\epsilon}.
    \end{equation*}
\end{lemma}
\begin{proof}
    We prove this for $q<0$, but the case $q\geq 0$ follows identically.
    To do this, we bound $A_q(t)$ in two ways for $t$ sufficiently small.
    On one hand,
    \begin{equation*}
        A_q(t)\geq\sum_{\eta\in F(\alpha;t,\epsilon)}\rho(\eta)^q\geq t^{q(\alpha-\epsilon)}\#F(\alpha; t,\epsilon).
    \end{equation*}
    On the other hand, for $t$ sufficiently small (depending on $\epsilon$ and $\Delta_0$), $A_q(t)\leq t^{\tau_{\mathcal{L}}(q)-\epsilon}$.
    Combining these observations, we have
    \begin{equation*}
        \# F(\alpha; t,\epsilon)\leq t^{\tau_{\mathcal{L}}(q)-\epsilon}t^{-q(\alpha-\epsilon)}=t^{-\tau_{\mathcal{L}}^*(\alpha)-(1-q)\epsilon}
    \end{equation*}
    since $q\in\partial \tau_{\mathcal{L}}^*(\alpha)$ so that $\tau_{\mathcal{L}}^*(\alpha)=\alpha q-\tau_{\mathcal{L}}(q)$.
\end{proof}
We now begin the main proof.
\begin{proofref}{t:m-upper-bound}
    To see (i), suppose $x\in \kint_{\mathcal{L}}$ is arbitrary with unique symbolic representation $\gamma=(e_n)_{n=1}^\infty$.
    Let $\zeta\in\Omega^*$ be a prefix of $\gamma$ ending in $\mathcal{L}$.

    By \cref{p:nper-dim}, $\dim_{\loc}(\mu,x)=\underline{\dim}_{\loc}(\rho,\gamma)$, so there exists an increasing sequence $(n_j)_{j=1}^\infty$ such that
    \begin{equation*}
        \underline{\dim}_{\loc}(\rho,x)=\lim_{j\to\infty}\frac{\log \rho(\gamma|n_j)}{\log W(\gamma|n_j)}.
    \end{equation*}
    With $t_j=W(\gamma|n_j)$, since $\gamma|n_j\in\mathcal{F}_{t_j}$, we have for $j$ sufficiently large that $\gamma|n_j\in\Omega^*_{\mathcal{L},\zeta}$ so that
    \begin{equation*}
        \frac{\log \sum_{\eta\in\mathcal{F}_{\mathcal{L},\zeta}(t_j)}\rho(\eta)^q}{\log t_j}\leq q\frac{\log \rho(\gamma|n_j)}{\log t_j}.
    \end{equation*}
    Taking the limit infimum as $j$ goes to infinity yields
    \begin{equation*}
        \tau_{\mathcal{L}}(q)= \tau_{\mathcal{L},\zeta}(q)\leq q\dim_{\loc}(\mu,x)
    \end{equation*}
    where $q\in\R$ is arbitrary.
    It follows that $\dim_{\loc}(\mu,x)\in[\alpha_{\min}(\mathcal{L}),\alpha_{\max}(\mathcal{L})]$.

    We now see (ii).
    Since
    \begin{equation*}
        \kint_{\mathcal{L}}=\bigcup_{\{\Delta\in\mathcal{P}:\vs(\Delta)\in V(\mathcal{L})\}}\Delta\cap \kint_{\mathcal{L}},
    \end{equation*}
    it suffices to show that $\dim_H E_0\leq \tau_{\mathcal{L}}^*(\alpha)$ where
    \begin{equation*}
        E_0:= E_\mu(\alpha; \mathcal{L})\cap\Delta_0
    \end{equation*}
    and $\Delta_0=\pi(\zeta_0)$ is a fixed net interval with neighbour set in $\mathcal{L}$.
    We fix notation as above; in particular, we recall that $\zeta_0$ is the symbolic representation of $\Delta_0$.

    Again we assume $q< 0$; the case $q\geq 0$ follows similarly.
    Fix $\epsilon>0$ and set
    \begin{equation*}
        \mathcal{G}_n=\{\pi(\eta):\eta\in F(\alpha; 2^{-n},\epsilon)\}
    \end{equation*}
    where $\pi(\eta)$ is the net interval with symbolic representation $\eta$.
    By \cref{l:ft-count}, there exists $N=N(\epsilon)$ such that for all $n\geq N$,
    \begin{equation*}
        \#\mathcal{G}_n=\# F(\alpha; 2^{-n},\epsilon)\leq 2^{n(\tau_{\mathcal{L}}^*(\alpha)+(1-q)\epsilon)}.
    \end{equation*}
    Let $\mathcal{G}=\bigcup_{n=N(\epsilon)}^\infty G_n$.

    We first see that $\mathcal{G}$ is a Vitali cover for $E_0$.
    Let $x\in E_0$ be arbitrary.
    Since $x=\pi(\gamma)$ is an interior point, it has an approximation sequence $(n_j)_{j=1}^\infty$.
    Let $m$ be such that any path $\eta$ in $\mathcal{G}$ of length at least $m$ has $W(\eta)\leq 1/3$.
    Such a constant exists since there are only finitely many possible edge weights $W(e)\in(0,1)$.
    The choice of $m$ ensures that there exists some $m_j\in\N$ such that $W(\gamma|n_j)\leq 2^{-m_j}\leq W(\gamma|n_j-m)$.
    Since $\dim_{\loc}(\mu,x)$ exists and $R\cdot W(\gamma|n_j)\asymp 2^{-m_j}$ where $R>0$ is a fixed constant,
    \begin{equation}\label{e:2-bd}
        \lim_{j\to\infty}\frac{\log \mu(B(x,R\cdot W(\gamma|n_j)))}{\log \mu(B(x,2^{-m_j}))}=1.
    \end{equation}

    Now, by \cref{l:ap-sub} and \cref{l:approx-reg} applied to the constant sequence $k_j=m$, we have for $j$ sufficiently large and $0\leq i\leq m$ arbitrary
    \begin{equation*}
    \mu(B(x, R\cdot W(\gamma|n_j)))\leq\rho(\gamma|n_j)\leq \rho(\gamma|n_j-i)\leq \rho(\gamma|n_j)^{1-\epsilon}.
    \end{equation*}
    Moreover, we always have
    \begin{equation*}
        B(x,2^{-m_j})\supseteq B(x, W(\gamma|n_j))\supseteq \pi(\gamma|n_j)
    \end{equation*}
    so that $\rho(\gamma|n_j)\leq \mu(B(x,2^{-m_j}))$.
    Thus applying \cref{e:2-bd}, for all $j$ sufficiently small,
    \begin{equation*}
        \mu(B(x,2^{-m_j}))^{1+\epsilon}\leq \rho(\gamma|n_j-i)\leq\rho(\gamma|n_j)^{1-\epsilon}\leq\mu(B(x,2^{-m_j}))^{1-\epsilon}.
    \end{equation*}
    Finally, since $\dim_{\loc}(\mu,x)=\alpha$, for all $j$ sufficiently small and $0\leq i\leq m$ with $\gamma|(n_j-i)\in\mathcal{F}_{\mathcal{L},\zeta_0}(2^{-m_j})$,
    \begin{equation*}
        (2^{-m_j})^{\alpha+\epsilon}\leq \rho(\gamma|n_j-i)\leq (2^{-m_j})^{\alpha-\epsilon}.
    \end{equation*}
    Thus $\pi(\gamma|(n_j-i))\in\mathcal{G}$.
    Since this is true for all $j$ sufficiently large, we may take $\diam(\gamma|(n_j-i))$ arbitrarily small, so $\mathcal{G}$ is indeed a Vitali cover for $E_0$.

    Now suppose $\{E_i\}_{i=1}^\infty$ is any disjoint subcollection of $\mathcal{G}$: then for $s=\tau_{\mathcal{L}}^*(\alpha)+2(1-q)\epsilon$,
    \begin{align*}
        \sum_{i=1}^\infty\diam(E_i)^s &= \sum_{n=N(\epsilon)}^\infty\sum_{\Delta\in \mathcal{G}_n}\diam(\Delta)^s\leq \sum_{n=N(\epsilon)}^\infty 2^{-ns}\#\mathcal{G}_n\\
                                      &\leq \sum_{n=N(\epsilon)}^\infty\bigl(2^{-\tau_{\mathcal{L}}^*(\alpha)-2(1-q)\epsilon)}2^{\tau_{\mathcal{L}}^*(\alpha)+(1-q)\epsilon}\bigr)^n\\
                                      &= \sum_{n=N(\epsilon)}^\infty (2^{-(1-q)\epsilon})^n<\infty.
    \end{align*}
    Thus by the Vitali covering theorem for Hausdorff measure, we must have
    \begin{equation*}
        \mathcal{H}^s(E_0)\leq\sum_{i=1}^\infty\diam(E_i)^s<\infty
    \end{equation*}
    so that $\dim_H E_0\leq \tau_{\mathcal{L}}^*(\alpha)+2(1-q)\epsilon$.
    Since $\epsilon>0$ was arbitrary, the result follows.
\end{proofref}

\subsection{Irreducibility and the lower bound for the multifractal spectrum}
We recall that the notion of irreducibility was introduced in \cref{sss:irreducibility}.
Moreover, recall that a point $x\in K_{\mathcal{L}}$ is said to be an interior point of $K_{\mathcal{L}}$ if it only has symbolic representations that are eventually in $\mathcal{L}$, and the set of such points is denoted by $\kint_{\mathcal{L}}$.

We now introduce the notion of an interior path, and use this to relate the notions of $\xi$-regularity in $\Omega^\infty$ (introduced in \cref{d:xi-reg}) with regular points in $K$ (as defined in \cref{d:reg-point}).
\begin{definition}
    We say $\xi$ is an \defn{interior path} if whenever $(\Delta_i)_{i=0}^m$ is a sequence of net interals where $\Delta_{i+1}$ is a child of $\Delta_i$ corresponding to $\xi$, there is a neighbourhood of $\Delta_m$ in $K$ which is contained entirely in $\Delta_0$.
\end{definition}
Recall that $\Omega^\infty$ is the set of rooted infinite paths in $\mathcal{G}$ and $K_R$ is the set of regular points.
\begin{lemma}\label{l:xi-reg-proj}
    Let $\xi$ be an interior path in a loop class $\mathcal{L}$, and let $\gamma\in\Omega^\infty_{\mathcal{L}}$ be $\xi$-regular.
    Then for any path $\eta$ such that $\eta\gamma\in\Omega^\infty$, $\pi(\eta\gamma)\in \kint_{\mathcal{L}}\cap K_R$.
\end{lemma}
\begin{proof}
    This is a direct application of the definitions, noting that if $\gamma=(e_n)_{n=1}^\infty$, $\eta$ has length $m_1$, $\xi$ has length $m_2$, and $\xi$ appears at some position $n$, then some $j$ with $n+m_1\leq j\leq n+m_1+m_2$ is a point in the approximation sequence of $\pi(\eta\gamma)$.
\end{proof}

Recall that
\begin{equation*}
    E_\mu(\alpha;\mathcal{L})=\{x \in \kint_{\mathcal{L}}:\dim_{\loc}(\mu,x)=\alpha\}.
\end{equation*}
We will also need the following result, which follows by a similar argument to \cite[Prop. 3.15]{rut2021} or (in a somewhat more specialized case) \cite[Prop. 2.7]{hhn2018}.
\begin{lemma}\label{l:simple}
    Suppose $\mathcal{L}$ is a simple loop class and $x\in K_{\mathcal{L}}$.
    If $x$ is an interior point with $\pi^{-1}(x)=\{\gamma\}$, then
    \begin{equation*}
        \dim_{\loc}(\mu,x)=\dim_{\loc}(\rho,\gamma).
    \end{equation*}
    Otherwise $x\in K$ is a boundary point with $\pi^{-1}(x)=\{\gamma_1,\gamma_2\}$, and
    \begin{equation*}
        \dim_{\loc}(\mu,x)=\min\{\dim_{\loc}(\rho,\gamma_1),\dim_{\loc}(\rho,\gamma_2)\}.
    \end{equation*}
\end{lemma}
We now show here that the regular points in a non-simple $K_{\mathcal{L}}$ are abundant.
\begin{theorem}\label{t:m-lower-bound}
    Let $\mathcal{L}$ be an irreducible loop class which is not simple, or simple and contains an interior point.
    Then $E_\mu(\alpha;\mathcal{L})\neq\emptyset$ if and only if $f_{\mathcal{L}}(\alpha)\geq 0$ if and only if $\alpha\in[\alpha_{\min}(\mathcal{L}),\alpha_{\max}(\mathcal{L})]$.
    Moreover,
    \begin{equation*}
        \dim_H E_\mu(\alpha;\mathcal{L})\cap K_R\geq f_{\mathcal{L}}(\alpha)
    \end{equation*}
    for all $\alpha$.
\end{theorem}
\begin{proof}
    If $\mathcal{L}$ is simple, since $\mathcal{L}$ contains interior points, the result follows directly from \cref{l:simple}.

    Otherwise, $\mathcal{L}$ is not simple, so there exists some vertex $v\in V(\mathcal{L})$ and an interior path $\xi\in\Omega^*(\mathcal{L},v)$.
    Let $\zeta_1\in\Omega^*$ be any path ending at a vertex in $\mathcal{L}$.
    By \cref{t:reg-sub}, get $\Gamma\subseteq E_{\mathcal{L},\zeta_1}(\alpha)$ such that $\dim_H\Gamma\geq\dim_H E_{\mathcal{L},\zeta_1}(\alpha)$ and each $\gamma\in\Gamma$ is $\xi$-regular with $\dim_{\loc}(\rho,\gamma)=\alpha$.

    By \cref{l:xi-reg-proj} and \cref{c:reg-loc-dim}, $\pi(\Gamma)\subseteq E(\mathcal{L},\alpha)\cap K_R$.
    In particular, this proves $E_{\mu}(\mathcal{L};\alpha)\cap K_R$ is non-empty whenever $f_{\mathcal{L}}(\alpha)\geq 0$, and
    \begin{equation*}
        \dim_H E_{\mu}(\mathcal{L};\alpha)\cap K_R\geq\dim_H\pi(\Gamma).
    \end{equation*}
    We also know by \cref{t:multi-f} that $\alpha\in[\alpha_{\min}(\mathcal{L}),\alpha_{\max}(\mathcal{L})]$ if and only if $f_{\mathcal{L}}(\alpha)\geq 0$.
    The remaining implication follows from \cref{t:m-upper-bound}.

    It remains to prove that $\dim_H\pi(\Gamma)=\dim_H(\Gamma)$.
    We recall from \cref{l:pi-Lip} that $\pi$ is Lipschitz, so $\dim_H\pi(\Gamma)\leq\dim_H\Gamma$.
    Conversely, let $\Delta\in\mathcal{P}$ be the net interval with symbolic representation $\zeta_1$ and let $\{U_i\}_{i=1}^\infty$ be some $\epsilon$-cover of $\pi(\Gamma)\subseteq\Delta$.
    Without loss of generality, we may assume $U_i\subseteq\Delta$ for each $i\in\N$.
    Let $t_i=\diam U_i <\epsilon$ and let $b_i$ denote the maximal number of net intervals of generation $t_i$ which intersect $U_i$.
    Note that $b_i\leq 1/[a]+1$ where the diameter of any generation $t$ net interval is at least $at$.
    These net intervals have symbolic representations $\{\zeta_1\eta_{ij}:1\leq j\leq b_i\}$, and the corresponding cylinders $\mathcal{C}=\{[\eta_{ij}]:i\in\N,1\leq j\leq b_i\}$ cover $\Gamma$ and have diameter $W(\eta_{ij})\asymp t_i$.
    Thus there exists some $A>0$ such that $\mathcal{C}$ forms an $A\epsilon$-cover of $\Gamma$.

    It follows that for a suitable constant $c$,
    \begin{equation*}
        \sum_{i=1}^{\infty}\sum_{j=1}^{b_{i}}\left(\diam([\eta_{ij}])\right)^{s}\leq cA^{s}\sum_{i}(\diam(U_{i}))^{s}
    \end{equation*}
    and therefore for each $\epsilon >0$,
    \begin{equation*}
        H_{\epsilon A}^{s}(\Gamma)\leq cA^{s}H_{\epsilon }^{s}(\pi(\Gamma )).
    \end{equation*}
    Letting $\epsilon \rightarrow 0$, we deduce that $H^{s}(\pi(\Gamma ))\geq (cA^{s})^{-1}H^{s}(\Gamma )$.
    This implies $\dim_{H}\pi(\Gamma )\geq \dim_{H}(\Gamma )$, so that $\dim_H\pi(\Gamma)=\dim_H\Gamma$.
\end{proof}
If $\mathcal{L}$ is an irreducible non-simple loop class, then necessarily $\mathcal{L}$ contains an interior path.
The only additional case occurs when $\mathcal{L}$ is a simple loop class without an interior path.
In this case, it may hold that every $x\in K_{\mathcal{L}}$ has two symbolic representations, and the local dimension is always given by the symbolic representation of the adjacent path not eventually in $\mathcal{L}$.
This motivates the following definition.
\begin{definition}\label{d:degen}
    We say that a loop class $\mathcal{L}$ is \defn{non-degenerate} if $\mathcal{L}$ is not simple, or if $\mathcal{L}$ is simple and there exists some $x\in K$ such that
    \begin{equation*}
        \dim_{\loc}(\mu,x)=\dim_{\loc}(\rho,\gamma)
    \end{equation*}
    for some $\gamma\in\Omega^\infty_{\mathcal{L}}$.
    We say that $\mathcal{L}$ is \defn{degenerate} otherwise.
\end{definition}
\begin{corollary}\label{c:m-spectrum}
    Suppose every loop class in $\mathcal{G}$ is irreducible, with non-degenerate loop classes $\mathcal{L}_1,\ldots,\mathcal{L}_m$.
    Then the multifractal spectrum of $\mu$ is given by
    \begin{equation*}
        f_\mu(\alpha)=\max\{f_{\mathcal{L}_1}(\alpha),\ldots,f_{\mathcal{L}_m}(\alpha)\}
    \end{equation*}
    for each $\alpha\in\R$.
\end{corollary}
\begin{proof}
    Combining the general upper bound from \cref{t:m-upper-bound} and the lower bound \cref{t:m-lower-bound} using irreducibility, it follows for each $1\leq i\leq m$ that
    \begin{equation*}
        \dim_H E_\mu(\alpha;\mathcal{L}_i)=f_{\mathcal{L}_i}(\alpha).
    \end{equation*}
    Of course,
    \begin{equation*}
        \bigcup_{i=1}^m E_\mu(\alpha;\mathcal{L}_i)\supseteq E_\mu(\alpha)\cap\kint.
    \end{equation*}
    Moreover, if $x\notin\kint$, by \cref{l:simple}, then $\dim_{\loc}(\mu,x)=\dim_{\loc}(\rho,\gamma)$ for some infinite path $\gamma\in\Omega^\infty_{\mathcal{L}}$.
    Then this $\mathcal{L}$ is non-degenerate, and $K\setminus\kint$ is countable and hence has Hausdorff dimension 0.
    Thus
    \begin{equation*}
        f_\mu(\alpha)=\dim_H E_\mu(\alpha)=\dim_H E_\mu(\alpha)\cap\kint
    \end{equation*}
    as required.
\end{proof}
\subsection{Decomposability and bounds for the \texorpdfstring{$L^q$}{Lq}-spectrum}
Recall that the notion of decomposability was introduced in \cref{sss:decomposable}.

Similarly to how we bounded the multifractal formalism $f_\mu$ in terms of the functions $f_{\mathcal{L}}$ for loop classes $\mathcal{L}$, in this section, we establish bounds for the $L^q$-spectrum $\tau_{\mu}$ in terms of the functions $\tau_{\mathcal{L}}$.
We first note the following general upper bound.
\begin{lemma}\label{l:lq-upper-bound}
    Let $\mu$ be a self-similar measure satisfying the $\Phi$-FNC, with loop classes $\mathcal{L}_1,\ldots,\mathcal{L}_m$.
    Then
    \begin{equation*}
        \tau_\mu(q)\leq\limsup_{t\to 0}\frac{\log\sup\sum_i\mu(B(x_i,t))^q}{\log t}\leq \min\{\tau_{\mathcal{L}_1}(q),\ldots,\tau_{\mathcal{L}_m}(q)\}
    \end{equation*}
    where the supremum is taken over all centred packings $\{B(x_i,t)\}_i$ of $K=\supp\mu$.
\end{lemma}
\begin{proof}
    The first inequality follows by definition.

    To see the second inequality, let $\mathcal{L}$ be an arbitrary loop class.
    Let $\zeta\in\Omega^*$ be a path ending at a vertex in $\mathcal{L}$.
    Then by definition
    \begin{equation*}
        \sum_{\eta\in\mathcal{F}(t)}\rho(\eta)^q\geq\sum_{\eta\in\mathcal{F}_{\mathcal{L},\zeta}(t)}\rho(\eta)^q.
    \end{equation*}
    Now the same proof as \cref{p:lq-lim} shows that
    \begin{equation*}
        \limsup_{t\to 0}\frac{\log\sup\sum_i\mu(B(x_i,t))^q}{\log t}=\limsup_{t\to 0}\frac{\log\sum_{\eta\in\mathcal{F}(t)}\rho(\eta)^q}{\log t}\leq\tau_{\mathcal{L},\zeta}(q)=\tau_{\mathcal{L}}(q)
    \end{equation*}
    by existence of the limit defining $\tau_{\mathcal{L},\zeta}$ given in \cref{l:lq-limit}.
    But $\mathcal{L}$ was arbitrary, so the result follows.
\end{proof}
We now have the following result establishing our lower bound as well.
Note the similarity of this result and proof to \cite[Thm. 5.2]{hhstoappear}.
\begin{theorem}\label{t:lq-lower-bound}
    Let $\mu$ be a self-similar satisfying the $\Phi$-FNC with decomposable transition graph $\mathcal{G}$.
    Let $\mathcal{G}$ have loop classes $\mathcal{L}_1,\ldots,\mathcal{L}_m$.
    Then
    \begin{equation*}
        \tau_\mu(q)=\min\{\tau_{\mathcal{L}_1}(q),\ldots,\tau_{\mathcal{L}_m}(q)\}.
    \end{equation*}
    for any $q\in\R$.
    Moreover, the limit defining $\tau_\mu(q)$ exists for any $q\in\R$.
\end{theorem}
\begin{proof}
    For each loop class $\mathcal{L}_i$, fix a path $\zeta_i\in\Omega^*$ ending at a vertex $v_i\in V(\mathcal{L}_i)$.
    Now for each vertex $w\in V(\mathcal{L}_i)$, let $\gamma_{i,w}$ be a path in $\mathcal{L}_i$ from $v_i$ to $w$.
    Let $s_0>0$ be such that
    \begin{equation*}
        s_0^{1/m}\leq\min_{i}\min_{w\in\mathcal{L}_i}W(\zeta_i\gamma_{i,w}).
    \end{equation*}
    Similarly, since there are only finitely many initial and transition paths, there is $s_1>0$ such that if $\eta\in\mathcal{F}(t)$ has decomposition $(\lambda_1,\ldots,\lambda_n)$, then
    \begin{equation*}
        s_1 t\geq W(\lambda_1)\cdots W(\lambda_n).
    \end{equation*}
    Next, define sets of path weights
    \begin{align*}
        \Lambda_i&:=\{W(\eta):\eta\in\mathcal{F}_{\mathcal{L}_i,\zeta_i}\}\\
        \Lambda(t) &:= \{(t_1,\ldots,t_m)\in\Lambda_1\times\cdots\times\Lambda_m: s_1 t\geq t_1\cdots t_m\geq s_0 t\}.
    \end{align*}
    Since there are only finitely many edge weights $W(e)$ for $e\in E(\mathcal{G})$, it follows that there is some $k\in\N$ such that $\#\Lambda(t)\leq(-\log t)^k$ for all $t$ sufficiently small.

    We now construct a function
    \begin{equation}\label{e:Psi-def}
        \Psi:\mathcal{F}(t)\to\bigcup_{(t_1,\ldots,t_m)\in\Lambda(t)}\mathcal{F}_{\mathcal{L}_1,\zeta_1}(t_1)\times\cdots\times\mathcal{F}_{\mathcal{L}_m,\zeta_m}(t_m)
    \end{equation}
    as follows.
    Suppose the path $\eta\in\mathcal{F}(t)$ has decomposition $(\lambda_1,\ldots,\lambda_m)$.
    Then if the path $\lambda_i$ begins at vertex $w_i\in V(\mathcal{L}_i)$, we set
    \begin{equation*}
        \Psi(\eta)=(\zeta_1\gamma_{1,w_1}\lambda_1,\ldots,\zeta_m\gamma_{m,w_m}\lambda_m).
    \end{equation*}
    Note that $\Psi$ is well-defined by choice of $s_0$ and the definition of $\Lambda(t)$.

    Since there are only finitely many transition paths, there is a uniform bound on the number of paths with the same decomposition.
    Moreover, since there are only finitely many paths $\gamma_{i,w_i}$, for a fixed path $\eta$, the number of distinct decompositions of paths $\eta'$ with $\Psi(\eta)=\Psi(\eta')$ is also uniformly bounded.
    Thus, even though $\Psi$ need not be injective, there is some constant $N\in\N$ (independent of $t$) such that each fibre of $\Psi$ has cardinality at most $N$.

    Fix
    \begin{equation*}
        \theta(q):=\min\{\tau_{\mathcal{L}_1}(q),\ldots,\tau_{\mathcal{L}_m}(q)\}.
    \end{equation*}
    By \cref{l:lq-limit}, for any $\epsilon>0$ and all $t$ sufficiently small,
    \begin{equation*}
        \sum_{\eta_i\in\mathcal{F}_{\mathcal{L}_i,\zeta_i}(t)}\norm{T(\eta_i)}^q\leq t^{\tau_{\mathcal{L}_i}(q)-\epsilon}\leq t^{\theta(q)-\epsilon}.
    \end{equation*}
    Moreover, by the decomposability assumption and \cref{l:left-prod}, it follows that if $\Psi(\eta)=(\eta_1,\ldots,\eta_m)$, then
    \begin{equation*}
        \norm{T(\eta)}^q\preccurlyeq_q\norm{T(\eta_1)}^q\cdots\norm{T(\eta_m)}^q.
    \end{equation*}
    Thus for all $t$ sufficiently small,
    \begin{align*}
        \sum_{\eta\in\mathcal{F}(t)}\rho(\eta)^q&\preccurlyeq_q\sum_{(t_1,\ldots,t_m)\in\Lambda(t)}\left(\sum_{\eta_1\in\mathcal{F}_{\mathcal{L}_1,\zeta_1}(t_1)}\cdots\sum_{\eta_m\in\mathcal{F}_{\mathcal{L}_m,\zeta_m}(t_m)}\norm{T(\eta_1)}^q\cdots\norm{T(\eta_m)}^q\right)\\
                                                &= \sum_{(t_1,\ldots,t_m)\in\Lambda(t)}\left(\sum_{\eta_1\in\mathcal{F}_{\mathcal{L}_1,\zeta_1}(t_1)}\norm{T(\eta_1)}^q\right)\cdots\left(\sum_{\eta_m\in\mathcal{F}_{\mathcal{L}_m,\zeta_m}(t_m)}\norm{T(\eta_m)}^q\right)\\
                                                &\preccurlyeq_q \sum_{(t_1,\ldots,t_m)\in\Lambda(t)} t_1^{\theta(q)-\epsilon}\cdots t_m^{\theta(q)-\epsilon}\\
                                                &\preccurlyeq_q\#\Lambda(t) t^{\theta(q)-\epsilon}.
    \end{align*}
    Since $\#\Lambda(t)$ grows polynomially in $\log t$, it follows by \cref{p:lq-lim} that
    \begin{equation*}
        \tau_\mu(q)\geq\liminf_{t\to 0}\frac{\log\sum_{\eta\in\mathcal{F}(t)}\rho(\eta)^q}{\log t}\geq\theta(q)-\epsilon.
    \end{equation*}
    But $\epsilon>0$ was arbitrary, and combining this with \cref{l:lq-upper-bound} yields the desired result.
\end{proof}
\begin{remark}\label{r:q-pos-min}
    In fact, since for any path $\eta$ with decomposition $(\lambda_1,\ldots,\lambda_m)$, we have
    \begin{equation*}
        \norm{T(\eta)}\preccurlyeq\norm{T(\lambda_1)}\cdots\norm{T(\lambda_m)}
    \end{equation*}
    with no assumptions on the transition graph $\mathcal{G}$, the same proof as above shows that
    \begin{equation*}
        \tau_{\mathcal{L}}(q)=\tau_\mu(q)=\min\{\tau_{\mathcal{L}_1}(q),\ldots,\tau_{\mathcal{L}_m}(q)\}
    \end{equation*}
    by \cref{p:ess-formula} for $q\geq 0$ without the decomposability assumption.
\end{remark}
\begin{remark}
    Unlike the results for the multifractal formalism in \cref{c:m-spectrum}, we note that \cref{l:lq-upper-bound} and \cref{t:lq-lower-bound} write the $L^q$-spectrum in terms of \emph{all} loop classes, and not just the non-degenerate loop classes.
\end{remark}

\section{Applications and examples}\label{s:multi-examples}
Throughout this section, naturally, $(S_i,p_i)_{i\in\mathcal{I}}$ is a WIFS satisfying the finite neighbour condition with respect to the iteration rule $\Phi$, and has transition graph $\mathcal{G}$ and associated self-similar measure $\mu$.
\subsection{Consequences of the main results}\label{ss:cons}
Our first application, which follows essentially from the bounds in the previous section along with standard properties of concave functions, describes precisely when the multifractal formalism holds.
\begin{corollary}\label{c:multi-validity}
    Suppose $\mathcal{G}$ is irreducible and decomposable, and suppose the maximal loop classes $\mathcal{L}_1,\ldots,\mathcal{L}_m$ are non-degenerate.
    Then $\mu$ satisfies the multifractal formalism at $\alpha$ if and only if $\alpha\in\partial\tau_{\mathcal{L}_i}(q)$ for some $1\leq i\leq m$ and $q\in\R$ with $\min\{\tau_{\mathcal{L}_1}(q),\ldots,\tau_{\mathcal{L}_m}(q)\}=\tau_{\mathcal{L}_i}(q)$.
    In particular, if the derivative $\alpha=\tau_\mu'(q)$ exists at some $q\in\R$, then $\mu$ satisfies the multifractal formalism at $\alpha$.
\end{corollary}
\begin{proof}
    Since $\mathcal{G}$ is decomposable, $\tau_\mu=\min\{\tau_{\mathcal{L}_1}(q),\ldots,\tau_{\mathcal{L}_m}\}$ by \cref{t:lq-lower-bound}.

    First suppose $f_\mu(\alpha)=\tau_\mu^*(\alpha)$, so there is some $\mathcal{L}_i$ such that
    \begin{equation*}
        \tau_\mu^*(\alpha)=f_{\mathcal{L}_i}(\alpha)=\tau_{\mathcal{L}_i}^*(\alpha)
    \end{equation*}
    by \cref{c:m-spectrum} and \cref{t:multi-f}.
    Since $\tau_\mu^*(\alpha)=\tau_{\mathcal{L}_i}^*(\alpha)$, there are $q_1,q_2\in\R$ such that $\alpha\in\partial\tau_{\mathcal{L}_i}(q_1)\cap\partial\tau_\mu(q_2)$: therefore, $\tau_{\mathcal{L}_i}(q_1)-\tau_\mu(q_2)=\alpha(q_1-q_2)$.
    Without loss of generality, suppose $q_1<q_2$.
    Since $\tau_\mu(q_1)\leq \tau_{\mathcal{L}_i}(q_1)$ and $\tau_{\mathcal{L}_i}(q_1)\leq\tau_{\mathcal{L}_i}(q_2)-(q_2-q_1)\alpha$ by concavity, this can only happen when $\tau_{\mathcal{L}_i}(q_1)=\tau_\mu(q_1)$ and $\alpha\in\partial\tau_{\mathcal{L}_i}(q_1)$, as required.

    Conversely, suppose $\alpha\in\partial\tau_{\mathcal{L}_i}(q)$ where $\tau_{\mathcal{L}_i}(q)=\tau_\mu(q)$.
    Since $\tau_\mu\leq\tau_{\mathcal{L}_i}$, it follows that $\alpha\in\partial\tau_\mu(q)$ so that $\tau_\mu^*(\alpha)=\tau_{\mathcal{L}_i}^*(\alpha)$.
    But $\tau_{\mathcal{L}_i}(q)\leq\min\{\tau_{\mathcal{L}_1}(q),\ldots,\tau_{\mathcal{L}_m}(q)\}$ by assumption so
    \begin{equation*}
        \tau_\mu^*(\alpha)=\tau_{\mathcal{L}_i}^*(\alpha)=\max\{\tau_{\mathcal{L}_1}^*(\alpha),\ldots,\tau_{\mathcal{L}_m}^*(\alpha)\}=f_\mu(\alpha)
    \end{equation*}
    by \cref{c:m-spectrum}.

    If $\tau_\mu'(q)$ exists, it follows immediately that $\alpha\in\partial\tau_{\mathcal{L}_i}(q)$ for any $i$ such that $\tau_{\mathcal{L}_i}(q)=\min\{\tau_{\mathcal{L}_1}(q),\ldots,\tau_{\mathcal{L}_m}(q)\}$.
\end{proof}
Our next result was obtained in \cite{rut2021} under the weak separation condition, but we obtain it here (in a slightly more specialized case) as a direct corollary of the prior results.
\begin{corollary}\label{c:one-loop}
    Suppose $\mathcal{G}$ has exactly one loop class $\mathcal{L}$.
    Then $\mu$ satisfies the multifractal formalism.
\end{corollary}
\begin{proof}
    Since $\mathcal{L}$ is the only loop class, it must be essential, so it is irreducible by \cref{l:ess-irred}.
    Since there is only one loop class and therefore no transition paths, the decomposability condition holds vacuously, and by \cref{t:lq-lower-bound}, $\tau_\mu=\tau_{\mathcal{L}}$.
    Thus the result follows from the multifractal formalism for irreducible graph-directed systems proven in \cref{t:multi-f}.
\end{proof}
We now prove the following result, which completely characterizes the validity of the multifractal formalism in terms of a qualitative property of the multifractal spectrum.
\begin{corollary}\label{c:ir-de}
    Suppose the transition graph $\mathcal{G}$ is irreducible and decomposable, and every loop class is non-degenerate.
    Then $\mu$ satisfies the multifractal formalism if and only if $f_\mu$ is a concave function.
\end{corollary}
\begin{proof}
    If $\mu$ satisfies the multifractal formalism, then $f_\mu=\tau_\mu^*$ where $\tau_\mu^*$ is a concave function.

    Conversely, suppose $f_\mu$ is a concave function.
    We have by \cref{t:lq-lower-bound} and \cref{c:m-spectrum} that
    \begin{equation}
        f_\mu=\max\{f_{\mathcal{L}_1},\ldots,f_{\mathcal{L}_m}\}\text{ and } \tau_{\mu}=\min\{\tau_{\mathcal{L}_1},\ldots,\tau_{\mathcal{L}_m}\}.\label{e:taumu}
    \end{equation}
    where $\mathcal{G}$ has loop classes $\mathcal{L}_1,\ldots,\mathcal{L}_m$.

    Now let $\alpha_0\in\R$ be arbitrary.
    Let $q$ be the unique value such that $\alpha_0\in\partial\tau_\mu(q)=[\alpha_1,\alpha_2]$.
    If $\alpha_1=\alpha_2$, $\tau_\mu$ is differentiable at $q$ and we are done by \cref{c:multi-validity}.
    Otherwise, by \cref{e:taumu}, there exist two loop classes, say $\mathcal{L}_1$ and $\mathcal{L}_2$, such that $\alpha_1\in\partial\tau_{\mathcal{L}_1}(q)$, $\alpha_2\in\partial\tau_{\mathcal{L}_2}(q)$ and $\tau_{\mathcal{L}_1}(q)=\tau_{\mathcal{L}_2}(q)=\tau_\mu(q)$.
    Observe that $\tau_\mu^*(\alpha)=\alpha q-\tau_\mu(q)$ for any $\alpha\in[\alpha_1,\alpha_2]$.
    Moreover, since concave conjugation is order reversing, by \cref{e:taumu}, $f_\mu(\alpha_1)=f_{\mathcal{L}_1}(\alpha_1)=\tau_\mu^*(\alpha_1)$ and $f_\mu(\alpha_2)=f_{\mathcal{L}_2}(\alpha_2)=\tau_\mu^*(\alpha_2)$.
    But $f_\mu(\alpha_0)\leq\tau_\mu^*(\alpha_0)$ and $f_\mu$ is concave by assumption, forcing $\tau_\mu^*(\alpha_0)=f_\mu(\alpha_0)$ as required.
\end{proof}
\begin{remark}
    For IFS of the form $(\lambda x+d_i)_{i\in\mathcal{I}}$ satisfying the finite type condition, the following version of the reverse implication was first observed in \cite[Rem. 5.3]{fen2009}: if $\tau_\mu=\tau_{\mathcal{L}}$ for an essential loop class $\mathcal{L}$, then $\mu$ satisfies the multifractal formalism.
    This result follows for any IFS satisfying the $\Phi$-FNC by combining \cref{p:ess-formula}, the fact that the essential loop class is always irreducible, and the general upper bound $f_\mu\leq\tau_\mu^*$,
\end{remark}
A sufficient condition for the measure $\mu$ to fail the multifractal formalism is for the set of attainable local dimensions of $\mu$ to not be a closed interval.
In general, this condition is not necessary.
However, in certain situations, we can determine that it is necessary and sufficient.
\begin{corollary}\label{c:loc-dim-set}
    Suppose the transition graph $\mathcal{G}$ is decomposable.
    Suppose in addition that every non-essential loop class is simple and non-degenerate.
    Then $\mu$ satisfies the multifractal formalism if and only if the set of local dimensions
    \begin{equation*}
        \{\dim_{\loc}(\mu,x):x\in K\}
    \end{equation*}
    is a closed interval.
\end{corollary}
\begin{proof}
    The forward direction is immediate.

    Conversely, denote the loop classes by $\{\mathcal{L}_1,\ldots,\mathcal{L}_m\}$.
    If $\mathcal{L}$ is any simple loop class, then $f_{\mathcal{L}}(\alpha)=0$ for precisely one value of $\alpha$, and is $-\infty$ otherwise.
    Since the essential loop class and any simple loop class is irreducible, by \cref{r:ess-unique}, we have
    \begin{equation*}
        f_{\mathcal{L}}(\alpha)=\max\{f_{\mathcal{L}_1}(\alpha),\ldots,f_{\mathcal{L}_m}(\alpha)\}=f_\mu(\alpha).
    \end{equation*}
    Thus the result follows by \cref{c:ir-de}.
\end{proof}

\subsection{A family of examples of Testud}\label{ss:tes-ex}
Let $\ell\geq 2$ be a positive integer.
Let $P,N\subseteq\{0,1,\ldots,\ell-1\}$ where $\{0,\ell-1\}\subseteq P\cup N$.
Let $\mathcal{I}=P\times\{1\}\cup N\times\{-1\}$ and for $(i,\pm 1)\in\mathcal{I}$, define
\begin{align*}
    S_{(i,1)}(x) &= \frac{x}{\ell}+\frac{i}{\ell} & S_{(i,-1)}(x) &= -\frac{x}{\ell}+\frac{i+1}{\ell}.
\end{align*}
In this subsection, we study the multifractal theory of the IFS $\{S_{\underline{i}}\}_{\underline{i}\in\mathcal{I}}$.
This family of IFS was studied in \cite{tes2006a} and \cite{os2008} under the assumption that $P=\{0,1,\ldots,\ell-1\}$.
We do not require this assumption in our analysis.

Fix the iteration rule $\Phi$ from \cref{ex:uniform-transition}.
Write $V=\{v_{1},v_{-1},v_{\pm 1}\}$ where $v_{1}=\{x\mapsto x\}$, $v_{-1}=\{x\mapsto -x+1\}$ and $v_{\pm 1}=v_{1}\cup v_{-1}$.
Since the images $S_{(i,\pm 1)}((0,1))$ are either disjoint or coincide exactly,
\begin{equation*}
    \mathcal{P}_n = \{S_\sigma([0,1]):\sigma\in\mathcal{I}^n\}.
\end{equation*}
In particular, if $\Delta\in\mathcal{P}$ is any net interval, then $\vs(\Delta)\in V$.
Thus $(S_{\underline{i}})_{\underline{i}\in\mathcal{I}}$ satisfies the $\Phi$-FNC.
Note that $v_1=\vroot\in V(\mathcal{G})$.

If $P\cap N=\emptyset$, then the IFS satisfies the open set condition with respect to the open interval $(0,1)$.
Otherwise, there exists some index $i$ such that $(i,1)$ and $(i,-1)$ are both in $\mathcal{I}$, so that $v_{\pm 1}$ is a neighbour set in $V$.
For the remainder of this section, we will assume that this is the case.
The open set condition may hold even when $P\cap N\neq\emptyset$ with respect to an open set that is not an interval (take, for example, $\ell=4$, $P=\{0,1,3\}$, and $N=\{1\}$), but for simplicity we omit this discussion.

\subsubsection{Properties of the transition graph}
We begin with a description of the transition graph $\mathcal{G}$.
\begin{proposition}
    Suppose $P\cap N\neq\emptyset$.
    There is a unique essential loop class $\mathcal{G}_{\ess}$, and $v_{\pm 1}\in V(\mathcal{G}_{\ess})$.
    Moreover, exactly one of the following holds:
    \begin{enumerate}[nl,r]
        \item We have $P=N$.
            Then $\mathcal{G}_{\ess}$ is the only loop class and $V(\mathcal{G}_{\ess})=\{v_{\pm 1}\}$.
        \item There is some $i$ such that $i\in P\setminus N$ and $\ell-1-i\notin P$ or $i\in N\setminus P$ and $i,\ell-1-i\notin N$.
            Then $\mathcal{G}_{\ess}$ is the only loop class and $V(\mathcal{G}_{\ess})=V(\mathcal{G})=\{v_{1},v_{-1},v_{\pm 1}\}$.
        \item Otherwise, there is exactly one non-essential loop class $\mathcal{L}$.
            In this case, if $P\setminus N\neq\emptyset$, then $v_1\in V(\mathcal{L})$, and if $N\setminus P\neq\emptyset$, then $v_{-1}\in V(\mathcal{L})$.
    \end{enumerate}
\end{proposition}
\begin{proof}
    If $i\in P\cap N$, then $\{(i,1),(i,-1)\}\subset\mathcal{I}$ so that $S_{(i,1)}([0,1])=S_{(i,-1)}([0,1])=\Delta$ is a net interval with neighbour set $v_{\pm 1}$.
    This neighbour set is essential since if $\Delta=S_\sigma([0,1])$ is any net interval, then $S_{\sigma(i,1)}([0,1])$ is a net interval with neighbour set $v_{\pm 1}$.

    It is clear that exactly one of the conditions must hold.
    We verify corresponding properties of the transition graph $\mathcal{G}$.
    \begin{enumerate}[nl,r]
        \item If $P=N$, then for any net interval in $\mathcal{P}_1$, we see that $\vs(\Delta)=v_{\pm 1}$.
            Thus every outgoing edge from $\vroot$ ends at the vertex $v_{\pm 1}$.
        \item Suppose there is some $i\in P\setminus N$ with $\ell-1-i\notin N$.
            Let $\Delta=S_\sigma([0,1])\in\mathcal{P}_n$ have $\vs(\Delta)=v_{\pm 1}$.
            Then $S_{\sigma(i,1)}([0,1])$ is a net interval with $\vs(\Delta)=v_1$ and $S_{\sigma(\ell-1-i,-1)}([0,1])$ is a net interval with $\vs(\Delta)=v_{-1}$.

            The other case follows similarly.
        \item Finally, suppose (i) and (ii) do not hold.
            Let $S_\sigma([0,1])=\Delta$ be a net interval with $\vs(\Delta)=v_{\pm 1}$, and suppose $r_\sigma>0$, and $\tau$ has $r_\tau<0$ and $S_\tau([0,1])=\Delta$ as well.
            Suppose $i\in P$ so that $\Delta'=S_{\sigma(i,1)}([0,1])$ is a child of $\Delta$.
            Then negating the condition (ii), we either have $i\in N$ (and $\Delta'$ has neighbours generated by $\sigma(i,1)$ and $\sigma(i,-1)$) or $\ell-1-i\in P$ (and $\Delta'$ has neighbours generated by $\sigma(i,1)$ and $\tau(i,1)$), so $\Delta'$ has neighbour set $v_{\pm 1}$.
            The other case $i\in N$, or the cases where $\Delta'=S_{\tau(i,\pm 1)}([0,1])$, follow similarly.
            Thus $V(\mathcal{G}_{\ess})=\{v_{\pm 1}\}$.

            Since $P\neq N$, if $P\setminus N\neq\emptyset$, there is an edge from $\vroot=v_1$ to $v_1$ and if $N\setminus P\neq\emptyset$, there are edges from $v_1$ to $v_{-1}$ and $v_{-1}$ to $v_1$.
            Thus the claim follows.
    \end{enumerate}
\end{proof}
We can now observe the following result.
\begin{lemma}\label{l:tes-ir-dec}
    With any choice of probabilities, the transition graph $\mathcal{G}$ is irreducible and decomposable.
\end{lemma}
\begin{proof}
    The essential loop class $\mathcal{G}_{\ess}$ is always irreducible by \cref{l:ess-irred}.
    If there is a loop class $\mathcal{L}$, we observe that either $V(\mathcal{L})$ consists of a single vertex or $V(\mathcal{L})=\{v_1,v_{-1}\}$ and there are edges joining $v_1$ and $v_{-1}$ and $v_{-1}$ and $v_1$.

    Since the neighbour sets $v_1$ and $v_{-1}$ have cardinality one, irreducibility follows by \cref{l:irred}.
    Decomposability follows directly from \cref{l:size-one-loops}.
\end{proof}
\subsubsection{Multifractal properties of associated measures}
We can compute formulas for the loop class $L^q$-spectra.
\begin{proposition}
    \begin{enumerate}[nl,r]
        \item Suppose $\mathcal{L}$ is a non-essential loop class.
            Let $\mathcal{J}=(P\setminus N)\cup (N\setminus P)$, and for $j\in\mathcal{J}$, write $p_j=p_{(j,1)}$ if $j\in P\setminus N$, and $p_j=p_{(j,-1)}$ if $j\in N\setminus P$.
            Then
            \begin{equation*}
                \tau_{\mathcal{L}}(q)=\frac{\log\sum_{j\in\mathcal{J}}p_j^q}{-\log \ell}.
            \end{equation*}
        \item Let $T(x)=1-x$.
            Then with $\nu=\mu+\mu\circ T$,
            \begin{equation*}
                \tau_{\mathcal{G}_{\ess}}(q)=\tau_\nu(q).
            \end{equation*}
    \end{enumerate}
\end{proposition}
\begin{proof}
    \begin{enumerate}[nl,r]
        \item Observe that there is a bijection between paths in $\Omega^n_{\mathcal{L}}$ and words in $\mathcal{J}^n$.
            Moreover, if $\eta\in\Omega^n_{\mathcal{L}}$ has corresponding sequence $(j_1,\ldots,j_n)\in\mathcal{J}^n$, a direct computation gives that $\rho(\eta)=\norm{T(\eta)}=p_{j_1}\cdots p_{j_n}$.
            Thus since $\vroot=v_1\in V(\mathcal{L})$,
            \begin{align*}
                \tau_{\mathcal{L}}(q)=\tau_{\mathcal{L},\emptyset}(q)&= \lim_{n\to\infty}\frac{\log\sum_{(j_1,\ldots,j_n)\in\mathcal{J}^n}p_{j_1}\cdots p_{j_n}}{-n\log\ell}\\
                                                                     &= \lim_{n\to\infty}\frac{\log\left(\sum_{j\in\mathcal{J}}p_j\right)^n}{-n\log\ell}\\
                                                                     &= \frac{\log \sum_{j\in\mathcal{J}}p_j^q}{-\log\ell}
            \end{align*}
            as claimed.

        \item Let $\Delta\in\mathcal{P}_1$ have $\vs(\Delta)=v_{\pm 1}$.
            By \cref{p:ess-formula}, $\tau_{\mathcal{G}_{\ess}}(q)=\tau_{\mu|_\Delta}(q)$.
            But for any Borel set $E\subseteq\Delta$, we have by \cref{e:qi-formula} since $\vs(\Delta)=\{\id, T\}$
            \begin{align*}
                \mu(E) &= \mu(E)p_{(i,1)}+\mu\circ T(E) p_{(i,-1)}\asymp \nu(E).
            \end{align*}
            Thus $\tau_{\mathcal{G}_{\ess}}(q)=\tau_\nu(q)$ for any $q\in\R$.
    \end{enumerate}
\end{proof}
We now observe the following conclusion.
\begin{theorem}
    If there is no non-essential loop class, then $\mu$ satisfies the multifractal formalism.
    Otherwise, there is a single non-essential loop class $\mathcal{L}$.
    Then
    \begin{align*}
        \tau_\mu(q)&=\min\{\tau_{\mathcal{L}}(q),\tau_{\mathcal{G}_{\ess}}(q)\}\\
        f_\mu(\alpha) &= \max\{\tau_{\mathcal{L}}^*(\alpha),\tau_{\mathcal{G}_{\ess}}^*(\alpha)\}.
    \end{align*}
\end{theorem}
\begin{proof}
    This follows directly from \cref{l:tes-ir-dec} by the general results \cref{c:m-spectrum} and \cref{t:lq-lower-bound}.
\end{proof}
\subsection{Bernoulli convolutions with Pisot contractions}\label{ss:bconv-Pisot}
\subsubsection{Simple Pisot contractions}
A simple Pisot number is the unique positive real root of a polynomial
\begin{equation*}
    p_k(x)=x^k-x^{k-1}-\cdots-x-1
\end{equation*}
for some $k\geq 2$.
We denote this number by $r_k$.
Naturally, $r_k$ is a \emph{Pisot number}, which is a real algebraic number strictly greater than $1$ with Galois conjugates having modulus strictly less than 1.
Note that $r_2=(\sqrt{5}+1)/2$ is the Golden ratio, and $1<r_2<r_3<r_4<\cdots<2$.

We are interested in the (possibly biased) Bernoulli convolution associated with the parameter $\lambda=1/r_k$, which we view as a self-similar measure associated with the IFS
\begin{align*}
    S_1(x)&=\lambda x & S_2(x) &= \lambda x+(1-\lambda)
\end{align*}
and probabilities $p_1,p_2>0$ with $p_1+p_2=1$.
It is known (since at least \cite{nw2001}) that the IFS $(S_i)_{i=1,2}$ satisfies the finite type condition, and thus satisfies the finite neighbour condition with respect to the iteration rule from \cref{ex:uniform-transition}.

In \cite{fen2005}, Feng proved, with probabilities $p_1=p_2=1/2$, that the associated self-similar measure satisfies the multifractal formalism.
Here, we show how this result can be obtained as a special case of our general results.

Fix any probabilities $p_1,p_2>0$ with $p_1+p_2=1$.
We first obtain basic results on the structure of the transition graph $\mathcal{G}$ and some information on sets of local dimensions.
\begin{proposition}\label{p:s-b-set}
    The transition graph $\mathcal{G}$ has a unique essential loop class $\mathcal{G}_{\ess}$ and two non-essential simple loop classes $\mathcal{L}_1$ and $\mathcal{L}_2$.
    Both loop classes $\mathcal{L}_1$ and $\mathcal{L}_2$ have a single vertex which is a neighbour set which has cardinality one.

    Each $\Omega_{\mathcal{L}_i}^\infty$ consists of a single path $\gamma_i$, where $\pi(\gamma_1)=0$ and $\pi(\gamma_2)=1$, and
    \begin{equation}\label{e:loc-dim-formula}
        \begin{aligned}
            \dim_{\loc}(\mu,0)&=\dim_{\loc}(\rho,\gamma_1)=\frac{\log p_1}{\log\lambda}\\
            \dim_{\loc}(\mu,1) &= \dim_{\loc}(\rho,\gamma_2)=\frac{\log p_2}{\log\lambda}.
        \end{aligned}
    \end{equation}
    Moreover, there exists a $\gamma\in\Omega^\infty_{\mathcal{G}_{\ess}}$ such that
    \begin{equation*}
        \dim_{\loc}(\mu,\pi(\gamma))=\frac{\log p_1p_2}{2\log \lambda}.
    \end{equation*}
\end{proposition}
\begin{proof}
    We will assume that $k\geq 3$; the case $k=2$ is similar, but easier (in fact, full details of the computation are given in \cite[Sec. 5.1]{hr2021}).

    By a direct computation, the part of the graph $\mathcal{G}$ spanned by $\Omega^2$ is given in \cref{f:gm-graph}, along with the net intervals in $\mathcal{P}_2$ drawn in \cref{f:netiv-diag}.
    The net intervals labelled $\Delta_i$ for $i=1,2,3$ have neighbour sets $\vs(\Delta_i)=v_i$, which are the labelled vertices in the partial transition graph.

    We can now see that the leftmost child of $\Delta_1$ has neighbour set $v_2$, and the corresponding edge $e_{12}$ has $T(e_{12})=\begin{pmatrix}a p_1\end{pmatrix}$ and $W(e_{12})=\lambda$ for some constant $a>0$.
    Similarly, there is an edge $e_{21}$ from $v_2$ to $v_1$ with $T(e_{21})=\begin{pmatrix}a^{-1} p_2\end{pmatrix}$ and $W(e_{21})=\lambda$.
    From here, a straightforward induction argument (using the fact that $\lambda^k+\lambda^{k-1}+\cdots+\lambda-1=0$) yields that, in fact, $v_1,v_2,v_3$ are vertices in a unique essential loop class $\mathcal{G}_{\ess}$, and the cycles labelled as $\mathcal{L}_1$ and $\mathcal{L}_2$ indeed make up simple loop classes.

    Since the edge $e_1$ (resp. $e_2$) corresponds to the left-most (resp. right-most) child of the base net interval $[0,1]$ and $\vroot$ is not in any loop class, it follows for $i=1,2$ that $\Omega_{\mathcal{L}_i}$ consists of a single path $\gamma_i=(e_i',e_i,e_i,\ldots)$ with $\pi(\gamma_1)=0$ and $\pi(\gamma_2)=1$.
    Moreover, since $T(e_i)=\begin{pmatrix}p_i\end{pmatrix}$ and $W(e_i)=\lambda$, $\norm{T(\gamma_i|n)}\asymp p_i^{n}$ and thus \cref{e:loc-dim-formula} holds.
    Now since $\theta=(e_{12},e_{21})$ is a cycle and an interior path in $\mathcal{G}_{\ess}$, let $\gamma$ denote any path of the form $\gamma_0\theta\theta\ldots$, so $\gamma\in\Omega^\infty_{\mathcal{G}_{\ess}}$ and by \cref{l:left-prod}
    \begin{equation*}
        \dim_{\loc}(\mu,\pi(\gamma)) = \dim_{\loc}(\rho,\gamma)=\frac{\log p_1p_2}{2\log\lambda}.
    \end{equation*}
    as claimed.
\end{proof}
\begin{figure}[ht]
    \begin{tikzpicture}[
    baseline=(current bounding box.center),
    ]
    \node[vtx,label=above:{$\vroot$}] (vroot) at (0,0) {};
    \node[vtx] (l1) at (-2,-1) {};
    \node[vtx] (l2) at (2,-1) {};
    \node[vtx] (v0) at (0,-2) {};
    \node[vtx,label=below:{$v_1$}] (v1) at (-2,-3) {};
    \node[vtx,label=below:{$v_2$}] (v2) at (2,-3) {};
    \node[vtx,label=below:{$v_3$}] (v3) at (0,-3) {};

    \draw[edge] (vroot) -- node[elbl]{$e_1'$} (l1);
    \draw[edge] (vroot) -- node[elbl]{$e_2'$} (l2);
    \draw[edge] (l1) .. controls +(135:2) and +(225:2) .. node[elbl]{$e_1$} (l1);
    \draw[edge] (l2) .. controls +(45:2) and +(-45:2) .. node[elbl]{$e_2$} (l2);

    \draw[edge] (vroot) -- (v0);
    \draw[edge] (v0) -- (v3);
    \draw[edge] (l1) -- (v0);
    \draw[edge] (l1) -- (v1);
    \draw[edge] (l2) -- (v0);
    \draw[edge] (l2) -- (v2);

    \draw[thick,dotted] ($ (l1) + (-0.7,0) $) ellipse (1cm and 0.7cm);
    \node[fill=white] (l1Label) at ($ (l1) + (-0.7,0) + (180:1cm and 0.7cm) $) {$\mathcal{L}_1$};

    \draw[thick,dotted] ($ (l2) + (0.7,0) $) ellipse (1cm and 0.7cm);
    \node[fill=white] (l2Label) at ($ (l2) + (0.7,0) + (0:1cm and 0.7cm) $) {$\mathcal{L}_2$};
\end{tikzpicture}
    \caption{Partial transition graph for the simple Pisot Bernoulli convolution}
    \label{f:gm-graph}
\end{figure}
\begin{figure}[ht]
    \def\lb{0.544}
\begin{tikzpicture}[
    xscale=14,
    netiv/.style={thick,shorten <= 1pt,shorten >= 1pt,<->}]
    \drawiv{0}{\lb^2}{0}
    \drawiv{\lb-\lb^2}{\lb}{0.5}
    \drawiv{1-\lb}{1-\lb+\lb^2}{0}
    \drawiv{1-\lb^2}{1}{0.5}
    \foreach \x/\lbl in {0/$0$,\lb-\lb^2/$\lambda-\lambda^2$,\lb^2/$\lambda^2$,1-\lb/$1-\lambda$,\lb/$\lambda$,1-\lb^2/$1-\lambda^2$,1-\lb+\lb^2/$1-\lambda+\lambda^2$,1/$1$} {
        \draw[thick,dotted] (\x,0.8) -- (\x,-1.1) node[below,rotate=300,anchor=west]{\lbl};
    }

    \draw[netiv] (\lb^2,-0.7) -- node[fill=white]{$\Delta_1$} (1-\lb,-0.7);
    \draw[netiv] (1-\lb,-0.7) -- node[fill=white]{$\Delta_3$} (\lb,-0.7);
    \draw[netiv] (\lb,-0.7) -- node[fill=white]{$\Delta_2$} (1-\lb^2,-0.7);
\end{tikzpicture}
    \caption{Net intervals in $\mathcal{P}_2$ for the simple Pisot Bernoulli convolution}
    \label{f:netiv-diag}
\end{figure}
\begin{theorem}\label{t:simple-Pisot-mf}
    Let $\mu$ the Bernoulli convolution associated with the Pisot number $r_k$.
    Then $\mu$ satisfies the multifractal formalism if and only if $p_1=p_2=1/2$.
\end{theorem}
\begin{proof}
    It follows from a general observation in \cite[Thm. 3.1]{hh2019} that if $p_1\neq 1/2$, then the set of attainable local dimensions of $\mu$ is not a closed interval (this holds for any overlapping biased Bernoulli convolution, with no separation assumptions).
    Thus $\mu$ does not satisfy the multifractal formalism.

    Conversely, when $p_1=p_2=1/2$, it follows from \cref{p:s-b-set} that the set of local dimensions is a closed interval.
    The IFS is decomposable by \cref{l:size-one-loops}, so by \cref{c:loc-dim-set}, $\mu$ satisfies the multifractal formalism.
\end{proof}
\subsubsection{Other Pisot contractions}
More generally, we can take $r\in(1,2)$ to be any Pisot number.
Let $\mu$ be the Bernoulli convolution with parameter $\lambda=1/r$ associated with probabilities $p_1$ and $p_2$.
We have the following result.
\begin{theorem}
    Suppose $r$ is the Pisot number which is the unique positive real root of any of the polynomials below:
    \begin{itemize}[nl]
        \item $x^3-2x^2+x-1$.
        \item $x^4-x^3-2x^2+1$.
        \item $x^4-2x^3+x-1$.
    \end{itemize}
    Let $\mathcal{G}$ be the transition graph associated with the Bernoulli convolution with parameter $\lambda=1/r$.
    Then $\mathcal{G}$ has one essential loop class $\mathcal{G}_{\ess}$ and two simple loop classes $\mathcal{L}_1$ and $\mathcal{L}_2$, each of which has a single vertex which is a neighbour set of size one.
    Moreover, the set of local dimensions is a closed interval with right endpoint $\log 2/\log r$ when $p_1=p_2=1/2$.

    In particular, $\mu$ satisfies the multifractal formalism if and only if $p_1=p_2=1/2$.
\end{theorem}
\begin{proof}
    This follows by a direct computation, preferably with the aid of a computer: the net intervals in $\mathcal{P}_2$ have the same relative placement and the corresponding transition matrices are the same as given in \cref{p:s-b-set}.
    Thus the conclusion follows by the same argument as \cref{t:simple-Pisot-mf}
\end{proof}
\subsection{A family of non-equicontractive examples}\label{ss:non-e}
Fix parameters $\lambda_1,\lambda_2>0$ and consider the IFS given by
\begin{align}\label{e:wifs-1}
    S_1(x) &= \lambda_1 x & S_2(x) &= \lambda_2 x +\lambda_1(1-\lambda_2) & S_3(x) &= \lambda_2 x+(1-\lambda_2)
\end{align}
where $\lambda_1+2\lambda_2-\lambda_1\lambda_2\leq 1$.
Note that the case $\lambda_1=\lambda_2=1/3$ is discussed in \cref{e:gen-ifs}.
This IFS was first introduced in \cite[Prop. 4.3]{lw2004}, and the multifractal analysis of this measure was studied extensively in \cite{dn2017,rut2021}.

The IFS in \cref{e:wifs-1} is a special case of the following general construction.
Fix parameters $\lambda_1,\lambda_2>0$ and some $k\in\N$, and for $j\in\{0,1,\ldots,k\}$ let $\beta_j=\lambda_1\cdot (\lambda_2/\lambda_1)^{j}$.
Then consider the IFS given by the $k+2$ maps
\begin{equation}\label{e:w-ifs}
    \begin{aligned}
        S_0(x) &= \lambda_1 x\\
        S_i(x) &= \beta_i x+\sum_{j=1}^i\beta_{j-1}(1-\beta_j)\text{ for each }i\in\{1,\ldots,k\}\\
        S_{k+1}(x) &= \lambda_2 x + (1-\lambda_2)
    \end{aligned}
\end{equation}
under the constraint $S_k(1)+\lambda_2\leq 1$.
This IFS coincides with \cref{e:wifs-1} when $k=1$, and coincides with \cite[Ex. 8.5]{dn2017} when $k=2$.

The author proved in \cite[Thm. 5.7]{rut2021} that any self-similar measure associated with the IFS \cref{e:wifs-1} satisfies the multifractal formalism.
However, the proof in that paper is complicated by the use of the iteration rule given in \cref{ex:weighted-transition}.
If we instead take the iteration rule from \cref{ex:uniform-transition} with corresponding transition graph $\mathcal{G}$, the situation is much more straightforward, even with our general setup.
\begin{proposition}\label{p:w-ifs-graph}
    The transition graph $\mathcal{G}$ is strongly connected.
\end{proposition}
\begin{proof}
    The definition of the IFS $(S_i)_{i=0}^{k+1}$ ensures for each $i=1,\ldots,k$ that
    \begin{equation}\label{e:overlap-exact}
        S_{i-1}\circ S_{k+1}=S_0\circ S_i\text{ and }S_{i-1}(0)<S_i(0)<S_{i-1}(1)<S_i(1),
    \end{equation}
    and by assumption $S_{k}(1)\leq S_{k+1}(0)$.
    Thus the net intervals in $\mathcal{P}_1$ are the intervals
    \begin{align*}
        \Delta_0 &= [0,S_1(0)] & \Delta_k&=[S_{k-1}(1),S_k(1)] & \Delta_{k+1} &= S_{k+1}([0,1])\\
    \end{align*}
    and
    \begin{align*}
        \Delta_{i,i+1} &= [S_i(1)\cap S_{i+1}(0)] \text{ for }i=0,1,\ldots,k-1\\
        \Delta_i &= [S_i(1),S_{i+1}(0)]\text{ for }i=1,\ldots,k-1
    \end{align*}
    which are ordered from left to write as $(\Delta_0,\Delta_{0,1},\Delta_1,\ldots,\Delta_{k-1,k},\Delta_k,\Delta_{k+1})$.
    Note that $\vroot=\vs(\Delta_{k+1})$, and set $v_i=\vs(\Delta_i)$ for $i=0,\ldots,k$ and $v_{i,i+1}=\vs(\Delta_{i,i+1})$ for $i=0,\ldots,k-1$.
    Set $V=\{\vroot\}\cup \{v_i:i=0,\ldots,k\}\cup\{v_{i,i+1}:i=0,\ldots,k-1\}$.

    It follows from \cref{e:overlap-exact} that the net intervals in $\mathcal{P}_2$ contained in $S_i([0,1])$ for all $0\leq i\leq k+1$ are just the intervals $S_i(\Delta_j)$ and $S_i(\Delta_{j,j+1})$ with $\vs(S_i(\Delta_j))=v_j$ and $\vs(S_i(\Delta_{j,j+1}))=v_{j,j+1}$ for all valid $j$.
    Tracking inclusion of these net intervals in the net intervals in $\mathcal{P}_1$ yields the graph $\mathcal{G}'$ with (unlabelled) edges given by
    \begin{itemize}[nl]
        \item $(\vroot, v) \text{ for all }v\in V$.
        \item $(v_0,v) \text{ for all }v\in V\setminus\{\vroot\}$.
        \item $(v_k,v) \text{ for all }v\in V\setminus\{v_0,v_{0,1}\}$.
        \item $(v_i,v) \text{ for all }v\in V\setminus\{v_0,v_{0,1},\vroot\}\text{ and }i\in\{1,\ldots,k-1\}$.
        \item $(v_{i-1,i}, v)\text{ for all }v\in\{v_0,v_{0,1}\}\text{ and }i\in\{1,\ldots,k\}$.
    \end{itemize}
    In particular, we observe that $\mathcal{G}'$ is strongly connected.
    Note that, for certain choices of $k,\lambda_1,\lambda_2$, the list $V$ of neighbour sets given above may include repetitions.
    In any case, the transition graph $\mathcal{G}$ is given by identifying vertices in $\mathcal{G}'$ corresponding to the same neighbour set, so $\mathcal{G}$ is strongly connected.
\end{proof}
\begin{theorem}
    Let $\mu$ be any self-similar measure associated with the IFS $(S_i)_{i=0}^{k+1}$ from \cref{e:w-ifs}.
    Then $\mu$ satisfies the multifractal formalism.
\end{theorem}
\begin{proof}
    This is immediate from \cref{p:w-ifs-graph} and \cref{c:one-loop}.
\end{proof}
\bibliographystyle{plain}
\bibliography{texproject/citation-main}
\end{document}